\documentclass[10pt]{amsart}
\usepackage{amssymb, tikz, tikz-cd, xypic, hyperref, stmaryrd, enumerate, stackrel, todonotes, adjustbox}

\usetikzlibrary{cd, fit, calc, backgrounds}

\usepackage[outline]{contour}
\contourlength{5pt}

\newtheorem{innercustomthm}{Theorem}
\newenvironment{customthm}[1]
  {\renewcommand\theinnercustomthm{#1}\innercustomthm}
  {\endinnercustomthm}

\newtheorem{theorem}{Theorem}[section]
\newtheorem{lemma}[theorem]{Lemma}
\newtheorem{corollary}[theorem]{Corollary}
\newtheorem{proposition}[theorem]{Proposition}
\theoremstyle{definition}
\newtheorem{definition}[theorem]{Definition}
\newtheorem{remark}[theorem]{Remark}
\newtheorem{note}[theorem]{Note}
\newtheorem{notation}[theorem]{Notation}
\newtheorem{example}[theorem]{Example}
\numberwithin{equation}{section}
\numberwithin{figure}{section}

\newtheorem*{ack}{Acknowledgements}

\numberwithin{equation}{section}
\allowdisplaybreaks

\newcommand{\C}{\mathbb{C}}
\newcommand{\N}{\mathbb{N}}

\newcommand{\Z}{\mathbb{Z}}
\newcommand{\F}{{\bf F}}
\newcommand{\G}{{\bf G}}

\newcommand{\al}{\alpha}
\newcommand{\be}{\beta}
\newcommand{\ga}{\gamma}
\newcommand{\wt}{\widetilde}
\newcommand{\wh}{\widehat}
\newcommand{\ov}[1]{\stackrel{{}_\frown}{#1}}
\newcommand{\mc}{\mathcal }

\newcommand{\bu}{\bullet}
\newcommand{\na}{\nabla}
\newcommand{\n}{\nabla}

\newcommand{\Del}{{\bf \Delta}}
\newcommand{\DC}{{\bf \Delta}C}

\newcommand{\CMan}{\mathbb{C}Man}

\newcommand{\sAb}{\mc{A}b^{\Del^{op}}}
\newcommand{\sSet}{s\mc{S}et}
\newcommand{\cSet}{c\mc{S}et}
\newcommand{\Chain}{\mc {C}h}
\newcommand{\DK}{DK}
\newcommand{\DKSet}{\underline{DK}}
\newcommand{\Set}{\mc{S}et}
\newcommand{\Cat}{{\mathcal Cat}}
\newcommand{\dgCat}{dgCat}
\newcommand{\Tot}{{ Tot}}

\newcommand{\Hom}{{Hom}}

\newcommand{\CH}{{\bf{Ch}}}
\newcommand{\Ch}{{\bf Ch}}

\newcommand{\EZ}{{\widehat \Delta}}
\newcommand{\both}{{\widetilde \Delta}}
\newcommand{\Cyc}{{\Delta C}}

\newcommand{\ul}{[u]^{\bu\leq 0}}

\newcommand{\OM}{{\bf \Omega}}
\newcommand{\Ups}{\wh\Upsilon}
\newcommand{\Om}{\Omega}
\newcommand{\Ohol}{\Om_{hol}}
\newcommand{\IVB}{{\bf{IVB}}}

\newcommand{\perf}{{\mathcal{P}\hspace{-.3mm}er\hspace{-.7mm}f}}
\newcommand{\PERF}{{\bf Perf}}
\newcommand{\Perf}{\mathcal Perf^\nabla}

\newcommand{\dgN}{dg\mathcal{N}}

\newcommand{\tg}{\tilde g}
  
\newcommand{\CN}{\mathrm{\check{N}}}   
\newcommand{\CechNerve}{\mathrm{\check{N}}}
\newcommand{\OmXHom}{C_L^\bu(X, \Om(U)\hat\otimes \perf(U))}
\newcommand{\sub}{\subset}

\newcommand{\CohSh}{{\bf{CohSh}}}

\newcommand{\CechSh}[1]{{#1}^{\check{\dagger}}} 
\newcommand{\nCechSh}[1]{{#1}_{\le n}^{\check{\dagger}}} 
\newcommand\restr[2]{{
  \left.\kern-\nulldelimiterspace 
  #1 
  \vphantom{\big|} 
  \right|_{#2} 
  }}
\newcommand{\bcosk}[1]{{\bf cosk}_{#1}}
\newcommand{\bsk}[1]{{\bf sk}_{#1}}

\newcommand{\ShgO}[1]{\text{Sh}\mathcal{O}_{#1}^\bu}
\newcommand{\ShO}[1]{\text{Sh}\mathcal{O}_{#1}}
\newcommand{\CohShO}[1]{\text{CohSh}\mathcal{O}_{#1}}

\newcommand{\bmat}[1]{\begin{bmatrix}#1 \end{bmatrix}}

\newcommand{\subsc}[2]{#1; #2}

 \makeatletter
\renewcommand*\env@matrix[1][*\c@MaxMatrixCols c]{%
  \hskip -\arraycolsep
  \let\@ifnextchar\new@ifnextchar
  \array{#1}}
\makeatother

\DeclareMathOperator*{\colim}{colim}
\DeclareMathOperator*{\holim}{holim}
\DeclareMathOperator*{\hocolim}{hocolim}

\author[C. Glass]{Cheyne Glass}
\address{Cheyne Glass, State University of New York at New Paltz, Department of Mathematics, 1 Hawk Dr., New Paltz, NY 12561}
  \email{glassc@newpaltz.edu}

\author[M. Miller]{Micah Miller}
\address{Micah Miller, Borough of Manhattan Community College The City University of New York, Department of Mathematics, 199 Chambers Street, New York, NY 10007}
  \email{mmiller@bmcc.cuny.edu}

\author[T. Tradler]{Thomas Tradler}
  \address{Thomas Tradler, New York City College of Technology The City University of New York, Department of Mathematics, 300 Jay Street, Brooklyn, NY 11201}
  \email{ttradler@citytech.cuny.edu}

\author[M. Zeinalian]{Mahmoud Zeinalian}
  \address{Mahmoud Zeinalian, Lehman College, The City University of New York, Department of Mathematics, 250 Bedford Park Blvd W, Bronx, NY 10468}
  \email{mahmoud.zeinalian@lehman.cuny.edu}

\title{Chern Character for Infinity Vector Bundles}


\begin{document}

\begin{abstract}
Coherent sheaves on general complex manifolds do not necessarily have resolutions by finite complexes of vector bundles. However D. Toledo and Y.L.L. Tong showed that one can resolve coherent sheaves by objects analogous to chain complexes of holomorphic vector bundles, whose cocycle relations are governed by a coherent infinite system of homotopies. In the modern language such objects are obtained by the $\infty$-sheafification of the simplicial presheaf of chain complexes of holomorphic vector bundles. We define a Chern character as a map of simplicial presheaves, whereby the connected components of its sheafification recovers the Chern character of Toledo and Tong. As a consequence our construction extends Toledo Tong and O'Brian Toledo Tong's definition of the Chern character to the settings of stacks and in particular the equivariant setting. Even in the classical setting of complex manifolds, the induced maps on higher homotopy groups provide new Chern-Simons, and higher Chern-Simons, invariants for coherent sheaves.
\end{abstract}
\maketitle

\setcounter{tocdepth}{1} 
\tableofcontents

\section{Introduction}
The celebrated Hirzebruch-Riemann-Roch theorem (HRR) \cite{Hirz} is a generalization of the classical Riemann-Roch theorem for holomorphic line bundles on compact Riemann surfaces. In HRR, the setting of a line bundle on a Riemann surface is generalized to an arbitrary holomorphic bundle $E$ on a smooth projective variety $X$ over the complex numbers. The main tool in proving HRR is a resolution of the diagonal $X\to X\times X$, thought of as a coherent sheaf on $X\times X$, by a finite chain complex of vector bundles. 

The Atiyah-Singer index theorem \cite{AS}, and the theory of elliptic and pseudo-elliptic differential operators, can further be thought of as a far-reaching generalization of HRR and other high-powered theorems, such as the Gauss-Bonnet theorem, to a much vaster context. For example, using the Atiyah-Singer index theorem one can readily extend HRR to holomorphic bundles on even compact complex manifolds that are not necessarily algebraic (see for example \cite{Fre} for an exposition). 

Unfortunately, such techniques \cite{ABP, Gi} use differential geometric methods that heavily rely on an axillary choice of a Hermitian metric on the manifold as well as the bundle. For example, one uses the metric to establish a heat flow and smooth out the diagonal de Rham current $X\to X\times X$ into a differential form (the heat kernel). However, generally, in complex geometry choosing a metric can be thought of as unnatural and out of context unless within the very specialized realm of K\"ahler geometry. 

Casting this as deficiency is not only a matter of taste but concerns applications of these ideas to settings where local automorphisms are involved, such as the equivariant as well as the `stacky' discussion. One would therefore desire an intrinsic complex geometric discussion whereby one establishes HRR, and similar theorems, for general complex manifolds and holomorphic vector bundles outside metric geometry. 

In a series of papers \cite{TT76, TT78a, TT78b, OTT1, OTT2, OTT3, TT86} published in and around the 1980's, Toledo and Tong made several remarkable conceptual breakthroughs by providing local Cech cohomological proofs of HRR \cite{OTT2} and Grothendieck Riemann Roch (GRR) \cite{OTT3}. Through the modern lens, one may interpret their work as a hands-on theory of infinity stacks, which only much more recently has been made into a full-fledged mathematical theory. One of the key constructions in \cite{OTT1} is that of the Chern class for a coherent analytic sheaf on a complex manifold. While their construction is the one we focus on in the present paper, there is also another approach to calculating Chern classes for coherent analytic sheaves, as shown in \cite{Gr, TT86} and later formalized by Timothy Hosgood in \cite{Ho1, Ho2, Ho3}.

To get a taste for the type of math Toledo and Tong invented and utilized, consider the question of resolving the diagonal $X\to X\times X$, or more generally an arbitrary coherent sheaf, on a complex manifold, by a finite chain complex of vector bundles. One knows that when the complex manifold admits a positive line bundle such resolutions always exist (see \cite[page 705]{GH}). While in the algebraic setting, the canonical line bundle provides such a line bundle, general complex manifolds may not support them. Toledo and Tong obviated such difficulties by resolving the problems in a homotopical setting in which strict identities are replaced with a coherent infinite system of homotopies. For instance, as a complex vector bundle is a bunch of transition functions satisfying the familiar cocycle conditions, they showed that by requiring the cocycle condition to hold up to an infinite system of homotopies, not only every coherent sheaf on a complex manifold could be resolved by these more general objects, but also all of the necessary complex geometric arguments would remain valid. 

Let us be more specific and start with a coherent sheaf on a complex manifold. Choose a good Stein cover for the manifold on which the coherent sheaf can be locally resolved by a chain complex of vector bundles; such a cover always exists. By restricting these resolutions to double intersections, we obtain two resolutions for the same coherent sheaf on that intersection which, by the uniqueness of resolutions over Stein manifolds, are then related by a quasi-isomorphism. On triple intersections, the three relevant quasi-isomorphisms may not fit to give you a chain complex of vector bundles, but the discrepancy can be killed by a homotopy. These assigned homotopies to triple intersections may not satisfy the required compatibilities on quadruple intersections but the discrepancy can be killed by a higher homotopy. Repeating this pattern ad infinitum gives rise to an infinite system of homotopies. 

Historically, the use of coherent infinite systems of homotopy in a different context was known to some algebraic topologists almost 30 years prior but even there it was considered rather esoteric. In the 1960's Jim Stasheff showed how the based loop space of a pointed space was an $A_\infty$ monoid (\cite{St}). Nowadays these mathematical objects are inescapable and it is common knowledge among a large group of algebraic topologists that $A_\infty$ algebras are just as good as differential associative algebras and have the same homotopy theories (\cite{Le}). Similarly, Toledo and Tong, showed that these generalized objects are just as good as chain complexes of vector bundles as far as coherent cohomologies were concerned. While they did not make a formal claim about their corresponding homotopy theories, they showed how Ext and Tor of such generalized objects can be defined, calculated, and subsequently, be used to prove duality theorems a la Grothendieck and establish HRR and GRR.

Surprisingly, since their work very little has been done to formalize the homotopy theory of these objects. For example, in `Descente Pour Les N-Champs', Carlos Simpson writes:

\begin{quote}Dans les travaux de O'Brian, Toledo et Tong [77], [78], [104], [105], [106] consacres `a une autre question issue de SGA 6, celle des formules de Riemann-Roch, on trouve des calculs de Cech qui sont certainement un exemple de situation de descente pour les complexes. Un meilleur cadre g\'en\'eral pour ces calculs pourrait contribuer \'a notre comprehension des formules de Riemann-Roch.
\end{quote}

This roughly translates to the following: 

\begin{quote}
In the work of O'Brian, Toledo, and Tong [77], [78], [104], [105], [106] devoted to another question arising from SGA 6 regarding the Riemann-Roch formulas, one can find Cech calculations that are examples of a descent for complexes. A better general framework for these calculations could contribute to our understanding of the Riemann-Roch formulas.
\end{quote}

In the current paper, we have taken the first step in providing a homotopy theoretic framework for some of Toledo-Tong's mathematical objects. By simply finding the right homotopy theoretic setting, their constructions, extend far beyond what they had intended and point to new and exciting advances. For example, their construction of a Chern character for coherent sheaves in Hodge cohomology is easily generalized to the equivariant setting, or even to the setting of stacks. In addition, secondary and higher Chern characters are now an inseparable part of the discussion. 

The inherent inclusion of these higher Chern characters points to the possibility of proving a version of GRR as a commutative diagram of spaces such that after applying $\pi_0$ to the diagram one would obtain a diagram of sets which is the O'Brian-Toledo-Tong's GRR. Note that classical objects such as K-groups and cohomology groups are sets with additional algebraic structures.

In section \ref{SEC: IVB OM} we begin by defining the simplicial presheaves  $\IVB$ and $\OM$ which will be the domain and codomain of our Chern map, respectively. For a fixed complex manifold $U \in \CMan$, we first consider the dg-category $\Perf(U)$ of finite chain complexes of holomorphic bundles with connection\footnote{The use of $\Perf$ is meant to allude to the study of perfect complexes.}, where there is no requirement that morphisms are compatible with connections. Taking the maximal Kan complex of the dg-nerve we obtain a simplicial set $\PERF (U)$. Applying this construction objectwise over $\CMan$ and noting that maps $f\in \CMan^{op}(U,V)$ induce maps of Kan complexes $\PERF(U) \xrightarrow{f^*} \PERF(V)$ via pull-backs, we obtain a simplicial presheaf $\PERF$ which is fibrant in the (global) projective model structure. Since the simplices $\PERF(U)_n= \sSet(\Delta^n,\PERF(U))$ lack the cyclic structure we will need later on to construct our trace map, we define a weakly-equivalent (see proposition \ref{PROP:PERF-EZ=PERF-Delta}) simplicial presheaf $\IVB(U)_n := \sSet(\EZ^n,\PERF(U))$ given by mapping the cyclic sets $\EZ^n$ into $\PERF(U)$. Here, $\EZ^n$ is the nerve of the category whose set of objects is $\mathbb{Z} / (n+1) \mathbb{Z}$ and all hom-sets have a single morphism (see example \ref{EXA:EZ}). Next, we define $\OM$ in the same way we did in \cite{GMTZ}; more precisely $\OM(U)$ is the simplicial set whose $k$-simplices are decorations of all $i$-dimensional faces of the standard $k$-simplex with sequences of forms, all even for $i$ even, and all odd for $i$ odd, in such a way that the alternating sum of all forms sitting on the $(i-1)$-dimensional faces of any $i$-dimensional face add up to $0$. 

The Chern map $\Ch : \IVB \to \OM$ is then defined in section \ref{SEC: Chern Map} as follows. An $n$-simplex in $\IVB(U)_n$ consists of $(n+1)$-many dg-bundles with connection, $(\mc E_{i}, d_i, \nabla_i)$, and a collection of maps $g=\{\left( g_{(i_0\dots i_k)}: \mc E_{i_k} \to \mc E_{i_0}\right)\}_{(i_0,\dots, i_k)\in \EZ^n}$ satisfying the Maurer-Cartan condition (see definition \ref{DEF: MC}). First, in definition \ref{DEF:trace_g}, we define a trace map, $Tr_g$ similar to that of \cite[Proposition 3.2]{OTT1}, satisfying the following condition: 
\begin{equation}\tag{\ref{EQU:Trg-chain-map}}
Tr_g \circ (\hat \delta + D + [g,-]) = \delta \circ Tr_g,
\end{equation}
Using this trace map, $\Ch$ is then defined (in definition \ref{DEF:CH:PERF-to-OM}) by assigning to an $n$-simplex in $\IVB(U)_n$ as above decorations of the non-degenerate $k$-faces of $\Delta^n$ given by the following elements in $\OM(U)_n$: 
\begin{align}
Tr_g(A^k)_{\alpha}\cdot \frac{u^k}{k!} &=Tr_g((\nabla (d+g))^k)_{\alpha}\cdot \frac{u^k}{k!} \tag{\ref{EQ: Chern Tr A^k}}\\
&\nonumber = \sum \pm tr (g \cdot \nabla (d+g)\cdot   \nabla (d+g) \cdot \ldots \cdot \nabla (d+g))_\alpha \cdot \frac{u^k}{k!}
\end{align}
for $k>0$ and for $k=0$ we assign the Euler characteristic. Our first main result is that this provides a map of (objectwise Kan) simplicial presheaves:

\begin{customthm}{\ref{THM:CH-is-map-of-simplicial-presheaves}}
The Chern character $\CH:\IVB\to \OM$ defined above is a map of simplicial presheaves.
\end{customthm}

In section \ref{SEC-Sheafify} we construct what we call the \v{C}ech sheafification, $\CechSh{\CH}:\CechSh{\IVB}\to \CechSh{\OM}$ of the Chern map. Given a simplicial presheaf $\F$, the idea is that for each open cover $\left( U_i \to X\right)_{i \in I}$ we can form the \v{C}ech nerve simplicial presheaf, $\CechNerve{U}_{\bullet}$, and then compute the homotopy limit induced by the simplicial mapping space $\underline{sPre}\left( \CechNerve{U}_{\bullet}, \F\right)=  \holim\limits_i  {\prod\limits_{\alpha_0 \dots \alpha_i}}\F(U_{\alpha_0 \dots \alpha_i}) $ by taking the totalization of the induced cosimplicial simplicial set $\F \left( \CechNerve{U}_{\bullet}\right)$ defined in \eqref{EQ: F of CN}. The \v{C}ech sheafification $\CechSh{F}$ is then defined (Definition \ref{DEF: cech sheafify}) by taking the colimit over all covers:
\begin{equation}\tag{\ref{EQ: cech sheafify}}
 \CechSh{{\F}}(X) := \colim_{( {U}_{\bu} \to X) \in \check{S}} Tot\left( {\F}(\CechNerve {U}_{\bu} ) \right).
 \end{equation}
 As the construction is functorial in simplicial presheaves and preserves Kan complexes we obtain a sheafified Chern map, $\CechSh{\CH}:\CechSh{\IVB}\to \CechSh{\OM}$ which is a map of Kan complexes. The rest of the section is devoted to showing how $\CechSh{\CH}$ is related to the Chern character map of \cite{OTT1} which begins with theorem \ref{THM: OTT are vertices} stating that the twisting cochains of \cite{OTT1} include into the vertices of $\CechSh{\IVB}$. The full correspondence is given in theorem \ref{THM:CH(IVB)=CH(TwCoch)} which shows that if we restrict $\IVB$ to the simplicial presheaf $\CohSh$ considering only non-positively graded chain complexes whose homology is concentrated in degree zero, then we fully recover the data from the the Chern map in \cite{OTT1} by the connected components of our sheafified Chern map: 
\begin{customthm}{\ref{THM:CH(IVB)=CH(TwCoch)}}
For a given coherent sheaf, the formula for the Chern character \eqref{EQU:Chern-char-a-la-OTT} from \cite{OTT1} is given by the terms in the formula \eqref{EQ: Tr g nabla g} of the Chern character map 
\begin{equation}\tag{\ref{EQ: pi_0 Chern}}
\left\{\substack{ \text{Iso. classes of}\\  \text{coherent sheaves}}\right\}   \simeq \pi_0\left( \CechSh{\CohSh}\right) \xrightarrow{\pi_0\left(\CechSh{\CH}\right)}  \pi_0\left( \CechSh{\OM}\right)\simeq \bigoplus_{\substack{p,q\\p+q \text{ even}}} H^p\left( \Omega^q\right)
\end{equation}
 applied to the corresponding twisting cochain interpreted (by theorem \ref{THM: OTT are vertices}) as a $0$-simplex in $\CechSh{\CohSh}$.
 \end{customthm}
 
 Section \ref{SEC: Local Proj} upgrades the results from the previous section to statements about (hyper-)sheaves. Recall that a simplicial presheaf is a (hyper-)sheaf if it is object-wise Kan and satisfies descent with respect to all hypercovers. By restricting our attention to simplicial presheaves of finite homotopy type taking values in Kan complexes, we prove in proposition \ref{Prop: Cech Fibrant} that the aforementioned \v{C}ech sheafification construction computes the (hyper-)sheafification. In particular, proposition \ref{Prop: restricted Cech Ch sheaf map} states that if we restrict to complex manifolds of bounded dimension, and restrict the homotopy type of $\IVB$ then  $\CechSh{\Ch}: \nCechSh{\IVB} \to \CechSh{\OM}$ is a map of hypersheaves. If instead we consider again $\CohSh$, we see that its sheafification is a classifying stack for coherent sheaves, $ \mathbb{R} Hom (X, \CohSh) \simeq \CechSh{\CohSh}$:
 \begin{customthm}{\ref{THM:IVBStack}}
The simplicial presheaf $\CohSh$ is a classifying pre-stack for coherent sheaves.
 \end{customthm}
Once again restricting to manifolds of bounded dimension, theorem \ref{THM: OTT Chern Sheaves} states that our sheafified Chern map $ \CechSh{\Ch}: \CechSh{\CohSh} \to \CechSh{\OM} $ is a map of (hyper-)\linebreak sheaves whose connected components yields the Chern map from \cite{OTT1}. Finally, remark \ref{REM: Generalized Chern Ch} describes how our Chern character map generalizes to all stacks, with an eye towards future work in the equivariant setting.

\begin{ack}
T. Tradler was partially supported by a PSC-CUNY grant. C. Glass would like to thank the Max Planck Institute for Mathematics in Bonn, Germany, for their hospitality during his stay. M. Zeinalian would like to thank Julien Grivaux and Tim Hosgood for insightful conversations about Toledo and Tong's work. 
\end{ack}

\begin{notation}\label{NOT:Delta-etc}
The simplicial category is denoted by $\Del$. It has objects $[n]=\{0,\dots, n\}$ for $n\in \N_0$, and morphisms $\phi\in \Del([k],[n])$ non-decreasing maps $\phi:[k]\to [n]$, i.e., $\phi(i)\leq \phi(j)$ for $i\leq j$. The morphisms are generated by face maps $\delta_j:[n]\to [n+1]$ (injection that skips the element $j$ in $[n+1]$) and degeneracies $\sigma_j:[n]\to [n-1]$ (surjection that maps $j$ and $j+1$ to $j$).

Simplicial objects in a category $\mc C$ are functors $\Del^{op}\to \mc C$, where the induced face and degeneracy morphisms are denoted by $d_j$ and $s_j$, respectively. We denote the category of simplicial sets by $\sSet=\Set^{\Del^{op}}$. Cosimplicial objects in a category $\mc C$ are functors $\Del\to \mc C$.

Given an object $X$ in a locally small category, $\mc C$, we can consider its representable presheaf $yX:= \mc C(- , X)$ given by the Yoneda embedding. Further, given a presheaf $F$ on $\mc C$, we can consider its simplicially constant presheaf $c F$ defined by $c F (Y)_n := F(Y)$. When context is clear we may drop the ``$y$'' or ``$c$''. For example, for an object $X$ we might write $X$ for the simplicial presheaf defined by $X(Y)_n:= \mc C \left(Y, X\right)$. 
\end{notation}

\section{The Simplicial Presheaves $\IVB$ and $\OM$}\label{SEC: IVB OM}

We define two simplicial presheaves on the site of complex manifolds; first $\IVB:\CMan^{op}\to \sSet$ is the presheaf which will later give rise to infinity vector bundles (see definition \ref{DEF:inf-vec-bun}), and $\OM:\CMan^{op}\to \sSet$ is the presheaf of holomorphic forms. In the next section we will then define the Chern character map as a map of simplicial presheaves $\CH:\IVB\to \OM$.

Let $\CMan$ be the category whose object consist of complex manifolds, and morphisms are holomorphic maps. Furthermore, denote by $\dgCat$ the category of differential graded categories, i.e., categories $\mc C$ such that for any for any two objects $C_1, C_2$ of $\mc C$, the space of morphisms $\Hom(C_1,C_2)$ is a cochain complex, with the composition being a cochain map, and the identity morphisms being closed. 

\begin{definition}\label{DEF:Perf} 
Denote by $\perf:\CMan^{op}\to \dgCat$ given by setting $\perf(U)$ to be the dg-category, whose objects $\mc E=(E_\bu,d,\nabla)\in \perf(U)$ are finite chain complexes of holomorphic vector bundles $E_\bu\to U$  over $U$ with differential $d:E_\bu\to E_{\bu+1}$, and with a holomorphic connection $\nabla$ on $E_\bu$. Morphisms $\Hom(\mc E,\mc E')$ consist of graded morphisms of vector bundles $f:E_\bu\to E'_\bu$ which \emph{need not have any special compatibility} with respect to the connections $\na$ and $\na'$. The dg structure on $\Hom(\mc E,\mc E')$ is the induced one by the differential and gradings on $\mc E$ and $\mc E'$; in particular, the differential of an $f\in\Hom(\mc E,\mc E')$ is defined to be $D(f):=f\circ d-(-1)^{|f|}d'\circ f$.

A holomorphic map $\varphi: U \to U'$ induces a functor $\perf(\varphi): \perf(U') \to\perf(U)$ by pulling back bundles via $\varphi$. 
\end{definition}

Since $\perf(U)$ is a dg-category, we can apply the dg-nerve $\dgN(\perf(U))$, which gives a simplicial set. 
\begin{note}\label{REM:simplicial-dgN(Perf(U))}
Explicitly, we can describe the simplicial structure of the dg-nerve $\dgN(\mc C)$ of a dg-category $\mc C$ (for us, it will always be $\mc C=\perf(U)$) as follows; cf. \cite[1.3.1.6]{Lurie} and \cite[Definition 2.2.8]{Fa}:
\begin{enumerate}
\item
A $0$-simplex in $\dgN(\mc C)_0$ consists of an object $\mc E$ of $\mc C$.
\item
A $1$-simplex in $\dgN(\mc C)_1$ consists of $(\mc E_1,\mc E_0,g_{01})$, i.e., two objects $\mc E_0$, $\mc E_1$ of $\mc C$ and a morphism $g_{01}:\mc E_1\to \mc E_0$ in $\mc C$ of degree $0$, which is closed, i.e. $Dg_{01}=0$, where we denoted the differential in $Hom_{\mc C}$ by $D$. (In the case of $\mc C=\perf(U)$, the internal differential $D$ is given by the differentials $d$ and $d'$ on $E$ and $E'$, respectively, via $Df=f\circ d-(-1)^{|f|}d'\circ f$, so that $Dg_{01}=0$ means that $g_{01}$ are chain maps of dg-vector bundles.)
\item
A $2$-simplex in $\dgN(\mc C)_2$ consists of $(\mc E_0,\mc E_1,\mc E_2, g_{01}, g_{12}, g_{02},g_{012})$, i.e., three objects $\mc E_0$, $\mc E_1$, $\mc E_2$ of $\mc C$ three morphisms $g_{ij}: \mc E_j\to \mc E_j$ of degree $0$, where $i,j\in \{0,1,2\}$ with $i<j$, and another morphism $g_{012}:\mc E_2\to \mc E_0$ of degree $-1$ satisfying $Dg_{012}=g_{01}\circ g_{12}-g_{02}$.
\item\label{ITEM:n-simplex-dgN(Perf(U))}
An $n$-simplex in $\dgN(\mc C)_n$ consists of $n+1$ dg-vector bundles $\mc E_0,\dots, \mc E_n$ and morphisms $g_{i_0\dots i_k}:\mc E_{i_k}\to \mc E_{i_0}$ of degree $1-k$ for each sequence $i_0,\dots, i_k\in \{0,\dots, n\}$ with $i_0< \dots< i_k$ and $k\geq 1$, such that  
\begin{equation}\label{EQU:dg-i0-...ik-relation}
D(g_{i_0\dots i_k})=\sum_{j=1}^{k-1} (-1)^{j-1} g_{i_0\dots \widehat{i_j}\dots i_k}+\sum_{j=1}^{k-1} (-1)^{k(j-1)+1} g_{i_0\dots i_{j}}\circ g_{i_{j}\dots i_k}.
\end{equation}
\item\label{ITEM:morphism-dgN(Perf(U))}
For a morphism $\phi:[n]\to [m]$ in $\Del$, there is an induced map $\phi^\sharp_{\dgN}:\dgN(\mc C)_m\to \dgN(\mc C)_n$, given by mapping $(\mc E_i, g_{i_0,\dots,i_k})_{\text{all indices}}\in\dgN(\mc C)_m$ to $(\mc E_{\phi(i)}, \wt g_{i_0,\dots,i_k})_{\text{all indices}}\in \dgN(\mc C)_n$, which is defined by either $\wt g_{i_0 \dots i_k}=g_{\phi(i_0) \dots \phi(i_k)}$ if $\phi$ is injective on $\{i_0\dots i_k\}$, or $\wt g_{i_0i_1}=id_{ E_{\phi(i_0)}}$ if $\phi(i_0)=\phi(i_1)$, or $\wt g_{i_0\dots i_k}=0$ in all other cases, i.e., when $k\geq 2$ and $\phi(i_p)=\phi(i_{p+1})$ for some $p=0,\dots, k-1$.
\end{enumerate}
\end{note}

In the later chapters, we will use the dg-nerve of $U$ as \emph{local} building blocks of chain complexes of vector bundles on a complex manifold. To obtain a reasonable gluing we will want the chain maps $g_{i_0 i_1}$ to be \emph{homotopy equivalences}. This can be achieved in a natural way by taking the maximal Kan subcomplex $\dgN(\perf(U))^\circ$ of $\dgN(\perf(U))$; cf. \cite[Cor 1.5]{J}. 
\begin{definition}
Denote by $\PERF:\CMan^{op}\to \sSet$ the simplicial presheaf given by $\PERF(U):=\dgN(\perf(U))^\circ$, i.e., the maximal Kan subcomplex of the dg-nerve of $\perf(U)$.
\end{definition}
We have the following characterization of the simplicies of $\dgN(\perf(U))^\circ$ via By \cite[Theorem 2.2]{J}, for example. 
\begin{lemma}\label{LEM:simplicial-dgN(Perf(U))^circ}
An $n$-simplex in $\dgN(\perf(U))^\circ$ consists precisely of an $n$-simplex in $\dgN(\perf(U))$ as described in note \ref{REM:simplicial-dgN(Perf(U))} \eqref{ITEM:n-simplex-dgN(Perf(U))} above, with the extra condition that all morphisms $g_{i_0 i_1}:\mc E_{i_1}\to \mc E_{i_0}$ are \emph{homotopy equivalences}.
\end{lemma}

Now, all chain maps $g_{i_0 i_1}$ on the edges of all simplicies of $\dgN(\perf(U))^\circ$ are homotopy equivalences. In order to be able to define the Chern character below, we will need to find homotopy inverses of these together with compatible higher homotopies. This can be achieved as follows. First, using the Yoneda lemma for simplicial sets, we know that the $n$-simplicies of a simplicial set $X_\bu$ are precisely the simplicial set maps from $\Delta^n:=\Del(-,[n])$ into $X_\bu$, i.e., $X_n=X([n])\cong \text{Nat}(\Del(-,[n]),X)=\sSet(\Delta^n,X)$. Thus, we define $\PERF^{\Delta}:\CMan^{op}\to \sSet$ by setting
\begin{equation}\label{EQU:Perf^Delta}
\PERF^\Delta(U)_n:=\dgN(\perf(U))^\circ_n=\sSet(\Delta^n,\dgN(\perf(U))^\circ).
\end{equation}
More generally, we define:
\begin{definition}\label{DEF:PERF^Q}
Let $Q$ be a cosimplicial simplicial set, i.e., $Q:\Del\to \sSet$. In more detail, we denote by $Q^n=Q([n])\in \sSet$ the image of $[n]\in \Del$ under $Q$, which is itself a simplicial set, $Q^n_\bu:\Del^{op}\to \Set$, $Q^n_k:=Q^n([k])\in \Set$. Then, define $\PERF^{Q}:\CMan^{op}\to \sSet$ by setting
\begin{equation}\label{EQU:DEF:PERF^Q}
\PERF^Q(U)_n:=\sSet(Q^n,\dgN(\perf(U))^\circ).
\end{equation}
Since the $\{Q^n\}_n$ are a cosimplicial object in $\sSet$, this induces for each $(f:[n]\to [m])\in \Del$ a map $\PERF^Q(U)_m\to \PERF^Q(U)_n$, making $\PERF^Q(U)$ into a simplicial set.

For a holomorphic map $\varphi:U\to U'$, the induced map $\PERF^Q(U')\to \PERF^Q(U)$ is given by composition with the map $\perf(U')\to \perf(U)$ from definition \ref{DEF:Perf}, i.e., by pulling back via $\varphi$.
\end{definition}
We are mainly interested in the following examples \ref{EXA:Delta} and \ref{EXA:EZ}.
\begin{example}\label{EXA:Delta}
Let $\Delta:\Del\to \sSet$ be given by $\Delta^n:=\Del(-,[n])$ be the standard simplicial $n$-simplex given by morphism of $\Del$ into $[n]$, i.e., its $k$-simplicies $\phi\in \Delta^n_k= \Del([k],[n])$ are non-decresing maps from $[k]$ to $[n]$, i.e., if we set $i_j:=\phi(j)$ these are sequences of indices $(i_0\leq \dots\leq i_k)$ with $i_0,\dots, i_k\in \{0,\dots, n\}$. Face maps are $d_j:\Delta^n_k\to \Delta^n_{k-1}$ remove the $j$th index $i_j$, and degeneracies $s_j:\Delta^n_k\to \Delta^n_{k+1}$ repeat the $j$th index $i_j$. 

By Yoneda, any simplicial set map $\Delta^n\to X$ is completely determined by the image of its non-degenerate $n$-simplex. Thus, by \eqref{EQU:Perf^Delta}, $\PERF^\Delta(U)$ has $n$-simplicies given as described precisely by lemma \ref{LEM:simplicial-dgN(Perf(U))^circ}, i.e., by note \ref{REM:simplicial-dgN(Perf(U))} with homotopy equivalences on edges.
\end{example}
Before we give our second main example for $\PERF^Q$, we record a useful lemma about simplicial set maps into the dg-nerve $\PERF^Q(U)$.
\begin{lemma}\label{LEM:sSet(X,dgN(Perf))}
Let $X_\bu$ be a simplicial set, and let $\mc C$ be a dg-category (for us, $\mc C=\perf(U)$). Then a simplicial set map $X\to \dgN(\mc C)$ is precisely given by the following data:
\begin{enumerate}
\item 
for each $0$-simplex $\alpha\in X_0$, we have an object $\mc E_\alpha$ of $\mc C$
\item \label{ITEM:g-alpha-maps}
for each non-degenerate $k$-simplex $\alpha\in X_k$ with $k\geq 1$, there is a morphism $g_\alpha: \mc E_{\alpha (k)}\to \mc E_{\alpha (0)}$ of degree $1-k$ satisfying the compatibility condition:
\begin{equation}\label{EQU:dg-alpha-relation}
D(g_{\alpha})=\sum_{j=1}^{k-1} (-1)^{j-1} g_{\alpha(0,\dots, \widehat{j},\dots k)}+\sum_{j=1}^{k-1} (-1)^{k(j-1)+1} g_{\alpha(0,\dots, j)}\circ g_{\alpha(j,\dots, k)}
\end{equation}
\end{enumerate}
Here, for a disjoint union decomposition $\{0,\dots, k\}=\{i_0,\dots, i_p\}\sqcup \{j_0,\dots, j_q\}$ with $i_0<i_1<\dots<i_p$ and $j_0<j_1<\dots<j_q$, we denote by $\alpha(i_0,\dots, i_p):=d_{j_0}\circ \dots \circ d_{j_q}(\alpha)\in X_{p}$ the face of $\alpha$ corresponding to indices $\{i_0,\dots, i_p\}\subseteq \{0,\dots, k\}$.

In particular, a simplicial set map $X\to \dgN(\perf(U))^\circ$ has the same data as given above with the extra condition that the maps $g_{\alpha}$ for $\alpha\in X_1$ are homotopy equivalences.
\end{lemma}
\begin{proof}

Let $\mc F:X\to\dgN(\mc C)$ be a map of simplicial sets, and for $\ell\geq 0$, let $\alpha\in X_\ell$ be an $\ell$-simplex. Thus, $\mc F(\alpha) \in \dgN(\mc C)_\ell$, and by note \ref{REM:simplicial-dgN(Perf(U))}, there are dg-vector spaces $\mc E_0^{\alpha},\dots, \mc E_\ell^{\alpha}$, and for all $i_0,\dots ,i_k\in \{0,\dots, \ell\}$, $k\geq 1$, with $i_0<\dots <i_k$ there are maps $g_{i_0\dots i_k}^{\alpha}:\mc E^{ \alpha}_{i_k}\to \mc E^{\alpha}_{i_0}$ satisfying \eqref{EQU:dg-i0-...ik-relation}. We claim that the data of the highest maps $g_{0\dots p}^{\rho}$, for all non-degenerate $\rho\in X_p$, is sufficient to recover all other maps $g_{i_0\dots i_k}^{\alpha}$: For $\alpha \in X_\ell$ and $i_0,\dots ,i_k\in \{0,\dots, \ell\}$ with $i_0<\dots <i_k$ with $k<\ell$, use the commutative diagram for $\phi:[k]\to [\ell], \phi(p):=i_p$,
\[
\begin{tikzcd}
X_\ell \arrow{r}{\mc F_\ell} \arrow[swap]{d}{\phi^\sharp_{X}} & \dgN(\mc C)_\ell \arrow{d}{\phi^\sharp_{\dgN}} \\
X_k \arrow{r}{\mc F_k}& \dgN(\mc C)_k
\end{tikzcd}
\]
mapping the $i_0<\dots<i_k$-component $g_{i_0\dots i_k}^{\alpha}$ of $\mc F_\ell(\alpha)$ under $\phi^\sharp_{\dgN}$ to $\wt g_{0\dots k}=g_{i_0\dots i_k}^{\alpha}$ (by note \ref{REM:simplicial-dgN(Perf(U))} \eqref{ITEM:morphism-dgN(Perf(U))}, since $\phi$ is injective). Now, $\phi=\delta_{j_q}\circ\dots \circ \delta_{j_0}$ for $\{i_0,\dots, i_k\}\sqcup \{j_0,\dots, j_q\}=\{0,\dots, k\}$ with  $j_0<j_1<\dots<j_q$, so that the left vertical map $\phi^\sharp_{X}$ maps $\phi^\sharp_{X}(\alpha)=d_{j_0}\circ\dots \circ d_{j_q}(\alpha)=\alpha(i_0 ,\dots, i_k)$. Then, $\mc F_k$ maps this to the $0<\dots<k$-component $g_{0\dots k}^{\alpha(i_0,\dots, i_k)}$. By the commutativity of the diagram, we get that $g_{i_0\dots i_k}^{\alpha}=g_{0\dots k}^{\alpha(i_0,\dots, i_k)}$. This shows that the maps $g_{\alpha}:=g^{\alpha}_{0\dots \ell}$ for all $\alpha\in X_\ell$ for $\ell\geq 1$, together with the implicit dg-vector spaces $\mc E_\alpha=\mc F_0(\alpha)$ for all $0$-simplicies $\alpha\in X_0$, gives the complete data of the map of simplicial sets $\mc F:X\to \dgN(\mc C)$. Equation \eqref{EQU:dg-i0-...ik-relation} for $g^\alpha_{0\dots \ell}$ using a fixed $\alpha\in X_\ell$ becomes precisely \eqref{EQU:dg-alpha-relation} via the identifications $g_{\alpha}=g^{\alpha}_{0\dots \ell}$ and $g_{i_0\dots i_k}^{\alpha}=g_{0\dots k}^{\alpha(i_0,\dots, i_k)}$.

Note moreover, by a similar argument, that degenerate simplicies map to either the identity $g_{s_j(\alpha)}=id_{E_\alpha}$ for $\alpha\in X_0$, or $g_{s_j(\alpha)}=0$ for $\alpha\in X_\ell$ with $\ell\geq 1$.

Finally, $\mc F:X\to \dgN(\perf(U))^\circ$ lands in $\dgN(\perf(U))^\circ$ precisely if all maps $g_\alpha$ given by $\mc F(\alpha)$ for $\alpha\in X_1$ are homotopy equivalences by lemma \ref{LEM:simplicial-dgN(Perf(U))^circ}.
\end{proof}

\begin{example}\label{EXA:EZ}
Let $\EZ:\Del\to \sSet$ be given as follows. Let $\EZ^n\in \sSet$ be the nerve of the category $E\Z_{n+1}^{\Cat}$, whose objects are elements of $\Z_{n+1}=\Z/(n+1)\Z$, and which has exactly one morphism between any two objects. More explicitly, $\EZ^n=E\Z_{n+1}=\mathcal N(E\Z_{n+1}^{\Cat})$ has $k$-simplicies given by a sequence of $k$ composable morphisms $\llbracket i_0\rrbracket\to \llbracket i_1\rrbracket\to \dots \to \llbracket i_k\rrbracket$ where $\llbracket i_0\rrbracket,\dots, \llbracket i_k\rrbracket\in \Z_{n+1}$, or, more concisely, the $k$-simplicies $\EZ^n_k$ are sequences $(i_0,\dots, i_k)$ of indices $i_0,\dots, i_k\in \{0,\dots, n\}$, i.e., $\EZ^n_k\cong \{0,\dots, n\}^k$. Face maps $d_j:\EZ^n_k\to \EZ^n_{k-1}$ remove the $j$th index $i_j$, and degeneracies $s_j:\EZ^n_k\to \EZ^n_{k+1}$ repeat the $j$th index $i_j$. 

For example, for the simplicial set $\EZ^1$ a $k$-simplex consists of a sequence $(i_0,\dots,i_k)$ of $0$s and $1$s; a $k$-simplex is degenerate iff any two adjacent indices are equal, $i_j=i_{j+1}$; thus there are exactly two non-degenerate $k$-simplicies: $(0,1,0,1,\dots)$ and $(1,0,1,0,\dots)$ for any $k$. The geometric realization of $\EZ^1$ is thus $S^\infty$.

By lemma \ref{LEM:sSet(X,dgN(Perf))}, any simplicial set map $\EZ^n\to \dgN(\perf(U))$ is given by $n+1$ holomorphic dg-vector bundles with holomorphic connections $\mc E_0,\dots, \mc E_n$ together with maps $g_{i_0\dots i_k}:E_{i_k}\to E_{i_0}$ for a non-degenrate $k$-simplex $\alpha=(i_0,\dots, i_k)\in \EZ^n_k=\{0,\dots, n\}^k$ without directly repeating indices, satisfying equation \eqref{EQU:dg-alpha-relation}:
\begin{multline}\label{EQU:dg-EZ-relation}
g_{i_0\dots i_k}\circ d+(-1)^k \cdot d\circ g_{i_0\dots i_k}=D(g_{i_0\dots i_k})\\
=\sum_{j=1}^{k-1} (-1)^{j-1} g_{i_0\dots \widehat{i_j}\dots i_k}+\sum_{j=1}^{k-1} (-1)^{k(j-1)+1} g_{i_0\dots i_{j}}\circ g_{i_{j}\dots i_k}.
\end{multline}
Note furthermore, that for a \emph{degenerate} simplex $(i_0,\dots, i_k)$ of $\EZ^n$ where the two consecutive indices $i_j=i_{j+1}$ are equal, we also have a map $g_{jj}=id_{E_j}$ or $g_{i_0\dots j j \dots i_k}=0$ when $k\geq 2$, satisfying equation \eqref{EQU:dg-EZ-relation}.

For a morphism $\phi:[n]\to [m]$ in $\Del$ we get an induced map of simplicial sets $\phi_\bu:\EZ^n_\bu\to\EZ^m_\bu$ by mapping $\phi_k:\EZ^n_k\to\EZ^m_k$, $\phi_k(i_0,\dots,i_k)=(\phi(i_0),\dots, \phi(i_k))$. This gives the cosimplicial simplicial set $\EZ$. In particular, we can use definition \ref{DEF:PERF^Q} to get the simplicial set $\PERF^\EZ(U)$, whose $n$-simplicies are precisely $\PERF^{\EZ}(U)_n=\sSet(\EZ^n,\dgN(\perf(U))^\circ)$, i.e., simplicial set maps from $\EZ^n$ to $\dgN (\perf(U))^\circ$, which were described explicitly in the previous paragraph.
\end{example}

We note, that for the simplicial presheaf $\PERF^{\EZ}$ the ``maximal Kan'' condition follows automatically.
\begin{lemma}\label{LEM:perf=perf-Kan}
 Simplicial set maps from $\EZ^n$ to $\dgN(\perf(U))$ take values in its maximal Kan subsimplex, i.e., 
\begin{equation}\label{EQU:EZ-in-Kan}
\PERF^{\EZ}(U)_n=\sSet(\EZ^n,\dgN(\perf(U))^\circ)=\sSet(\EZ^n,\dgN(\perf(U))).
\end{equation}
\end{lemma}
\begin{proof}
Any edge $g_{i_0 i_1}$ is automatically a homotopy equivalences with chain homotopy inverse $g_{i_1 i_0}$, since we have the homotopies $g_{i_0 i_1 i_0}\circ d+d\circ g_{i_0 i_1 i_0}=g_{i_0 i_0}-g_{i_0 i_1}\circ g_{i_1 i_0}=id_{ E_{i_0}}-g_{i_0 i_1}\circ g_{i_1 i_0}$ and $g_{i_1 i_0 i_1}\circ d+d\circ g_{i_1 i_0 i_1}=g_{i_1 i_1}-g_{i_1 i_0}\circ g_{i_0 i_1}=id_{ E_{i_1}}-g_{i_1 i_0}\circ g_{i_0 i_1}$. The claim follows from lemma \ref{LEM:simplicial-dgN(Perf(U))^circ}.
\end{proof}

Note that there is a map of cosimplicial simplicial sets $\Delta\to \EZ$, given by $\Delta^n_k\to \EZ^n_k$, $\Delta^n_k=\Del([k],[n])\ni \phi \mapsto (i_0,\dots, i_k):=(\phi(0),\dots,\phi(k))\in \EZ^n_k$. We thus get an induced map of simplicial sets $\PERF^{\EZ}(U)\to \PERF^{\Delta}(U)$.
\begin{proposition}\label{PROP:PERF-EZ=PERF-Delta}
For an object $U\in \CMan$, the map of simplicial sets $\PERF^{\EZ}(U)\to \PERF^\Delta(U)$ is a weak equivalence.
\end{proposition}
\begin{proof}
Since  $\dgN(\perf(U))^\circ$ is (by definition) a Kan complex, and $\PERF^\EZ(U):=\sSet(\EZ^\bu,\dgN(\perf(U))^\circ)$ and $\PERF^\Delta(U)=\sSet(\Delta^\bu,\dgN(\perf(U))^\circ)$, the proposition follows from proposition \ref{PROP:X=whX}. 
\end{proof}

In the later chapters we are mainly use $\PERF^Q$ for $Q=\EZ$, and we therefore make the following definition.
\begin{definition}\label{DEF:PERF}
Denote by $\IVB:=\PERF^{\EZ}:\CMan^{op}\to \sSet$, i.e., by \eqref{EQU:DEF:PERF^Q}, 
\begin{equation}\label{EQU:DEF-IVB(U)n}
\IVB(U)_n=\PERF^{\EZ}(U)_n=\sSet(\EZ^n,\dgN(\perf(U))^\circ).
\end{equation}
For a motivation of this notation, see definition \ref{DEF:inf-vec-bun}.
\end{definition}
The reason why we want to consider the cosimplicial simplicial set $\EZ$ is that it has an important additional cyclic structure which $\Delta$ is lacking, as we will explain now.

\begin{definition}
Let $\DC$ be the cyclic category; cf. \cite[6.1.1]{Loday}. More precisely, $\DC$ has the same objects $[n]=\{0,\dots,n\}$ for $n\in \N_0$ as $\Del$, and has morphisms generated by face maps $\delta_j$ and degeneracy maps $\sigma_j$ (as in $\Del$; cf. notation \ref{NOT:Delta-etc}), together with an additional cyclic operator $\tau_n:[n]\to [n]$; see \cite[6.1.1]{Loday} for more details. It is convenient to represent morphisms $\phi\in \DC([k],[n])$ by set maps $\phi:[k]\to [n]$ such that there exists a non-decreasing function $\wt \phi:\{0,\dots, k\}\to \N_0$ with $\wt \phi(k)\leq \wt \phi(0)+n$ so that $\phi(j)\equiv \wt \phi(j) ($mod $\Z_n)$.

Then, a cyclic object in a category $\mc C$ is a functor $X:\DC^{op}\to \mc C$. Since $\DC\cong \DC^{op}$ are isomorphic (\cite[6.1.11]{Loday}), cyclic objects in $\mc C$ are cocyclic objects in $\mc C$ and vice versa. We denote the category of cyclic sets $X:\DC \to \Set$ by $\cSet$. Note that there is functor $\Del\to \DC$, which makes every cyclic object into a simplicial object by precomposition $(\DC^{op}\stackrel X \to \mc C)\mapsto(\Del^{op}\to\DC^{op}\stackrel X \to\mc C)$, and similarly every cocyclic object is a cosimplicial object. In particular, every cosimplicial cyclic set is a cosimplicial simplicial set
\end{definition}

\begin{remark}
Note, that we have the canonical cyclic sets $\Cyc^n:=\DC(-,[n])$, which assemble for various $n$ to a cocyclic cyclic set $\Cyc^\bu:\DC\to \cSet$. In particular, this is also a cosimplicial cyclic set $\Del\to \DC\stackrel{\Cyc^\bu}\to\cSet$, so that we also have a third example of a simplicial presheaf $\PERF^\Cyc$ using our setup from definition \ref{DEF:PERF^Q}. By lemma \ref{LEM:sSet(X,dgN(Perf))}, an $n$-simplex in $\PERF^\Cyc(U)$ is given by maps $g_{i_0\dots i_k}$ for any ``cyclic set of indices'' $i_0=\phi(0), \dots, i_k=\phi(k)$ for some $\phi\in \DC([k],[n])$ (for example, for $n=9$ we would have maps such as $g_{457034}: E_4\to  E_4$). Unfortunately, the analog of proposition \ref{PROP:PERF-EZ=PERF-Delta} does not hold, i.e., $\PERF^\Cyc(U)$ and $\PERF^\Delta(U)$ are in general not weakly equivalent. (For example, the non-degenerate simplicies of $\DC^1$ as sequences of indices are $(0)$, $(1)$, $(0,1)$, $(1,0)$, $(0,1,0)$, $(1,0,1)$ but no higher ones due to cyclicity, so that the geometric realization of $\DC^1$ is the $2$-sphere $S^2$.)
\end{remark}

Now, while $\Delta^n$ is not a cyclic set, $\EZ^n$ is a cyclic set, and we will need to use the additional cyclic structure of $\EZ$ below to define our Chern character map.
\begin{lemma}
The simplicial set $\EZ^n$ as described in the first paragraph of example \ref{EXA:EZ}, together with the operator $t_k:\EZ^n_k\to \EZ^n_k$ given by $t_k(i_0,\dots, i_{k-1},i_k)=(i_k,i_0,\dots, i_{k-1})$ makes $\EZ^n$ into a cyclic set. This, in turn, makes $\EZ$ into a cosimplicial cyclic set.
\end{lemma}
\begin{proof}
Once checks that $t_k$ has the correct compatibility (see \cite[6.1.2 (b), (c)]{Loday}) with the face and degeneracy maps $d_j$ and $s_j$. For a morphism $\phi:[n]\to [m]$ in $\Del$ the induced map of simplicial sets $\phi_\bu:\EZ^n_\bu\to\EZ^m_\bu$, $\phi_k:\EZ^n_k\to\EZ^m_k$, $\phi_k(i_0,\dots,i_k)=(\phi(i_0),\dots, \phi(i_k))$ respects not only the face and degeneracy maps, but also the $t_k$ operator, i.e., $\EZ:\Del\to \cSet$ is a cosimplicial cyclic set.
\end{proof}

We have thus defined the simplicial presheaf $\IVB=\PERF^{\EZ}$, which will be the domain of our Chern character map for holomorphic dg-vector bundles over $U$ with connection. As for the range of the Chern character map, we use the same presheaf $\OM$ that we used in our previous paper \cite[Definition 2.3]{GMTZ} (for the Chern character map of holomorphic vector bundles that were not differential graded). For completeness sake, we will briefly review the definition of $\OM:\CMan^{op}\to \sSet$.

\begin{definition}\label{DEF:OM}
For an object $U\in \CMan$, consider the (non-negatively graded) cochain complex of holomorphic forms $\Ohol^\bu(U)$ on $U$ with zero differential $d=0$. Let $u$ be a formal variable of degree $|u|=-2$, denote by $\Ohol^\bu(U)[u]$ polynomials in $u$, and by $\Ohol^\bu(U)\ul$ its quotient by its positive degree part $\Ohol^\bu(U)[u]^{\bu>0}$. Applying the Dold-Kan functor to this chain complex gives a simplicial abelian group $DK(\Ohol^\bu(U)\ul)$, for which we consider its underlying simplicial set, denoted by an underline, i.e., $\OM(U)= \DKSet(\Ohol^\bu(U)\ul)$.
\[
\OM\quad :\quad \CMan^{op} \quad\stackrel{\Ohol^\bu(-)\ul}{\longrightarrow}\quad\Chain^- \quad \stackrel{\DKSet}{\longrightarrow} \quad\sSet
\]
Since holomorphic forms pull back via a holomorphic map $\varphi:U\to U'$, this assignment defines a simplicial presheaf $\OM:\CMan^{op}\to \sSet$ by $\OM:=\DKSet(\Ohol^\bu(\cdot)\ul):\CMan^{op}\to \sSet$.
\end{definition}

\begin{note}\label{REM:DK-def}
If $C=(C^{\bu\leq 0},d_C)$ is a non-positively graded chain complex, then the Dold-Kan functor $\DK(C)\in \sAb$, which is a simplicial abelian group, can be described as follows; cf. \cite[Appendix B]{GMTZ}. For $n\geq 0$, we may define $\DK(C)_n$ to be the abelian group (under addition) of cochain maps from the normalized cells of the standard simplex $\Delta^{n}$ to $C$, i.e., we may set
\begin{equation}\label{EQU:DK-def}
DK(C)_n:={\mc Chain}(N(\Z\Delta^{n}),C).
\end{equation}
Thus, this means that an element of $\DK(C)_n$ is given by a labeling of the non-degenerate cells of the standard simplex $\Delta^n$ by elements of $C$, in such a way that for a $k$-cell $\alpha$ of $\Delta^n$, whose boundary $(k-1)$-cells are $d_j(\alpha)$, we have 
\begin{equation}\label{EQU:DK-condition}
d_C(\alpha)=\sum_{j=0}^k (-1)^j\cdot d_j(\alpha).
\end{equation}
In the situation of definition \ref{DEF:OM}, the chain complex $C=\Ohol^\bu(U)\ul$ has a zero internal differential, i.e., $d_C=0$.
\end{note}

\section{Chern Character $\CH:\IVB\to \OM$}\label{SEC: Chern Map}

We now define a map of simplicial presheaves $\CH:\IVB\to \OM$, where $\IVB=\PERF^{\EZ}$ from definition \ref{DEF:PERF} and $\OM$ is from definition \ref{DEF:OM}. We start by defining cochains on a simplicial set $X$ with values in a dg category $\mc C$ (for us $\mc C=\perf(U)$), and, in the case when $X$ is a cyclic set, its trace map. The main example to keep in mind for the following definitions is the cyclic set $X=\EZ^n$.

\begin{definition}
A \emph{labelling} of a simplicial set $X$ by a dg-category $\mc C$ is a set map from the vertices of $X$ to the objects of $\mc C$, $L:X_0 \to Obj(\mc C)$, i.e., a choice of a objects $\mc E_{\alpha}:=L(\alpha)$ of $\mc C$ for each $\alpha \in X_0$.
\end{definition}

\begin{definition}\label{DEF:CLXHom}
Let $X$ be a simplicial set, let $\mc C$ be dg category, and let $L:X_0 \to Obj(\mc C)$ be a labelling so that we have a choice of objects $\mc E_\alpha$ for each $\alpha \in X_0$. We define the \emph{cochains on $X$ with values in $\mc C$} to be
\begin{equation}\label{EQU:DEF:CL(X,C)}
C_L^\bu(X,  \mc C):=\prod_{p\geq 1}\prod_{\alpha\in X_p} Hom^\bu_{\mc C}(\mc E_{\alpha(p)},\mc E_{\alpha(0)}),
\end{equation}
where we used notation from lemma \ref{LEM:sSet(X,dgN(Perf))} to denote the first and last vertices of $\alpha\in X_p$ by $\alpha(0)$ and $\alpha(p)$, respectively. In components, we will write $f\in C^\bu_L(X,\mc C)$ as $f=\{f_\alpha\}_{\alpha\in X}$, where, for $\alpha\in X_p$ and $p\geq 1$, we have $f_\alpha\in Hom^\bu_{\mc C}(\mc E_{\alpha(p)},\mc E_{\alpha(0)})$. 

Note, that $C_L^\bu(X,  \mc C)$ is a dg-algebra:
\begin{enumerate}
\item
a cochain  $f$ of bidegree $(p,q)$ assigns to a $p$-cell $\alpha\in X_p$ a degree $q$ map $f_\alpha \in Hom_{\mc C}^q(\mc E_{\alpha(p)}, \mc E_{\alpha(0)}) $, and is zero elsewhere; in this case the total degree of $f$ is $|f|=p+q$
\item\label{ITEM:CLXHom-differential}
a differential $\hat \delta:C_L^p(X, \mc C) \to C_L^{p+1} (X, \mc C)$ is induced by the face maps $d_i:X^{p+1} \rightarrow X^p$, so that if $\alpha\in X_{p+1}$ is a $p+1$ simplex of $X$, then the deleted \v{C}ech differential of $f$, denoted by $\hat \delta f$, is defined by 
\begin{equation}
(\hat \delta f)_\alpha := \sum_{i=1}^p (-1)^i f_{d_i (\alpha)} = \sum_{i=1}^p (-1)^i f_{\alpha( 0, \dots ,\widehat i, \dots, p+1)}
\end{equation}
Note that $d_0$ and $d_{p+1}$ are not used in the differential, which ensures the terms in the sum are all maps in $Hom^q_{\mc C}(\mc E_{\alpha(p+1)}, \mc E_{\alpha(0)})$.
\item\label{ITEM:CLXHom-internal-differential}
an internal differential $D: C_L^\bullet(X, \mc C) \to C_L^\bullet(X, \mc C)$ is induced by the dg structure on $\mc C$, so that if $\alpha\in X_p$ is a $p$-simplex and $f_\alpha \in Hom^q_{\mc C}(\mc E_{\alpha(p)},\mc E_{\alpha(0)})$, then $(Df)_\alpha:=(-1)^{p+q+1} \cdot D(f_\al)=(-1)^p\cdot (d\circ f_\al-(-1)^q\cdot f_\al\circ d)$ as a homomorphism in $Hom^{q+1}(\mc E_{\alpha(p)},\mc E_{\alpha(0)})$
\item\label{ITEM:CLXHom-product}
a product ``$f\cdot g$'' on $C_L^\bullet(X, \mc C)$, which, for $\alpha\in X_{p+r}$ is the extension of the maps $Hom^q_{\mc C}(\mc E_{\alpha(p)},\mc E_{\alpha(0)})\times Hom^s_{\mc C}(\mc E_{\alpha(p+r)},\mc E_{\alpha(p)})\to Hom^q_{\mc C}(\mc E_{\alpha(p+r)},\mc E_{\alpha(0)})$
\begin{equation}\label{EQU:CXHom-product}
(f_{\al(0,\dots, p)},g_{\al(p,\dots, p+r)}) \mapsto (f\cdot g)_{\al(0,\dots,p+r)}:= (-1)^{q\cdot r}\cdot f_{\alpha(0,\dots, p)}\circ g_{\alpha(p,\dots, p+r)}
\end{equation}
on the components of $C_L^\bullet(X, \mc C)$, to all of $C_L^\bullet(X, \mc C)$
\end{enumerate}
We note that in particular, $Df=d\cdot f-(-1)^{|f|}f\cdot d=:[d,f]$. It is well-known (and straightforward to check) that with these definitions the cochains on $X$ with values in $\mc C$, $C_L^\bu(X,  \mc C)$, becomes a dg-algebra.
\end{definition}

\begin{definition}\label{DEF: MC}
Given a simplicial set $X$, a dg category $\mc C$, and a labeling $L$. Then, we say an element $g\in C_L^\bullet(X, \mc C)$ is a \emph{Maurer Cartan element}, if
\begin{equation}\label{EQU:MC-condition}
\hat \delta g + D g+  g \cdot g =0
\end{equation}
We denote the set of all Maurer Cartan elements by  $MC(C_L^\bullet (X, \mc C))$.
\end{definition}

\begin{definition}\label{DEF:g-from-F:X->dgNerve(C)}
Let $X_\bu$ be a simplicial set, and let $\mc C$ be a dg-category. Then, by lemma \ref{LEM:sSet(X,dgN(Perf))}, a simplicial set map $\mc F:X\to \dgN(\mc C)$ induces objects $\mc E_\al$ for each $0$-simplex $\al\in X_0$, and maps $g_\alpha:\mc E_{\alpha(p)}\to \mc E_{\alpha(0)}$ for every $\alpha\in X_p$ with $p\geq 1$ (for degenerate simplicies, we take $g_{\alpha}=id_{\mc E_{\alpha(0)}}$ when $\alpha\in X_1$, and $g_{\alpha}=0$ when $\alpha\in X_p$ for $p\geq 2$). Thus, we can define a labeling $L:=\mc F_0:X_0\to \dgN(\mc C)_0=Obj(\mc C)$ of $X$ by $\mc C$ via $L(\alpha):=\mc E_\al$ for $\al\in X_0$. Moreover the $g_\al$ for $\al\in X_p$ for $p\geq 1$, assemble to an element $g=\{g_\al\}_{\al\in X}\in C_L^\bu(X,\mc C)$.
\end{definition}

\begin{corollary}\label{COR:F:X-to-dgN(C)}
The element $g\in C_L^\bu(X,\mc C)$ from definition \ref{DEF:g-from-F:X->dgNerve(C)} is a Maurer Cartan element, i.e., $g$ satisfies \eqref{EQU:MC-condition}. Moreover, $g$ has components of bidegree $(p,1-p)$ for $p\geq 1$, so that $g$ is of total degree $1$.
\end{corollary}
\begin{proof}
Each $g_\al$ for $\al\in X_p$ is of bidegree $(p,1-p)$; cf. lemma \ref{LEM:sSet(X,dgN(Perf))}\eqref{ITEM:g-alpha-maps}. For $\al\in X_{p+r}$ with $p,r\geq 1$ we have $g_{\al(0,\dots, p)}\cdot g_{\al(p,\dots, p+r)}=(-1)^{(1-p)(r-p)} g_{\al(0,\dots, p)}\circ g_{\al(p,\dots, p+r)}$, and since $(-1)^{(1-p)(r-p)}=(-1)^{(r+p)(p-1)}$, we see that \eqref{EQU:MC-condition} becomes exactly \eqref{EQU:dg-alpha-relation}.
\end{proof}

Now, consider the case $\mc C=\perf(U)$. In this case $C_L^\bullet(X, \mc C)$ becomes a direct product of holomorphic sections, i.e.,
\[
C_L^\bullet(X, \perf(U))=\prod_{p\geq 1}\prod_{\alpha\in X_p}\Gamma_{hol}(U, Hom( E_{\alpha(p)}, E_{\alpha(0)}))
\]
since morphisms $Hom_{\mc C} (\mc E_{1},\mc E_{2})$, which are bundle maps are in correspondence with holomorphic sections of the $Hom(E_{1},E_{2})$-bundle. Since we want to include higher holomorphic forms as well, we will include this dg-algebra in a larger dg-algebra of all holomorphic forms $C_L^\bu(X, \perf(U))\hookrightarrow \OmXHom$ defined as follows.

\begin{definition}\label{DEF:OmXHom}
Let $X$ be a simplicial set, and consider the dg category $\perf(U)$. Let $L:X_0 \to Obj(\perf (U))$ be a labelling as in definition \ref{DEF:CLXHom}, i.e., $\mc E_\alpha=L(\alpha)$. We define the dg-algebra
\begin{equation}\label{EQU:DEF:CL(X,OxP(U))}
\OmXHom:=\prod_{p\geq 0}\prod_{\alpha\in X_p}\Ohol^\bu (U, Hom^\bu(E_{\alpha(p)}, E_{\alpha(0)})).
\end{equation}
where we again denoted the first and last vertices of $\alpha\in X_p$ by $\alpha(0)$ and $\alpha(p)$, respectively. In components, we will write $f\in \OmXHom$ as $f=\{f_\alpha\}_{\alpha\in X}$, where, for $\alpha\in X_p$, we have $f_\alpha\in \Ohol^\bu(U, Hom^\bu( E_{\alpha(p)}, E_{\alpha(0)}))$. Note, that in \eqref{EQU:DEF:CL(X,OxP(U))} we included the $0$-simplicies ($p=0$) when compared to \eqref{EQU:DEF:CL(X,C)}. 

The dg-algebra structure on $\OmXHom$ is defined as follows.
\begin{enumerate}
\item
an $f\in \OmXHom$ has the triple degree $(k,p,q)$, if it assigns to a $p$-cell $\alpha\in X_p$ a holomorphic $k$-form with values in the appropriate Hom-bundle of degree $q$, $f_\alpha\in \Ohol^k(U,Hom^q(E_{\alpha(p)}, E_{\alpha(0)}))$, and vanishes elsewhere; in this case the total degree of $f$ is $|f|=k+p+q$
\item
a differential $\hat \delta:\OmXHom\to \OmXHom$, the deleted \v{C}ech differential, is defined just as in definition \ref{DEF:CLXHom}\eqref{ITEM:CLXHom-differential}, i.e., for $f\in \OmXHom$,
\begin{equation}
(\hat \delta f)_\alpha := \sum_{i=1}^p (-1)^i f_{d_i (\alpha)} = \sum_{i=1}^p (-1)^i f_{\alpha( 0, \dots ,\widehat i, \dots, p+1)}
\end{equation}
\item
a differential $D: \OmXHom\to \OmXHom$, the internal differential, is defined similarly to definition \ref{DEF:CLXHom}\eqref{ITEM:CLXHom-internal-differential}, i.e., if $f_\alpha\in \Ohol^k(U,Hom^q( E_{\alpha(p)}, E_{\alpha(0)}))$, then $(Df)_\al\in \Ohol^k(U, Hom^{q+1}( E_{\alpha(p)}, E_{\alpha(0)}))$,
\[ (Df)_\alpha:=(-1)^ p\cdot (d_{\al(0)}\circ f_\al-(-1)^{k+q}\cdot f_\al\circ d_{\al(p)})
\]
where $d_i$ denotes the differential of $E_i$
\item
a product ``$f\cdot g$'' similar to definition \ref{DEF:CLXHom}\eqref{ITEM:CLXHom-product}.  
More explicitly, consider the maps $\Ohol^k(U, Hom^q(\mc E_{\alpha(p)},\mc E_{\alpha(0)}))\times \Ohol^\ell (U,Hom^s(\mc E_{\alpha(p+r)},\mc E_{\alpha(p)}))\to \Ohol^{k+\ell}(U,Hom^q(\mc E_{\alpha(p+r)},\mc E_{\alpha(0)}))$
\begin{align}
(f_{\al(0,\dots, p)},g_{\al(p,\dots, p+r)}) \mapsto & (f\cdot g)_{\al(0,\dots, p+r)}
\\ & := (-1)^{(k+q)\cdot r}\cdot f_{\alpha(0,\dots, p)}\circ g_{\alpha(p,\dots, p+r)} \nonumber
\end{align}
where ``$\circ$'' denotes wedging forms and composing Hom-spaces, and extend them from the components of $\OmXHom$ to the whole space.
\end{enumerate}
We note that, again, $Df=d\cdot f-(-1)^{|f|}f\cdot d=[d,f]$. Just as in definition \ref{DEF:CLXHom}, $\OmXHom$ becomes a dg-algebra, and the inclusion $C_L^\bu(X,  \perf(U))\hookrightarrow \OmXHom$ is a dg-algebra morphism. Note that this inclusion consists of two separate inclusions of holomorphic functions into holomorphic forms, $\Gamma_{hol}(\dots)\hookrightarrow \Ohol(\dots)$, as well as non-zero simplicies into all simplicies, $\prod_{p\geq1}(\dots)\hookrightarrow \prod_{p\geq 0}(\dots)$. Note further, that $Df=d\cdot f-(-1)^{|f|} f\cdot d$, where $d=\{d_\al\}_{\al\in X}\in \OmXHom$ is given by the differentials $d_\al=d_{E_\al}$ for $\al\in X_0$ and $d_\al=0$ for all other $\al$.

Finally we remark, that  every Maurer Cartan element in $C_L^\bu(X,  \perf(U))$ is also a Maurer Cartan element in the larger dg-algebra $\OmXHom$.
\end{definition}

Now, for a vector bundle $E$, there is a trace map $tr:Hom(E,E)\to \C$. Following ideas from O'Brian-Toledo-Tong, \cite[page 238]{OTT1}, we will define a trace map
\[
\prod_{p\geq 0}\prod_{\alpha\in X_p}\Ohol^\bu (U, Hom^\bu(E_{\alpha(p)}, E_{\alpha(0)}))\to  \prod_{p\geq 0}\prod_{\alpha\in X_p} \Ohol^\bu(U,\C)
\]
Note that the left hand side is $\OmXHom$, and we denote the right hand side by $C^\bu(X,\Ohol(U)):= \prod_{p\geq 0}\prod_{\alpha\in X_p} \Ohol^\bu(U,\C)$. To fit this into our current setting, we need an additional \emph{cyclic} structure on $X$.
\begin{definition}\label{DEF:trace_g}
Let $X$ be a cyclic set. Let $\alpha\in X_p$ be a $p$-simplex, i.e., by our convention $\alpha=\alpha(0,\dots, p)$, then using the additional operator $\tau_p:[p]\to [p]$, we denote the induced map $t_p:X_p\to X_p$ by $\alpha(p,0,\dots, p-1):=t_p(\alpha)$.

Now, let $L:X_0 \to Obj(\perf (U))$ be a labelling, and let $g$ be a Maurer Cartan element of $C_L^\bu(X,  \perf(U))$. Then, we define the \emph{trace} map
\begin{align*}
Tr_g:&\OmXHom\to C^\bu(X,\Ohol(U))
\\
(Tr_g(f))_{\alpha\in X_s} :=& \sum_{0\leq k\leq \ell\leq s}(-1)^{(k+1)\cdot s+\ell-k}\cdot tr\left(g_{\alpha(\ell, \cdots, s, 0, \cdots ,k )} \circ f_{\alpha( k, \cdots,\ell )}  \right) 
\end{align*}
Note, that the trace on the right makes sense, since it is applied to $Hom(E_{\alpha(\ell)},E_{\alpha(\ell)})$.
\end{definition}

The following proposition follows the arguments from \cite[Proposition 3.2]{OTT1}.

\begin{proposition}\label{LEM:tr-is-chain-map}
Let $X$ be a cyclic set with labeling $L$, and let $g$ be a Maurer Cartan element in $C_L^\bu(X,  \perf(U))$. Then, the trace map $Tr_g$ satisfies
\begin{equation}\label{EQU:Trg-chain-map}
Tr_g \circ (\hat \delta + D + [g,-]) = \delta \circ Tr_g,
\end{equation}
where $\delta$ on the right is the (full) \v Cech differential, including first and last term, i.e., $(\delta f)_\alpha := \sum_{j=0}^{p+1} (-1)^j f_{\alpha( 0, \dots ,\hat j, \dots, p+1)}$ for $\alpha \in X_{p+1}$.
\end{proposition}
\begin{proof}
Let $f \in \OmXHom$, and let $\al\in X_s$. Then 
\[
(\delta (Tr_g(f)))_{\alpha}  = \sum_{j=0}^{s} (-1)^j \cdot Tr_g(f)_{\alpha(0,\dots,\widehat j,\dots, s)}=A+B+C
\]   
equals the sum of the following three terms:
\begin{align*}
A:=\hspace{-1mm}  &\sum_{0\leq k\leq \ell\leq s}\sum_{j=k+1}^{\ell-1}(-1)^{j+(k+1)(s-1)+\ell-k-1}\cdot  tr\left(g_{\alpha(\ell, \cdots, s, 0, \cdots ,k )} \circ f_{\alpha( k, \cdots,\widehat j, \cdots,\ell )}  \right)
\\
B:=\hspace{-2mm} &\sum_{0\leq k\leq \ell\leq s}\sum_{j=0}^{k-1}(-1)^{j+k(s-1)+\ell-k}\cdot  tr\left(g_{\alpha(\ell, \cdots, s, 0, \cdots ,\widehat j, \cdots,k )} \circ f_{\alpha( k, \cdots,\ell )}  \right)
\\
C:=\hspace{-2mm}  &\sum_{0\leq k\leq \ell\leq s}\sum_{j=\ell+1}^{s}(-1)^{j+(k+1)(s-1)+\ell-k}\cdot  tr\left(g_{\alpha(\ell, \cdots ,\widehat j, \cdots, s, 0, \cdots ,k )} \circ f_{\alpha( k, \cdots,\ell )}  \right)
\end{align*}
The first term $A$ in the above sum is equal to
\begin{align*}
A&=
 \sum_{0\leq k\leq \ell\leq s}\sum_{j=k+1}^{\ell-1}(-1)^{j+(k+1)s+\ell}\cdot  tr\left(g_{\alpha(\ell, \cdots, s, 0, \cdots ,k )} \circ f_{\alpha( k, \cdots,\widehat j, \cdots,\ell )}  \right)
\\&=
\sum_{0\leq k\leq \ell\leq s} (-1)^{(k+1)s+\ell-k}\cdot  tr\left( g_{\alpha(\ell, \cdots, s, 0, \cdots ,k )} \circ (\hat \delta f)_{\alpha( k, \cdots,\ell )}  \right) =(Tr_g( \hat \delta(f)))_\alpha.
\end{align*}
To evaluate $B+C$, note that
\begin{align}
\label{EQU:B+C}
&\sum_{0\leq k\leq \ell\leq s} (-1)^{(k+1)s+1}\cdot tr\left( (\hat \delta (g))_{\al(\ell,\dots, s, 0,\dots, k)}\circ f_{\al(k\dots, \ell)} \right)
\\ \nonumber
 =& \sum_{0\leq k\leq \ell\leq s}\sum_{j=\ell+1}^{s}(-1)^{(k+1)s+1+j-\ell}\cdot  tr\left(g_{\alpha(\ell, \cdots ,\widehat j, \cdots, s, 0, \cdots ,k )} \circ f_{\alpha( k, \cdots,\ell )}  \right)
 \\ \nonumber
 &+ \sum_{0\leq k\leq \ell\leq s}\sum_{j=0}^{k-1}(-1)^{(k+1)s+1+s-\ell+1+j}\cdot  tr\left(g_{\alpha(\ell, \cdots, s, 0, \cdots ,\widehat j, \cdots,k )} \circ f_{\alpha( k, \cdots,\ell )}  \right)
\\ \nonumber
=& C+B
\end{align}
We claim that this is equal to $(Tr_g(D(f)+[g,-](f)))_\al$, which we evaluate now. By definition \ref{DEF:OmXHom}, we may write $D(f)=d\cdot f-(-1)^{|f|}f\cdot d=[d,f]$, where $|f|$ denotes the total degree of $f$. Thus, if we define $\tg:=d+g$, i.e., for $\al\in X_0$, $\tg_{\al}=d_\al$, and for $\al\in X_{k}$ with $k\geq 1$, $\tg_{\al}=g_\al$, then $D(f)+[g,-](f)=[d+g,f]=[\tg,f]$. With this, we write $(Tr_g([\tg,f]))_\al=(Tr_g(\tg\cdot f-(-1)^{|\tg|\cdot |f|}f\cdot \tg))_\al=E+F$, which are given as follows. First, 
\begin{multline*}
E:=Tr_g(\tg \cdot f)_\al= \sum_{0\leq j\leq \ell\leq s}  (-1)^{(j+1)s+\ell-j}\cdot tr\left(g_{\al(\ell,\dots, s,0,\dots, j)}\circ (\tg \cdot f)_{\al(j,\dots,\ell)}\right)
\\
= \sum_{0\leq j\leq k\leq \ell\leq s}  (-1)^{(j+1)s+\ell-j+(1-k+j)(\ell-j)}\cdot tr\left(g_{\al(\ell,\dots, s,0,\dots, j)}\circ \tg_{\al(j,\dots, k)} \circ f_{\al(k,\dots,\ell)}\right)
\end{multline*}
where we used that the (De Rham, \v Cech, Hom)-triple degree of $\tg_{\al(j,\dots, k)}$ is $(0,k-j, 1-k+j)$. For the second term, we get
\begin{align*}
&F:=Tr_g(-(-1)^{|\tg|\cdot |f|}f\cdot \tg)_\al
\\ 
=& \sum_{0\leq k\leq j\leq s}  (-1)^{|f|+1+(k+1)s+j-k}\cdot tr\left(g_{\al(j,\dots, s,0,\dots, k)}\circ (f\cdot \tg)_{\al(k,\dots,j)}\right)
\\
=&  \sum_{0\leq k\leq \ell\leq j\leq s}  (-1)^{|f|+1+(k+1)s+j-k+(|f|-\ell+k)(j-\ell)} 
\\
& \hspace{6cm}\cdot tr\left(g_{\al(j,\dots, s,0,\dots, k)}\circ f_{\al(k,\dots, \ell)}\circ \tg_{\al(\ell,\dots,j)}\right)
\\
=&  \sum_{0\leq k\leq \ell\leq j\leq s}  (-1)^{|f|+1+(k+1)s+j-k+(|f|-\ell+k)(j-\ell)+(|f|+1-1-s+j-\ell)(1-j+\ell)} 
\\
& \hspace{6cm} \cdot tr\left(\tg_{\al(\ell,\dots,j)}\circ g_{\al(j,\dots, s,0,\dots, k)}\circ f_{\al(k,\dots, \ell)}\right)
\end{align*}
where we used that $tr(h\circ k)=(-1)^{a\cdot b}\cdot tr(k\circ h)$ when the (Hom-degree)$+$(De Rham-degree)$=$(total degree)$-$(\v Cech-degree) of $h$ and $k$ is $a$ and $b$, respectively, and that the \v Cech-degree of any $h_{\al(j,\dots, s,0,\dots, \ell)}$ is $1+s-j+\ell$. With this, we obtain
\begin{align}
 \label{EQU:E+F} &
 \sum_{0\leq k\leq \ell\leq s} (-1)^{(k+1)s}\cdot tr( (\tg \cdot \tg)_{\al(\ell,\dots, s, 0,\dots, k)}\circ f_{\al(k,\dots, \ell)})
\\ \nonumber
=& \sum_{0\leq k\leq \ell\leq j\leq s}  (-1)^{(k+1)s+(1-j+\ell)(1+s-j+k)} \cdot tr\left(\tg_{\al(\ell,\dots,j)}\circ g_{\al(j,\dots, s,0,\dots, k)}\circ f_{\al(k,\dots, \ell)}\right)
\hspace{-7mm}  
\\ \nonumber
& + \sum_{0\leq j\leq k\leq \ell\leq s}  (-1)^{(k+1)s+(\ell-s-j)(k-j)}\cdot tr\left(g_{\al(\ell,\dots, s,0,\dots, j)}\circ \tg_{\al(j,\dots, k)} \circ f_{\al(k,\dots,\ell)}\right)
\hspace{-7mm} 
\\ \nonumber
=&  F+E
\end{align}
where we used that $\tg=d+g$ and $d\cdot d=0$, and that the (DeRham, \v Cech, Hom)-triple degree of $g_{\al(\ell,\dots,s,0,\dots, j)}$ is $(0,1+s-\ell+j, \ell-s-j)$. Comparing the left-hand-sides of \eqref{EQU:B+C} and \eqref{EQU:E+F}, and using that $g$ is a Maurer Cartan element, so that $\hat \delta g=-(Dg+g\cdot g)=-\tg\cdot \tg$, we obtain that
\[
B+C=\eqref{EQU:B+C}=\eqref{EQU:E+F}=E+F=(Tr_g([\tg,f]))_\al=(Tr_g(D(f)+[g,-](f)))_\al.
\]
This concludes the proof of the proposition.
\end{proof}

We have one further structure on $\OmXHom$ coming from the holomorphic connections $\nabla$ of the objects $\mc E$ of $\perf(U)$. Note that there is an induced connection on the Hom-bundle $Hom^\bu(E,E')$ of two graded bundles $E$ and $E'$ with connections, which we also denote by $\nabla:\Ohol^\bu(U,Hom^\bu(E,E'))\to \Ohol^{\bu+1}(U,Hom^\bu(E,E'))$, and which is a graded derivation with respect to the wedge-composition ``$\circ$'' using the total degree of $\Ohol^{\bu}(U,Hom^\bu(E,E'))$.
\begin{definition}
Define $\nabla:\OmXHom\to \OmXHom$ to be given in components by the maps $(-1)^p\cdot \nabla:\Ohol^k(U, Hom^q(\mc E_{\alpha(p)},\mc E_{\alpha(0)}))\to \Ohol^{k+1}(U, Hom^q(\mc E_{\alpha(p)},\mc E_{\alpha(0)}))$. More explicitly, for $f\in \OmXHom$, $f=\{f_\al\}_{\al\in X}$, we define $\nabla f=\{(\nabla f)_\al\}_{\al\in X}$ to be given by $(\nabla f)_\al:=(-1)^{p}\cdot \nabla(f_\al)$ when $\al\in X_p$. 

One can check that $\nabla\circ \hat \delta=-\hat \delta\circ\nabla$, and that $\nabla(f\cdot g)=\nabla(f)\cdot g+(-1)^{|f|}f\cdot \nabla(g)$, where $|f|$ is the total degree of the triple grading.
\end{definition}

\begin{definition}
Let $X$ be a cyclic set and let $\mc F:X\to \dgN(\perf(U))$ be a simplicial set map. By definition \ref{DEF:g-from-F:X->dgNerve(C)}, we get a labeling $L:X_0\to Obj(\perf(U))$, and a Maurer Cartan element $g\in C_L^\bu(X,  \perf(U))\hookrightarrow \OmXHom$. For a vertex $\alpha\in X_0$, denote by $d_{E_\alpha}$ the internal differential of the chain complex of vector bundles $\mc E_\alpha$, out of which we build the element $d=\{d_\al\}_{\al\in X}\in \OmXHom$, given by $d_\alpha:=d_{E_\alpha}$, and which has triple degree $(0,0,1)$. Then, $d+g\in \OmXHom$, and we call $$A:=\nabla(d+g)\quad \in \OmXHom$$ the \emph{Atiyah class}, which is concentrated in degrees $(1,k,1-k)$ for $k\geq 0$.
\end{definition}

\begin{proposition}\label{PROP:delta(Tr(A^k))=0}
We have $(\hat \delta + D + [g,-])(a)=0$, and thus:
\[
\forall k\geq 0: \quad\quad \delta(Tr_g(A^k))=0
\]
\end{proposition}
\begin{proof}
We apply $\nabla$ to the Maurer Cartan equation \eqref{EQU:MC-condition}, i.e., to $\hat \delta g +Dg + g \cdot g =0$. Using $\nabla \hat \delta g=- \hat \delta \nabla g$, and $\nabla(g\cdot g)=\nabla g \cdot g- g\cdot \nabla g=-[g,\nabla g]$ together with
\[
\n Dg = \n (d\cdot g +g \cdot d)=\n d\cdot g -d\cdot \n g+\n g\cdot d -g\cdot \n d
=-D(\n g)-[g,\n d]
\]
we obtain
\begin{align*}
0=&\n (\hat \delta g +Dg + g \cdot g)=-\hat \delta (\n g)-D(\n g)-[g,\n d]-[g,\n g]
\\
=&-(\hat \delta +D+[g,-])(\n g+\n d)
\end{align*}
In the last equality, we also used that $\hat \delta(\nabla d)=0$ (since the deleted \v Cech differential vanishes on $0$-simplicies), and from $d^2=0$ it follows that $0=\nabla(d\cdot d)=\nabla d\cdot d-d\cdot \nabla d=-D(\nabla d)$. This shows that for $A=\n (d+g)$ we have $(\hat \delta + D + [g,-])(a)=0$.

Since $(\hat \delta + D + [g,-])$ is a derivation on $\OmXHom$, the $k$th powers of $a$ also satisfy $(\hat \delta + D + [g,-])(A^k)=0$. Thus,
\[
\delta(Tr_g(A^k))\stackrel{\eqref{EQU:Trg-chain-map}}=Tr_g((\hat \delta + D + [g,-])(A^k))=0,
\]
which is the claim.
\end{proof}

We are now ready to define our Chern character map $\CH:\IVB\to \OM$, which is a map of simplicial presheaves, as shown in theorem \ref{THM:CH-is-map-of-simplicial-presheaves} below.
\begin{definition}\label{DEF:CH:PERF-to-OM}
We define the \emph{Chern character} as a map $\CH:\IVB\to \OM$, that is, for a complex manifold $U$, and $k\geq 0$, we define a map $\CH(U)_n:\IVB(U)_n\to \OM(U)_n$.

For an $n$-simplex $\mc F\in \IVB(U)_n=\sSet(\EZ^n,\dgN(\perf(U))$, we have (by definition \ref{DEF:PERF^Q} and example \ref{EXA:EZ}) the data of $n+1$ dg-vector bundles $\mc E_0,\dots, \mc E_n$, and maps $g_{i_0\dots i_k}:E_{i_k}\to E_{i_0}$, so that $g=\{g_{(i_0\dots i_k)}\}_{(i_0,\dots, i_k)\in \EZ^n}$ satisfies the Maurer Cartan equation by corollary \ref{COR:F:X-to-dgN(C)}. To this we associate $\CH(U)_n(\mc F)\in \OM(U)_n$, which is a labeling of the non-degenerate cells of $\Delta^n$ by elements in $\Ohol^\bu(U)\ul$ (by definition \ref{DEF:OM} and note \ref{REM:DK-def}). Consider a non-degenerate $k$-cell of $\Delta^n$ given by the vertices $i_0,\dots, i_k$ of $\Delta^n$ with $i_0<\dots<i_k$.

If $k=0$, then we assign the Euler characteristic $\chi( E_{i_0})$ to this cell. If $k>0$, then we use $\alpha=(i_0,\dots, i_k)\in \EZ^n_k$ to assign the following expression to this cell:
\begin{align}\label{EQ: Chern Tr A^k}
Tr_g(A^k)_{\alpha}\cdot \frac{u^k}{k!} &=Tr_g((\nabla (d+g))^k)_{\alpha}\cdot \frac{u^k}{k!} \\
&\nonumber = \sum \pm tr (g \cdot \nabla (d+g)\cdot   \nabla (d+g) \cdot \ldots \cdot \nabla (d+g))_\alpha \cdot \frac{u^k}{k!}
\end{align}

For example, here are the assignments for simplicial degrees $0$, $1$, and $2$.
\begin{enumerate}
\item[$n=0$:]
A $0$-simplex $\mc F\in \IVB(U)_0$ is just the data of one object $\mc E=(E \rightarrow U, \nabla)$ of $\perf(U)$. Then, $\CH(U)_0(\mc F)$ is the labeling of the $\Delta^0$ by Euler characteristic of $\mc E$, denoted $\chi ( E) \in \Ohol^0(U)\ul$.

\item[$n=1$:]
A $1$-simplex $\mc F\in \IVB(U)_1$ consists of bundles $\mc E_0$ and $\mc E_1$ and sequences of morphism $g_{0101\dots}$ and $g_{1010\dots}$.  Then, $\CH(U)_1(\mc F)$ is the labeling of $\Delta^1$ given by $\chi(\mc E_i)$ on the vertices of $\Delta^1$, and on the edge of $\Delta^1$ we place the labeling $Tr_g(\nabla (d+g))_{(0,1)}\cdot u\in \Ohol^1(U)\ul$, where $(0,1)\in \EZ^1$:
\begin{equation*}
\begin{tikzpicture}[scale=0.5]
\node (E0) at (0, 0) {}; \fill (E0) circle (4pt) node[above] {$\chi(E_0)$};
\node (E1) at (12, 0) {}; \fill (E1) circle (4pt) node[above] {$\chi(E_1)$};
\draw [<-] (E0) -- node[above] {$ Tr_g( \nabla (d+g))_{(0,1)}\cdot u$}(E1);
\end{tikzpicture}
\end{equation*}
Explicitly, the trace has the following terms (using $g_i=d_{E_i}$ for the internal differential of $E_i$):
\begin{equation*}
Tr_g( \nabla (d+g))_{(0,1)} = tr( g_{101} \nabla g_1 -  g_{010} \nabla g_0 +  g_{10} \nabla g_{01})
\end{equation*}
  
\item[$n=2$:]
A $2$-simplex $\mc F\in \IVB(U)_2$ consists of bundles $\mc E_0$, $\mc E_1$ and $\mc E_2$ and sequences of morphism $g_{i_0 i_1\dots i_p}$ for $p\geq 1$ and $i_\ell\in \{0,1,2\}$ for any $0\leq \ell\leq p$.  Then, $\CH(U)_2(\mc F)$ is the labeling of $\Delta^2$ given by $\chi(\mc E_i)\in \Ohol^0(U)\ul$ on the vertices, $Tr_g(\nabla (d+g))_{(i,j)}\cdot u\in \Ohol^1(U)\ul$ on the edge of $\Delta^1$ we place the labeling $Tr_g(\nabla (d+g)\cdot \nabla (d+g))_{(0,1,2)}\cdot \frac{u^2}{2!}\in \Ohol^2(U)\ul$ on the non-degenerate $2$-cell, where $(0,1,2)\in \EZ^2$:
\begin{equation*}\hspace{0.5cm}
\begin{tikzpicture}[scale=0.5]
\node (E0) at (0, 0) {}; \fill (E0) circle (4pt) node[below] {$\chi(E_0)$};
\node (E1) at (10, 4) {}; \fill (E1) circle (4pt) node[above] {$\chi(E_1)$};
\node (E2) at (20, 0) {}; \fill (E2) circle (4pt) node[below] {$\chi(E_2)$};
\draw [<-] (E0) -- node[above left] {$Tr_g(\nabla (d+g))_{(0,1)}\cdot u$} (E1);
\draw [<-] (E1) -- node[above right] {$Tr_g(\nabla (d+g))_{(1,2)}\cdot u$} (E2);
\draw [<-] (E0) -- node[below] {$Tr_g(\nabla (d+g))_{(0,2)}\cdot u$} (E2);
\node (C) at (10,1) {$ Tr_g(\nabla (d+g)\cdot \nabla (d+g))_{(0,1,2)}\cdot \frac{u^2}{2!}$};
\end{tikzpicture}
\end{equation*}
 Explicitly, we have (again using $g_i=d_{E_i}$ for the internal differential of $E_i$):
 \begin{align*}
Tr_g(\nabla (d+g)\cdot \nabla (d+g))_{(0,1,2)} =& tr( g_{20} \nabla g_0 \nabla g_{012}+  g_{20} \nabla g _{01}\nabla g_{12}+ g_{20} \nabla g_{012} \nabla g_2   ) 
\\
 &-  tr(g_{201} \nabla g_1 \nabla g_{12} + g_{201}\nabla g_{12}  \nabla g_2 )  
\\
 &-  tr(g_{120} \nabla g_0 \nabla g_{01} +  g_{120}\nabla g_{01}  \nabla g_1 )  
 \\
 &+ tr(g_{2012}\nabla g_2 \nabla g_2 +  g_{1201} \nabla g _1\nabla g_1 +  g_{0120}\nabla g_0 \nabla g_0 )
\end{align*}
\end{enumerate}
\end{definition}

\begin{theorem}\label{THM:CH-is-map-of-simplicial-presheaves}
The Chern character $\CH:\IVB\to \OM$ defined above is a map of simplicial presheaves.
\end{theorem}
\begin{proof}
We use the notation from definition \ref{DEF:CH:PERF-to-OM}. First, we note that $\CH(U)_n(\mc F)$ is a well-defined element of  $\OM(U)_n$, i.e., we still need to show that the labeling satisfies \eqref{EQU:DK-condition}. Since the internal differential vanishes for $\Ohol^\bu(\cdot)\ul$, this amounts to showing that for each $p$-cell given by $\alpha=(i_0,\dots,i_p)$, the sum of the labelings on the boundary cells vanishes. This follows from proposition \ref{PROP:delta(Tr(A^k))=0}:
\begin{align*}
\sum_{j=0}^k (-1)^j \cdot d_j\Big((Tr_g(A^k))_\alpha\cdot \frac{u^k}{k!}\Big)
&=\sum_{j=0}^k (-1)^j \cdot (Tr_g (A^k))_{\alpha(0,\dots,\widehat j, \dots, k)}\cdot \frac{u^k}{k!}
\\&=\delta\Big((Tr_g(A^k))_\alpha\Big)\cdot \frac{u^k}{k!}\stackrel{\ref{PROP:delta(Tr(A^k))=0}}=0
\end{align*}

Next, we show that $\CH(U):\IVB(U)\to \OM(U)$ is a map of simplicial sets, i.e., that it respects the face and degeneracy maps. If $\delta_j:[n]\to [n+1]$ is the $j$th face map, then $d_j:\IVB(U)_{n+1}\to \IVB(U)_{n}$ is given by pre-composition with $\EZ^{n}\to \EZ^{n+1}$, $ \{0,\dots,n\}^k\ni (i_0,\dots, i_{k})\mapsto(\delta_j(i_0),\dots, \delta_j( i_{k}))\in \{0,\dots,n+1\}^k$. Thus, for $\mc F\in \IVB(U)_{n+1}$ with corresponding Maurer Cartan element $g$, we have $\CH(U)_n\circ d_j(\mc F)|_{\alpha=(i_0<\dots <i_k)}=Tr_g(A^k)_{(\delta_j(i_0)<\dots< \delta_j(i_k))}\cdot \frac{u^k}{k!}$. This is equal to taking $\CH(U)_{n+1}(\mc F)\in \DK(C)_{n+1}={\mc Chain}(N(\Z\Delta^{n+1}),C)$ where $C=\Ohol^\bu(U)\ul$, after applying $d_j:\DK(C)_{n+1}\to \DK(C)_n$ to it, and looking at the labeling of the cell $i_0<\dots<i_k$ of $\Delta^n$. Similarly, if $\sigma_j:[n]\to [n-1]$ is the $j$th degeneracy, and $s_j:\IVB(U)_{n-1}\to \IVB(U)_{n}$ is the induced map, then, for $\mc F\in \IVB(U)_{n-1}$ with corresponding Maurer Cartan element $g$, we get $\CH(U)_n\circ s_j(\mc F)|_{\alpha=(i_0<\dots <i_k)}=Tr_g(A^k)_{ (\sigma_j  (i_0)\leq \dots\leq \sigma_j(i_k))}\cdot \frac{u^k}{k!}$. Now, if $\sigma$ is injective on $\{i_0,\dots,i_k\}$, then, by note \ref{REM:simplicial-dgN(Perf(U))}\eqref{ITEM:morphism-dgN(Perf(U))}, this is equal to $Tr_g(A^k)_{ (\sigma_j  (i_0)< \dots< \sigma_j(i_k))}\cdot \frac{u^k}{k!}$, which is the labeling of $s_j\circ \CH(U)_{n-1}(\mc F)$ at $i_0<\dots <i_k$. In the case where $\sigma_j$ is not injective on $\{i_0,\dots, i_k\}$, we get that $g_{\sigma_j(i_0)\dots \sigma_j(i_k)}$ is either the identity or zero, so that in either case $\nabla g_{\sigma_j(i_0)\dots \sigma_j(i_k)}=0$, and thus $\CH(U)_n\circ s_j(\mc F)|_{\alpha=(i_0<\dots <i_k)}=0$, which is equal to the degeneracy $s_j:\DK(C)_{n-1}\to \DK(C)_n$ applied to $\CH(U)_{n-1}(\mc F)$ at the cell $i_0<\dots<i_k$.

Finally, we show that $\CH:\IVB\to \OM$ is a map of simplicial presheaves, i.e., that under a holomorphic map $\varphi:U\to U'$, the following diagram commutes
\[
\begin{tikzcd}
\IVB(U') \arrow{r}{\CH(U')} \arrow[swap]{d}{\IVB(\varphi)} & \OM(U')\arrow{d}{\OM(\varphi)} \\
\IVB(U) \arrow{r}{\CH(U)}& \OM(U)
\end{tikzcd}
\]
This follows, since both compositions are given by pullback via $\varphi$, i.e., for $\mc F'\in \IVB(U')$ with induced Maurer Cartan element $g'$ and induced differential $d'$ on $E'_\al$, we have
\begin{multline*}
\CH(U)_n\circ  \IVB_n(\varphi)(\mc F')|_{\alpha=(i_0<\dots< i_k)}=
Tr_{\varphi^*g'}(((\varphi^*\nabla) (\varphi^*(d'+g')))^k)_{\alpha}\cdot \frac{u^k}{k!} \\
=\varphi^*(Tr_{g'}((\nabla (d'+g'))^k)_{\alpha})\cdot \frac{u^k}{k!}
=\OM(\varphi)_n\circ \CH(U')_n(\mc F')|_{\alpha=(i_0<\dots< i_k)}.
\end{multline*}

This completes the proof of the theorem.
\end{proof}

\section{A Higher Chern Character for Coherent Sheaves}\label{SEC-Sheafify}

In this section, we apply a construction, which we will call \v{C}ech sheafification, to the Chern character map $\CH:\IVB\to \OM$ from Definition \ref{DEF:CH:PERF-to-OM}. More precisely, an endofunctor on simplicial presheves $\F \mapsto \CechSh{\F}$ is defined as the colimit over all \v{C}ech covers of the totalization of the presheaf applied to the cover (see Definition \ref{DEF: cech sheafify}), and then an explicit interpretation is offered for the induced map $\CechSh{\CH}:\CechSh{\IVB}\to \CechSh{\OM}$. Theorem \ref{THM: OTT are vertices} states that $0$-simplicies of $\CechSh{\IVB}$ are twisting cochains (up to equivalence) in the sense of O'Brian, Toledo, and Tong in \cite{OTT1}, and, in theorem \ref{THM:CH(IVB)=CH(TwCoch)},  that the induced Chern character $\CechSh{\CH}$ recovers the Chern character from \cite{OTT1}.

To fix some notation, let $\left( U_i \to X\right)_{i \in I}$ be an open cover, which is a particular diagram in $\CMan$. To this cover, we associate the augmented simplicial presheaf $\CechNerve{U}_{\bullet} \to X$ whose $p$-simplices are coproducts of representable presheaves given by $(p+1)$-fold intersections of the cover, 
\[ \CechNerve{U}_{p} = \coprod\limits_{i_0, \dots, i_p\in I} y U_{i_0, \dots, i_p},\]
where $yU$ denotes the Yoneda functor applied to $U$, i.e., $yU:\CMan^{op} \to \Set, V\mapsto \CMan(V,U)$, interpreted as a constant simplicial set.  
Given another simplicial presheaf $\F$ we abuse notation by writing $\F \left( \CechNerve{U}_{\bullet}\right)$ for the cosimplicial simplicial set with
\begin{equation}\label{EQ: F of CN}
 \F \left( \CechNerve{U}_{\bullet}\right)^{\ell}_p := \prod\limits_{i_0 \dots i_\ell} \F\left( U_{i_0, \dots , i_\ell}\right)_p.
 \end{equation}

\begin{definition}\label{DEF: cech sheafify}
Given a simplicial presheaf $\F :\CMan^{op}\to \sSet$, define its \emph{\v{C}ech sheafification} on a test manifold $X \in \CMan$ to be the simplicial set given by
\begin{equation}\label{EQ: cech sheafify}
 \CechSh{{\F}}(X) := \colim_{( {U}_{\bu} \to X) \in \check{S}} Tot\left( {\F}(\CechNerve {U}_{\bu} ) \right),
 \end{equation}
where $\check{S}$ is the category of all \v{C}ech covers, and $Tot$ is the totalization, which is reviewed in appendix \ref{SEC: explicit tot}.  (For further details about the totalization, see \cite[Appendix D.1]{GMTZ} as well as \cite[Definition 18.6.3]{Hirs}; specific examples of $Tot$ are worked out in note \ref{EX: PERF Tot description} below, as well as in \cite[Proof of Proposition 3.16]{GMTZ}.)
\end{definition}

\begin{proposition}\label{PROP: Cech Sh F Kan}
If $\F$ is a simplicial presheaf which takes values in Kan complxes then its \v{C}ech sheafification is a Kan complex.
\begin{proof}
By proposition \ref{PROP: tot cech kan}, for an open cover $U_{\bu}$ of $X$, $Tot \left(\IVB \left( \CN U_{\bu}\right) \right)$ is a Kan complex. Now, since our colimit over \v{C}ech covers is directed once we pass to simplicial presheaves $\CN U_{\bu}$, then one can check by hand that the colimit in $\CechSh{\IVB}(X)$ sends a diagram of projectively fibrant objects to a projectively fibrant object (i.e. $\CechSh{\IVB}$ takes values in Kan complexes). 
\end{proof}
\end{proposition}

\begin{definition}\label{DEF: sheafified Chern}
The \emph{\v{C}ech sheafified Chern character map} $\CechSh{\CH}:\CechSh{\IVB} \to\CechSh{\OM}$ is the map obtained by applying \v{C}ech sheafifications to the Chern character map from Definition \ref{DEF:CH:PERF-to-OM}. 
\end{definition}

\subsection{\v{C}ech Sheafification of $\IVB$ as Twisting Cochains}\label{SEC:Cech sheafifi of IVB}

In this subsection, the vertices of the simplicial presheaf, $\CechSh{\IVB}$, are examined and shown in theorem \ref{THM: OTT are vertices} to be precisely the twisting cochains from \cite{OTT1} up to equivalence. We thus define:

\begin{definition}\label{DEF:inf-vec-bun}
An \emph{infinity vector bundle} over a complex manifold $X$ is a $0$-simplex of $\CechSh{\IVB}(X)$.
\end{definition}

The following note looks at the $k$-simplices of $\CechSh{\IVB}(X)$ in general, before focusing more specifically on the $0$-simplicies.

\begin{note}\label{EX: PERF Tot description}
Fix a complex manifold $X$. Definition \ref{DEF: cech sheafify} applied to $F = \CechSh{\IVB}$ yields 
\begin{equation}\label{EQ: perf sheafify example 1}
\CechSh{\IVB}(X)=\colim_{(U_\bu \to X) \in \check{S}} Tot\left( \IVB (\CechNerve \mathcal{U}_{\bu}) \right). 
\end{equation}
Now fix a \v{C}ech cover, $U_\bu\to X$, and denote by $K^\bu_\bu$ the cosimpicial simplicial set whose $\ell$-cosimplicies are given by
\begin{equation}
K^{\ell}:= \IVB(\CechNerve \mathcal{U}_{\ell}) =  \PERF^{\EZ}(\CechNerve \mathcal U_\ell).
\end{equation}
Following \eqref{EQU:k-simplex-in-Tot}, a $k$-simplex in $Tot(K)$ consists of a collection $\{x^{(k,\ell)}\}_{\ell\geq 0}$ with
\begin{align*}
x^{(k,\ell)}\in &\sSet(\Delta^k\times\Delta^\ell,  K^\ell) =\sSet(\Delta^k\times\Delta^\ell,  \PERF^{\EZ}(\CechNerve \mathcal U_\ell)) \\
&\stackrel{\eqref{EQU:DEF:PERF^Q}}{=}\sSet( \Delta^k\times\Delta^\ell  , \sSet(\EZ,\dgN(\perf(\CechNerve \mathcal{U}_{\ell}))^\circ) ) 
\\
&=  \sSet\left(  \colim\limits_{\Delta^p \to  \Delta^k\times\Delta^\ell  } \EZ^{p}, \dgN(\perf(\CechNerve \mathcal{U}_{\ell}))^\circ \right),
\end{align*}
where in the last equality the calculation from equation \eqref{EQU:Tot-for-Deltahat} is used. Thus, according to appendix \ref{SEC: explicit tot} page \pageref{PAGE:Tot-for-Deltahat-map} these are given by $x_{\scalebox{0.6}{$\bmat{[c|c|c]\al_0& \dots & \al_p  \\  \be_0 &  \dots & \be_p}$}}^{(k, \ell)} \in  \dgN(\perf(\CechNerve \mathcal{U}_{\ell}))^\circ_p$ for certain paths in the $(k+1)\times (\ell+1)$ grid, (i.e., for any path within the indices of a non-decreasing path). Now, before taking into account any simplicial or coherence conditions, the $p$-cell $x_{\scalebox{0.6}{$\bmat{[c|c|c]\al_0& \dots & \al_p  \\  \be_0 &  \dots & \be_p}$}}^{(k, \ell)} \in  \dgN(\perf(\CechNerve \mathcal{U}_{\ell}))^\circ_p$ is itself (by example \ref{EXA:EZ} and lemma \ref{LEM:perf=perf-Kan}) given by the following collection of data:
\begin{multline}\label{EQU:x^(k,l)-for-Tot-Delta-full}
x^{(k,\ell)}=\{x_{\scalebox{0.6}{$\bmat{[c|c|c]\al_0& \dots & \al_p  \\  \be_0 &  \dots & \be_p}$}}^{(k, \ell)}\}, \text{ where}
\\
x_{\scalebox{0.6}{$\bmat{[c|c|c]\al_0& \dots & \al_p  \\  \be_0 &  \dots & \be_p}$}}^{(k, \ell)}=
\Bigg(
E^{(k,\ell)}_{\scalebox{0.6}{$\bmat{[c|c|c]\al_0& \dots & \al_p  \\  \be_0 &  \dots & \be_p}$}; \scalebox{0.6}{$\bmat{[c] \al_j  \\  \be_j };i_0,\dots,i_\ell$}}\to U_{i_0,\dots,i_\ell},
 \nabla^{(k,\ell)}_{\scalebox{0.6}{$\bmat{[c|c|c]\al_0& \dots & \al_p  \\  \be_0 &  \dots & \be_p}$}; \scalebox{0.6}{$\bmat{[c] \al_j  \\  \be_j  };i_0,\dots,i_\ell$}}, \\
 g^{(k,\ell)}_{\scalebox{0.6}{$\bmat{[c|c|c]\al_0& \dots & \al_p  \\  \be_0 &  \dots & \be_p}$}; \scalebox{0.6}{$\bmat{[c|c|c]\tilde\al_0& \dots & \tilde\al_q\\  \tilde\be_0 &  \dots & \tilde\be_q  };i_0,\dots,i_\ell$}}:E^{(k,\ell)}_{\scalebox{0.6}{$\bmat{[c|c|c]\al_0& \dots & \al_p  \\  \be_0 &  \dots & \be_p}$}; \scalebox{0.6}{$\bmat{[c]  \tilde\al_q  \\  \tilde\be_q };i_0,\dots,i_\ell$}}\to E^{(k,\ell)}_{\scalebox{0.6}{$\bmat{[c|c|c]\al_0& \dots & \al_p  \\  \be_0 &  \dots & \be_p}$}; \scalebox{0.6}{$\bmat{[c] \tilde\al_0 \\  \tilde\be_0 };i_0,\dots,i_\ell$}}\Bigg),
\end{multline}
where the ``$g$''s are associated to any sequence of indices \scalebox{0.8}{$\bmat{[c|c|c] \tilde\al_0& \dots & \tilde\al_q \\ \tilde\be_0 &  \dots & \tilde\be_q  }$} for $q\geq 0$, given by indices from \scalebox{0.8}{$\bmat{[c|c|c]\al_0& \dots & \al_p  \\  \be_0 &  \dots & \be_p}$}. Moreover, these ``$g$''s  satisfy the relations from \eqref{EQU:dg-EZ-relation}.  
Since the simplicies of $x^{(k,\ell)}$ fit together via the simplicial set relations, the above data \eqref{EQU:x^(k,l)-for-Tot-Delta-full} does not depend on the chosen $p$-cell determined by \scalebox{0.8}{$\bmat{[c|c|c]\al_0& \dots & \al_p  \\  \be_0 &  \dots & \be_p}$}, and thus $x^{(k,\ell)}$ is given by the data:
\begin{align}\label{EQU:x^(k,l)-for-Tot-Delta-within-k-ell}
x^{(k, \ell)}=
\Bigg(&
E^{(k,\ell)}_{\scalebox{0.6}{$\bmat{[c] \al   \\  \be};i_0,\dots,i_\ell$}}\to U_{i_0,\dots,i_\ell},
 \nabla^{(k,\ell)}_{\scalebox{0.6}{$\bmat{[c] \al  \\  \be };i_0,\dots,i_\ell$}}, \\
 \nonumber
& g^{(k,\ell)}_{\scalebox{0.6}{$\bmat{[c|c|c] \tilde\al_0& \dots & \tilde\al_q \\  \tilde\be_0 &  \dots & \tilde\be_q  };i_0,\dots,i_\ell$}}:E^{(k,\ell)}_{ \scalebox{0.6}{$\bmat{[c]  \tilde\al_q  \\  \tilde\be_q };i_0,\dots,i_\ell$}}\to E^{(k,\ell)}_{ \scalebox{0.6}{$\bmat{[c] \tilde\al_0 \\  \tilde\be_0 };i_0,\dots,i_\ell$}}
 \Bigg).
\end{align}
For example for $k=2$ and $\ell = 0,1$ some of this data is visualized below, the ``$\nabla$''s are omitted and also suppressed the open set indicies ``$i_0,..., i_\ell$'' are suppressed for better readability:
\begin{equation*}
 \resizebox{12cm}{!}{\begin{tikzpicture}[baseline={([yshift=-.5ex]current bounding box.center)},vertex/.style={anchor=base,
     circle,fill=black!25,minimum size=18pt,inner sep=2pt}]   
 \def\top{6};
 \def\bot{0};
 \def\mid{3};
 \def\dep{2.5};
   \foreach \t in {2,10}
    {
        \path[shorten >=0.2cm,shorten <=0.2cm,<-]  (\t,\top)       edge (\t,\bot);   
        \fill (\t, \bot)circle (1pt);
        \fill (\t, \top)circle (1pt);
        \path[shorten >=0.2cm,shorten <=0.2cm,<-]  (\t,\top)       edge (\t- \dep,\mid);   
        \path[shorten >=0.2cm,shorten <=0.2cm,<-]  (\t- \dep,\mid)       edge (\t,\bot); 
        \fill (\t- \dep,\mid)circle (1pt);
        }
   \foreach \t in {6}
    {
                \path[gray!60, shorten >=0.2cm,shorten <=0.2cm,<-]  (\t,\top)       edge (\t,\bot);   
        \fill (\t, \bot)circle (1pt);
        \fill (\t, \top)circle (1pt);
        \path[shorten >=0.2cm,shorten <=0.2cm,<-]  (\t,\top)       edge (\t- \dep,\mid);   
        \path[shorten >=0.2cm,shorten <=0.2cm,<-]  (\t- \dep,\mid)       edge (\t,\bot); 
        \fill (\t- \dep,\mid)circle (1pt);
        }
        
                \fill[white] (6, \mid)circle (3 pt);
     \path[shorten >=0.2cm,shorten <=0.2cm,<-]  (6,\top)       edge (10,\top);   
     \path[shorten >=0.2cm,shorten <=0.2cm,<-]  (6-\dep,\mid)       edge (10-\dep,\mid);   
     \path[shorten >=0.2cm,shorten <=0.2cm,<-]  (6,\bot)       edge (10,\bot);       
     
\node[above right] at (-1,7.2) {on $U_{i_0}$ for each fixed $i_0$:};
\node[above right] at (5,7.2) {on $U_{i_0,i_1}$ for each fixed $i_0$ and $i_1$:};
     
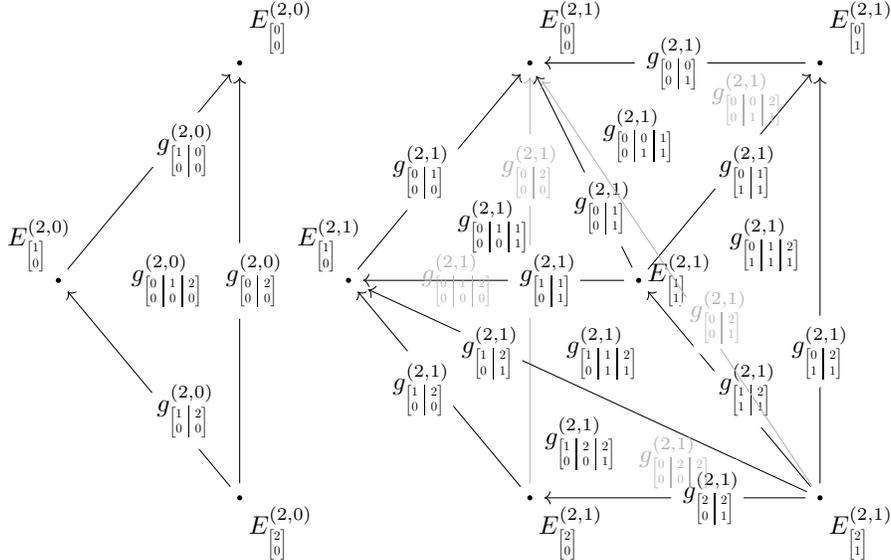
\node[above right] at (2,\top) {$E_{\scalebox{0.5}{$\bmat{[c] 0\\ 0   }$}}^{(2, 0)}$};
\node[above left, xshift=0.3cm] at (2- \dep,\mid) {$E_{\scalebox{0.5}{$\bmat{[c] 1\\ 0   }$}}^{(2, 0)}$};
\node[below right] at (2,\bot) {$E_{\scalebox{0.5}{$\bmat{[c] 2\\ 0   }$}}^{(2, 0)}$};

\node at (1.25,4.8)  {\contour{white}{$g_{\scalebox{0.5}{$\bmat{[c|c] 1&0\\ 0&0   }$}}^{(2, 0)}$}};
\node at (1.25,1.2)  {\contour{white}{$g_{\scalebox{0.5}{$\bmat{[c|c] 1&2\\ 0&0   }$}}^{(2, 0)}$}};
\node at (2.2,3)  {\contour{white}{$g_{\scalebox{0.5}{$\bmat{[c|c] 0&2\\ 0&0   }$}}^{(2, 0)}$}};

\node at (1,3)  {\contour{white}{$g_{\scalebox{0.5}{$\bmat{[c|c|c] 0& 1&2\\ 0 & 0&0   }$}}^{(2, 0)}$}};

     \path[shorten >=0.2cm,shorten <=0.2cm,<-]  (6,\top)       edge (10 - \dep,\mid);   
     \path[shorten >=0.2cm,shorten <=0.3cm,<-]  (6 - \dep,\mid)       edge (10,\bot);

\node[above right] at (6,\top) {$E_{\scalebox{0.5}{$\bmat{[c] 0\\ 0   }$}}^{(2, 1)}$};
\node[above left, xshift=0.3cm] at (6- \dep,\mid) {$E_{\scalebox{0.5}{$\bmat{[c] 1\\ 0   }$}}^{(2, 1)}$};
\node[below right] at (6,\bot) {$E_{\scalebox{0.5}{$\bmat{[c] 2\\ 0   }$}}^{(2, 1)}$};

\node[above right] at (10,\top) {$E_{\scalebox{0.5}{$\bmat{[c] 0\\ 1   }$}}^{(2, 1)}$};
\node[right] at (10- \dep,\mid) {$E_{\scalebox{0.5}{$\bmat{[c] 1\\ 1   }$}}^{(2, 1)}$};
\node[below right] at (10,\bot) {$E_{\scalebox{0.5}{$\bmat{[c] 2\\ 1   }$}}^{(2, 1)}$};

\node[gray!60] at (9,\top -0.5) {$g_{\scalebox{0.5}{$\bmat{[c|c|c] 0 & 0 & 2\\0 & 1&  1   }$}}^{(2, 1)}$};
\node[gray!60] at (5,\mid) {$g_{\scalebox{0.5}{$\bmat{[c|c|c] 0 & 1 & 2\\0 & 0&  0   }$}}^{(2, 1)}$};
\node[gray!60] at (6,\top-1.5)  {\contour{white}{$g_{\scalebox{0.5}{$\bmat{[c|c] 0&2\\  0&0   }$}}^{(2, 1)}$}};

\node at (6.7,\mid-2.25)  {$g_{\scalebox{0.5}{$\bmat{[c|c|c] 1&2&2\\  0&0&1   }$}}^{(2, 1)}$};
\node at (7,\mid-1)  {$g_{\scalebox{0.5}{$\bmat{[c|c|c] 1&1&2\\  0&1&1   }$}}^{(2, 1)}$};
\node at (7.5,\top-1)  {$g_{\scalebox{0.5}{$\bmat{[c|c|c] 0&0&1\\  0&1&1   }$}}^{(2, 1)}$};
\node at (5.5,\mid+0.75)  {$g_{\scalebox{0.5}{$\bmat{[c|c|c] 0&1&1\\  0&0&1   }$}}^{(2, 1)}$};
\node at (9.25,\mid+0.5)  {$g_{\scalebox{0.5}{$\bmat{[c|c|c] 0&1&2\\  1&1&1   }$}}^{(2, 1)}$};

\node at (6- \dep + 1,\top-1.5)  {\contour{white}{$g_{\scalebox{0.5}{$\bmat{[c|c] 0&1\\  0&0   }$}}^{(2, 1)}$}};
\node at (6- \dep + 1,\bot+1.5)  {\contour{white}{$g_{\scalebox{0.5}{$\bmat{[c|c] 1&2\\  0&0   }$}}^{(2, 1)}$}};
\node at (8,\top)  {\contour{white}{$g_{\scalebox{0.5}{$\bmat{[c|c] 0&0\\  0&1   }$}}^{(2, 1)}$}};
\node at (8.5,\bot)  {\contour{white}{$g_{\scalebox{0.5}{$\bmat{[c|c] 2&2\\  0&1   }$}}^{(2, 1)}$}};
\node[gray!60] at (8,\bot +0.5) {$g_{\scalebox{0.5}{$\bmat{[c|c|c] 0& 2 & 2\\ 0 & 0 & 1   }$}}^{(2, 1)}$};

\node at (10,\mid-1)  {\contour{white}{$g_{\scalebox{0.5}{$\bmat{[c|c] 0&2\\  1&1   }$}}^{(2, 1)}$}};
\node at (5.45,\mid-1)  {\contour{white}{$g_{\scalebox{0.5}{$\bmat{[c|c] 1&2\\  0&1   }$}}^{(2, 1)}$}};
\node at (6.25,\mid)  {\contour{white}{$g_{\scalebox{0.5}{$\bmat{[c|c] 1&1\\  0&1   }$}}^{(2, 1)}$}};
\node at (7,\mid+1)  {\contour{white}{$g_{\scalebox{0.5}{$\bmat{[c|c] 0&1\\  0&1   }$}}^{(2, 1)}$}};
\node at (9,\mid+1.5)  {\contour{white}{$g_{\scalebox{0.5}{$\bmat{[c|c] 0&1\\  1&1   }$}}^{(2, 1)}$}};
\node at (9,\mid-1.5)  {\contour{white}{$g_{\scalebox{0.5}{$\bmat{[c|c] 1&2\\  1&1   }$}}^{(2, 1)}$}};
 \path[gray!60, shorten >=0.2cm,shorten <=0.3cm,<-]  (6, \top)       edge (10,\bot);   
\node[gray!60, xshift = 0.1cm] at (8.5,\mid-0.5)  {\contour{white}{$g_{\scalebox{0.5}{$\bmat{[c|c] 0&2\\  0&1   }$}}^{(2, 1)}$}};
\end{tikzpicture}}
\end{equation*}

Now, by the compatibility relations \eqref{EQ:Tot explicit coherence} in $Tot(K)$, the data given by the right-hand side of \eqref{EQU:x^(k,l)-for-Tot-Delta-within-k-ell} is determined by the lowest $\ell$ for which a given set of indices ${\scalebox{0.8}{$\bmat{[c|c|c] \tilde\al_0& \dots & \tilde\al_q \\\  \tilde\be_0 &  \dots & \tilde\be_q }$}}$ can be obtained via a face map. For example, 
$$E^{(k,\ell+1)}_{\scalebox{0.6}{$\bmat{[c]  \al \\  \delta_j(\be) };i_0,\dots,i_{\ell+1}$}}\stackrel{\eqref{EQ:Tot explicit coherence}}=\left(\text{component of } d^j(x^{(k,\ell)})\right)=E^{(k,\ell)}_{\scalebox{0.6}{$\bmat{[c] \al   \\  \be};{i_0,\dots,\widehat{i_j},\dots,i_{\ell+1}}$}} \Big|_{U_{i_0,\dots,i_{\ell+1}}}$$
where $d^j$ acts by pulling back a bundle to a subset (by definitions \ref{DEF:PERF^Q} and \ref{DEF:Perf}), i.e., by restricting the vector bundle to this subset. In particular, $E^{(k,\ell)}_{\scalebox{0.6}{$\bmat{[c] \al   \\  \be};{i_0,\dots,i_{\ell}}$}} =E^{(k,0)}_{\scalebox{0.6}{$\bmat{[c] \al   \\  0};{i_\be}$}} \Big|_{U_{i_0,\dots,i_{\ell}}}$, and similar statements apply to the ``$g$''s.

Thus, the data of a $k$-simplex in $Tot(K)$ is given by (suppressing the tildes):
\begin{enumerate}
\item\label{ITEM:Es}
chain complexes of holomorphic vector bundles $E_{\subsc{\al}{i}}:=E^{(k,0)}_{\scalebox{0.6}{$\bmat{[c] \al  \\ 0 };i$}}\to U_{i} $ with differential $g_{\scalebox{0.6}{$\bmat{[c] \al  \\ 0 };i$}}=g^{(k,0)}_{\scalebox{0.6}{$\bmat{[c] \al  \\ 0 };i$}}$ for any index \scalebox{0.8}{$\bmat{[c] \al\\ 0 }$} on the $(k+1)\times(0+1)$ grid
\item\label{ITEM:nablas}
connections $\nabla_{\subsc{\al}{i}}:=\nabla^{(k,0)}_{\scalebox{0.6}{$\bmat{[c] \al  \\  0 };i$}}$ on $E_{\subsc{\al}{i}}$
\item\label{ITEM:gs}
maps $g_{\scalebox{0.6}{$\bmat{[c|c|c]  \al_0& \dots & \al_q  \\  \be_0 &  \dots & \be_q };i_0,\dots,i_\ell$}}:=g^{(k,\ell)}_{\scalebox{0.6}{$\bmat{[c|c|c]  \al_0& \dots & \al_q  \\  \be_0 &  \dots & \be_q };i_0,\dots,i_\ell$}}:E_{ {\subsc{\al_q}{i_{\beta_q}}} }|_{U_{i_0,\dots,{i_\ell}}}\to E_{{\subsc{\al_0}{i_{\beta_0}}} }|_{U_{i_0,\dots,{i_\ell}}}$ for $\ell\geq1$ and for any $\be$s which include all the indices from $0$ to $\ell$, i.e., for $\{\be_0,\dots,\be_q\}=\{0,\dots,\ell\}$; (this is because if there was a $j\in\{0,\dots, \ell\}$ with $j\notin \{\be_0,\dots,\be_q\}$, then the map $g^{(k,\ell)}_{\scalebox{0.6}{$\bmat{[c|c|c]  \al_0& \dots & \al_q  \\  \be_0 &  \dots & \be_q };i_0,\dots,i_j,\dots,i_\ell$}}$ would according to \eqref{EQ:Tot explicit coherence} just be a restriction of $g^{(k,\ell-1)}_{\scalebox{0.6}{$\bmat{[c|c|c]  \al_0& \dots & \al_q  \\  \ga_0 &  \dots & \ga_q };i_0,\dots,\widehat{i_j},\dots,i_\ell$}}$ to $U_{i_0,\dots, i_j,\dots, i_\ell}$ where $\be_i=\delta_j(\ga_i)$ for all $i$, and so the data could be recovered from the map $g^{(k,\ell-1)}_{\scalebox{0.6}{$\bmat{[c|c|c]  \al_0& \dots & \al_q  \\  \ga_0 &  \dots & \ga_q };i_0,\dots,\widehat{i_j},\dots,i_\ell$}}$ via restriction). Of course, as before, the sequence of indices \scalebox{0.8}{$\bmat{[c|c|c]  \al_0& \dots & \al_q \\  \be_0 &  \dots & \be_q  \\}$} has to come from a non-decreasing set of indices on a $(k+1)\times(\ell+1)$ grid (see appendix \ref{SUBSEC:Tot(sSet(EZ,Ktilde)}). Sometimes we simply write $g_{\scalebox{0.6}{$\bmat{[c|c|c]  \al_0& \dots & \al_q  \\  \be_0 &  \dots & \be_q }$}}$ when the context of the open set $U_{i_0,\dots,i_\ell}$ is clear.

In particular note that:
\begin{itemize}
\item
Using the fact that we land in the maximal Kan subcomplex $\dgN(\perf(U))^\circ$ of $\dgN(\perf(U))$, then for $q=1$, the maps on $1$-cells $g_{\scalebox{0.6}{$\bmat{[c|c]  \al_0& \al_1  \\  \be_0 & \be_1 };i_0,i_1$}}$ are all quasi-isomorphisms.
\item
Finally, these maps satisfy the relations from \eqref{EQU:dg-EZ-relation} on $U_{i_0,\dots,i_\ell}$:
\begin{multline}\label{EQ: MC for k simplex tot Perf}
g_{\scalebox{0.6}{$\bmat{[c|c|c]  \al_0& \dots & \al_q  \\  \be_0 &  \dots & \be_q }$}}\circ g_{\scalebox{0.6}{$\bmat{[c] \al_q  \\ \be_q }$}}
+(-1)^q \cdot g_{\scalebox{0.6}{$\bmat{[c] \al_0  \\ \be_0 }$}}\circ
g_{\scalebox{0.6}{$\bmat{[c|c|c]  \al_0& \dots & \al_q  \\  \be_0 &  \dots & \be_q }$}}
\\
= \sum_{j=1}^{q-1} (-1)^{j-1} g_{\scalebox{0.6}{$\bmat{[c|c|c|c|c]  \al_0& \dots &  \al_j & \dots & \al_q  \\  \be_0 &  \dots & \be_j &  \dots & \be_q }$}}
 +\sum_{j=1}^{q-1} (-1)^{q(j-1)+1} g_{\scalebox{0.6}{$\bmat{[c|c|c]  \al_0& \dots & \al_j  \\  \be_0 &  \dots & \be_j }$}}\circ g_{\scalebox{0.6}{$\bmat{[c|c|c]  \al_j& \dots & \al_q  \\  \be_j &  \dots & \be_q }$}}.
\end{multline}
\end{itemize}
\end{enumerate}
\end{note}

The above note is applied below for the case of $0$-simplicies, in order to relate them to twisting cochains defined by O'Brian, Toledo, and Tong in \cite[Definition 1.3]{OTT1}, which we now briefly review. 
\begin{note}\label{REM:twisting-chains-a-al-OTT}
 Let $(U_i\to X)_{i\in I}$ be a given cover, and let $E^\bu_i\to U_i$ be graded holomorphic vector bundles over $U_i$. Then, according to \cite[Definition 1.3]{OTT1}, $a$ is a \emph{twisting cochain} if $a=\sum_{j\geq 0}a^{j,1-j}$ with $a^{j,1-j}\in C^j(\mc U, Hom^{1-j}(E,E))$, which is given by a collection of bundle morphisms on intersections of open sets, $a^{j,1-j}=\{a_{i_0,\dots, i_j}:E_{i_j}|_{U_{i_0,\dots, i_j}}\to E_{i_0}|_{U_{i_0,\dots, i_j}}\}_{i_0,\dots, i_j\in I}$ satisfying conditions \cite[(1.5)]{OTT1} on each $U_{i_0,\dots,i_q}$:
\begin{equation}\label{EQU:twisting-cochain-condition}
\sum_{j=1}^{q-1} (-1)^j a_{i_0,\dots,\widehat{i_j},\dots,i_q} + \sum_{j=0}^q (-1)^{(1-j)(q-j)} a_{i_0,\dots,i_j}\circ a_{i_j,\dots,i_q}=0
\end{equation}
\end{note}

Note that compared to the data of a $k$-simplex in $\CechSh{\IVB}(X)$ (see note \ref{EX: PERF Tot description} \eqref{ITEM:Es}, \eqref{ITEM:nablas}, and \eqref{ITEM:gs}), there is a priori no chosen connection. A version of $\CechSh{\IVB}(X)$ is also provided then without connection. Recall from \eqref{EQU:DEF-IVB(U)n} that $\IVB(U)_n=sSet(\EZ^n,\dgN(\perf(U))^\circ)$. 
 \begin{definition}
Define $\wt{\perf}:\CMan^{op}\to \dgCat$ by setting $\wt{\perf}(U)$ to be the dg-category of finite chain complexes of holomorphic vector bundles, just as in definition \ref{DEF:Perf}, but with the difference that we do not choose any connection on $E_\bu$. Analogously to $\IVB$ from definition \ref{DEF:PERF}, define  $\wt\IVB:\CMan^{op}\to \sSet$ by setting $\wt\IVB(U)_n:=\sSet(\EZ^n, \dgN(\wt\perf(U))^\circ)$. 

For a \v{C}ech cover $(U_\bu\to X)$, note \ref{EX: PERF Tot description} can be repeated to obtain an explicit description of $Tot(\wt\IVB(\CechNerve \mathcal{U}_{\bu}))$. Indeed the data of a $k$-simplex of $Tot(\wt\IVB(\CechNerve \mathcal{U}_{\bu}))$ is given by the data of chain complexes of holomorphic vector bundles $E_{\al;i}$ as in \eqref{ITEM:Es} together with maps $g_{\scalebox{0.6}{$\bmat{[c|c|c]  \al_0& \dots & \al_q  \\  \be_0 &  \dots & \be_q };i_0,\dots,i_\ell$}}$ as in \eqref{ITEM:gs}, but \emph{without} any connections as stated in \eqref{ITEM:nablas}.
 \end{definition}
 
The following lemma relates the above definition to the one with connections.
 \begin{lemma}\label{LEM:IVB=IVBtilde}
 The dg-functor $\perf\to \wt\perf$ that forgets the connection induces a map of simplicial presheaves $\IVB\to \wt\IVB$, which after applying the \v{C}ech sheafification (definition \ref{DEF: cech sheafify}) yields an isomorphism of simplicial sets $\CechSh{\IVB}(X)\stackrel{\cong}\to\CechSh{\wt\IVB}(X)$.
 \end{lemma}
 \begin{proof}
For a fixed cover $(U_\bu\to X)$, the forgetful map $Tot(\IVB(\CechNerve \mathcal{U}_{\bu}))\to Tot(\wt\IVB(\CechNerve \mathcal{U}_{\bu}))$ forgets the information of the connections as stated in \eqref{ITEM:nablas} in note \ref{EX: PERF Tot description}. Taking colimit over covers, this descends to a well-defined map $\CechSh{\IVB}(X)\to \CechSh{\wt\IVB}(X)$ which is surjective, since every complex manifold has a (Stein) open cover so that for every open set of the cover, there exists a connection on the corresponding bundles.

It remains to check injectivity. Assume that two $k$-simplicies $x,x'\in \CechSh{\IVB}(X)_k$ are mapped, respectively, to $\wt x, \wt{x}'\in\CechSh{\wt\IVB}(X)_k$ by forgetting the connections, and that these are equal, i.e., $\wt x= \wt{x}'$. This means that there is a zig-zag of refinements and extensions with respect to the colimit over covers which connects $\wt x$ and $\wt x'$ in $\CechSh{\wt\IVB}(X)_k$. Since every $k$-simplex in $\CechSh{\wt\IVB}(X)$ has a refinement which is in the image of $\CechSh{\IVB}(X)$ under the forgetful functor, (i.e., it has a choice of connections on the bundles for each open set,) it is enough to consider the case where $\wt x $ and $\wt x'$ are both refinements of $\wt y\in \CechSh{\wt\IVB}(X)_k$, where $\wt y$ may \emph{not} be in the image of the forgetful functor. In order to prove injectivity, it is enough to show that there exists a $\wt z$ for which both $\wt x $ and $\wt x'$ are refinements, and which is in the image of the forgetful functor, so that taking a preimage $z$ of $\wt z$ shows that $x$ and $x'$ are equal in $\CechSh{\IVB}(X)_k$. To this end, note, that, if $x$ and $x'$ are represented on fixed covers $U_\bu$ and $U'_\bu$, respectively. Then we define $\wt z$ represented on the cover $U_\bu\sqcup U'_\bu$ as follows. To define the bundle data \eqref{ITEM:Es} for $\wt z$, if $V$ is an open set in the cover $U_\bu$ or $U'_\bu$ pick the bundle for that open set from $\wt x$ or $\wt x'$, respectively, which we note to be equal to bundles from $\wt y$ appropriately restricted. To define the maps ``$g$'' from \eqref{ITEM:gs} for $\wt z$, if $V_1,\dots, V_\ell$ are open sets from $U_\bu\sqcup U'_\bu$, we have bundles over $V_i$ coming from the data $\wt y$, and so we take the maps of bundles as provided by $\wt y$. Note, that, $\wt x$ and $\wt x'$ both extend $\wt z$, and, moreover, $\wt z$ is in the image of the forgetful functor by the extension $z$ of $x$ and $x'$, since there are connections on each of the bundles coming from the data \eqref{ITEM:nablas} provided by $x$ and $x'$.
 \end{proof}
 
 With this definition, the main theorem of this section is stated below. 
\begin{theorem}\label{THM: OTT are vertices}
The equivalence classes of \cite{OTT1} of twisting cochains inject into the vertices of ${\CechSh{\IVB}}(X)$.
\begin{proof}
By the previous lemma \ref{LEM:IVB=IVBtilde}, we may forget about the connections, and simply inject twisting cochains into vertices of ${\CechSh{\wt\IVB}}(X)$. By note \ref{REM:twisting-chains-a-al-OTT}, a twisting cochain on a cover $(U_i\to X)_{i\in I}$ with holomorphic vector bundles $E^\bu_i\to U_i$ is given by a collection $a=\{a_{i_0,\dots, i_j}\}_{i_0,\dots,i_j\in I, j\geq 0}$ satisfying \eqref{EQU:twisting-cochain-condition}. To this, we assign the data of a $0$-simplex in $\CechSh{\IVB}(X)$ as stated in \eqref{ITEM:Es} and \eqref{ITEM:gs} from page \pageref{ITEM:Es} as follows. First, the $E_{0;i}\to U_i$ from \eqref{ITEM:Es} are just the given $E_i$. As for the ``$g$''s in  \eqref{ITEM:Es} and \eqref{ITEM:gs}, define
\begin{equation}\label{EQU:twist->IVB_0}
g^{(k,\ell)}_{\scalebox{0.6}{$\bmat{[c|c|c]  0& \dots & 0  \\  \be_0 &  \dots & \be_q };i_0,\dots,i_\ell$}}:=a_{i_{\be_0},\dots,i_{\be_q}}.
\end{equation}
Note that the twisting cochain equations \eqref{EQU:twisting-cochain-condition} imply \eqref{EQ: MC for k simplex tot Perf}. Moreover, the equivalence of twisting cochains is generated by refinements and extensions (see \cite[p. 232 above Proposition 1.10]{OTT1}), which identifies the corresponding infinity vector bundles (due to the colimit in \eqref{EQ: perf sheafify example 1}).

To check injectivity, we give a map in the opposite direction, which is a left-inverse to the above map.
Explicitly, for a $0$-simplex in $\CechSh{\wt\IVB}(X)$ represented by a cover $(U_i\to X)_i$ and bundles $E^\bu_i$ with maps ``$g$'' as in \eqref{ITEM:Es} and \ref{ITEM:gs}, we define the twisting cochain
\begin{equation}\label{EQU:IVB_0->twist}
a_{i_{0},\dots,i_{j}}:=g^{(k,\ell)}_{\scalebox{0.6}{$\bmat{[c|c|c|c]  0&0& \dots & 0  \\  0 &1&  \dots & j };i_0,\dots,i_j$}},
\end{equation}
which preserves the twisting cochain equations \eqref{EQU:twisting-cochain-condition} due to \eqref{EQ: MC for k simplex tot Perf}. The colimit construction implies equivalence of twisting cochains. The composition of  these two constructions, which maps twisting cochains to $\CechSh{\wt\IVB}(X)_0$ via \eqref{EQU:twist->IVB_0} and then back to twisting cochains via  \eqref{EQU:IVB_0->twist}, is the identity on twisting cochains.

As a final remark, we note that there are different (non-equivalent) choices for a left-inverse other than \eqref{EQU:IVB_0->twist}. In fact, equation \eqref{EQU:twist->IVB_0} assigns the \emph{same} homotopy $a_{j_0,\dots,j_q}$ to any 
$g_{\scalebox{0.6}{$\bmat{[c|c|c]  \al_0& \dots & \al_q  \\  \be_0 &  \dots & \be_q };i_0,\dots,i_\ell$}}$ with $i_{\be_0}=j_0,\dots, i_{\be_q}=j_q$, while in ${\CechSh{\wt\IVB}}(X)_0$ these $g_{\scalebox{0.6}{$\bmat{[c|c|c]  \al_0& \dots & \al_q  \\  \be_0 &  \dots & \be_q };i_0,\dots,i_\ell$}}$ may generally be different. Therefore, any choice (consistent within the Maurer-Cartan equation \eqref{EQ: MC for k simplex tot Perf}) may thus be used as a left-inverse for \eqref{EQU:twist->IVB_0}.
\end{proof}
\end{theorem}

To end this subsection, consider the restriction of the simplicial presheaf $\IVB$ to the one which only utilizes chain complexes of vector bundles whose homology is concentrated in degree zero. Below we show that the associated simplicial presheaf contains (after sheafification) all of the data of isomorphism classes of coherent sheaves in its vertices. 

\begin{note}\label{RMK: CohSh Homology}
For the reader's convenience, we review here a construction from section 2 of \cite{TT78a}. Let $X \in \CMan$ and $a_{\bullet}$ be a twisting cochain for a cover $(U_\bu\to X)$ with holomorphic vector bundles $E^\bu_\bu$ (see \cite{OTT1} or note \ref{REM:twisting-chains-a-al-OTT} above). Consider the locally defined sheaf of $\mathcal{O}_X$-modules, $\mathcal{H}_i:= H_{\bu}\left( \Gamma\left( E_i\right), a_i\right)$, given by the homology of sections of  $E_i^{\bullet}$ with differential $a_i$ over $U_i$. Since each $a_{i,j}$ gives a quasi-isomorphism on the level of complexes, there is an induced isomorphism of sheaves on homology $a_{i,j} : \restr{\mathcal{H}_j}{U_{i,j}} \xrightarrow{\sim} \restr{\mathcal{H}_i}{U_{i,j}}$. Taking the colimit\footnote{Here we mean the concrete set-theoretic colimit given by a coproduct of $\mathcal{H}_i$ and then mod out by the equivalence generated by $a_{i,j}$ on $U_{i,j}$.} of the $\mathcal{H}_i$ over the diagram induced by these $a_{ij}$ produces a sheaf on $X$ which we will call \emph{the homology sheaf} and denote by $\mathcal{H}$. This construction further produces a map\footnote{Which importantly is \emph{not} coming from a map of complexes or even graded modules.} of simplicial presheaves
\begin{equation}
\CechSh{\IVB} \xrightarrow{\mc H} \mathcal{N} (\text{Sh}\mathcal{O}^\bu)
\end{equation}
where $\mc N$ denotes the nerve, and $\text{Sh}\mathcal{O}^\bu$ is the category of sheaves of graded $\mathcal{O}_X$-modules (without differential) with morphisms given by isomorphisms. The relevance of this construction to coherent sheaves is recorded in the following definition and proposition.
\end{note}

\begin{definition}\label{DEF: CohSh}
The simplicial presheaf $\CohSh \hookrightarrow \IVB$ is the sub-simplicial presheaf defined by considering the full sub-presheaf of dg-categories, $\perf_{\text{coh}} \hookrightarrow \perf$ utilizing only chain complexes of bundles whose homology is concentrated in degree zero and then taking $\CohSh(X)_n:=sSet(\EZ^n,\dgN(\perf_{\text{coh}}(U))^\circ) $.
\end{definition}

\begin{lemma}\label{LEM: CohSh Stein}
Given a manifold $M$ and a coherent sheaf $\mathcal{F}$ there exists an open cover by relatively compact Stein open submanifolds on which $\mathcal{F}$ is locally resolved by a chain complex of vector bundles. 
\begin{proof}
$M$ admits a cover $\{U_{i}\}_{i \in I}$ by Stein open subsets. For each Stein submanifold $U_i$, it admits an open cover by relatively compact open sets $\{V_{i,j}\}_{i\in I, j \in J_i}$. Now for each relatively compact open submanifold $V_{i,j}$ we cover it one final step further by open Stein sets $W_{i,j,k}$. As each $W_{i,j,k}$ is a subset of a relatively compact open Stein manifold $U_i$, then by \cite[Theorem 7.2.6]{Fi} $\mathcal{F}$ admits a resolution by vector bundles on $W_{i,j,k}$. 
\end{proof}
\end{lemma}

\begin{proposition}\label{PROP:CohSh justify}
The set of isomorphism classes of coherent sheaves on $X$ is in bijective correspondence with the connected components of $\CechSh{\CohSh} (X)$. 
\begin{proof}
Recall the map $\mc H : \CechSh{\IVB} (X) \to \mathcal{N}(\ShgO{X})$ from note \ref{RMK: CohSh Homology}. But since $\CohSh$ requires the local chain complex's homology to be concentrated in degree zero, the map's image lands in $\mathcal{N} (\ShO{X}) \hookrightarrow \mathcal{N} (\ShgO{X})$, where $\mc N (\ShO{X})$ is the nerve of the category of sheaves of $\mc O_X$-modules (concentrated in degree $0$). Since the image of our map is precisely an ${\mc O}_X$ which satisfies the properties of a coherent sheaf, then the map factors  through the nerve of the groupoid of coherent sheaves with isomorphisms, $\mc H: \CechSh{\CohSh}(X) \to \mathcal{N}( \CohShO{X})\hookrightarrow \mathcal{N} (\ShO{X})$ which in turn is well defined as a map which sends connected components of $\CechSh{\CohSh}$ to connected components of $\mathcal{N} (\CohShO{X})$, i.e., precisely the isomorphism classes of $\CohShO{X}$. 

To observe injectivity, we consider the image of two vertices $x, y \in \CechSh{\CohSh}(X)_0$, represented by cocycle data on some common refinement by a Stein cover, $( U_{\bu}\to X)$, whose images $\mc H (x) , \mc H(y) \in \mathcal{N} (\CohShO{X})$ are connected by an edge. In particular, this means that the global homology sheaves for $x$ and $y$ are isomorphic as $\mc O$-modules. In order to construct an edge $z \in \CechSh{\CohSh}(X)_1$ connecting $x$ and $y$, we first need local quasi-isomorphisms connecting the local resolutions for the chain complexes of bundles $x$ and $y$, respectively. These maps are given by recalling that these complexes over a Stein space are projective resolutions \cite[Cor. 2.4.5]{Fo} and so maps on homology induce chain maps between the complexes \cite[Theorem 4.1]{HS}. So far, these quasi-isomorphisms produce the edge data for $z$ on $U_i$, and the $1$-skeleton of the edge data for $z$ on higher intersections. To move up to the $2$-skeleton, say on $U_{i,j}$, we see that we now have two quasi-isomorphisms between the complexes for $x$ and $y$: one restricted from the quasi-iso over $U_i$ and the other from $U_j$. Again appealing to \cite[Theorem 4.1]{HS} we now know these two quasi-isomorphisms are chain-homotopic and this provides all of the data for $z$ on $U_i$'s, $U_{i,j}$'s, and the $2$-skeleton of the data on higher intersections. Now by \cite[Lemma 1.6]{OTT1} and the ensuing discussion there, one uses an inductive argument for how our higher homotopies of $z$ would be constructed to satisfy the Maurer-Cartan equation and since their constructions include into ours (see our proof of theorem \ref{THM: OTT are vertices}), one indeed can construct an edge $z$ connecting $x$ and $y$ to prove injectivity.

For surjectivity, applying Lemma \ref{LEM: CohSh Stein} and then following \cite[Propsoition 2.4]{TT78a}, for a coherent sheaf $\mathcal{F}$ there exists a Stein open cover $\left(U_{i} \hookrightarrow X \right)_{i \in I}$ so that we can choose a twisting cochain class in $\CechSh{\CohSh}(X)_0$ by locally/projectively resolving the coherent sheaf by a complex of vector bundles, coherent on intersections $U_{i,j}$ up to quasi-isomorphisms, and further coherent on $U_{i_0, \dots, i_p}$ by higher homotopies which again exist by virtue of \cite[Lemma 1.6]{OTT1} and the discussion which follows it. It follows that the map $\mathcal{H}$ is surjective on connected components since in the proof of theorem \ref{THM: OTT are vertices}, we show how their constructions include into ours.
\end{proof}
\end{proposition}

\subsection{\v{C}ech sheafification of the Chern map $\CH$}\label{SEC: cech CH}
This section continues the study of the \v{C}ech sheafified Chern character map $\CechSh{\CH}:\CechSh{\IVB} \to \CechSh{\OM}$, (where $ \CechSh{{\F}}(X) = \colim\limits_{( {U}_{\bu} \to X) \in \check{S}} Tot\left( {\F}(\CechNerve {U}_{\bu} ) \right)$ was defined in equation \eqref{EQ: cech sheafify}). In theorem \ref{THM: OTT are vertices} twisting cochains a la \cite{OTT1} were already interpreted as $0$-simplicies of $\CechSh{\IVB}$. Next, in note \ref{NOTE:Tot(Chern)} $\CechSh{\OM}$ is explicitly described as well as the map $\CechSh{\CH}$ for the case of $0$-simplicies. Comparing the formulas for the \v{C}ech sheafified Chern character map $\CechSh{\CH}$ with the Chern character map from \cite{OTT1} for a coherent sheaf (which is reviewed in \ref{NOTE:Chern-from-OTT}), shows, that these are given by precisely the same formulas. This result is stated in theorem \ref{THM:CH(IVB)=CH(TwCoch)}.

The following note reviews  $Tot( {\OM}(\CechNerve {U}_{\bu}))$.

\begin{note}\label{NOTE:Tot(Omega(NU))}
Fix a \v{C}ech cover $(U_\bu\to X)$. Then $Tot(\OM ( \CN   U_{\bu} ) )$ is the totalization of the cosimplicial simplicial set $\OM \left( \CN   U_{\bu} \right) =\DKSet(\Ohol^\bu(\CN U)\ul)$. Recall from note \ref{REM:DK-def}, that the $n$-simplices of Dold-Kan applied to the chain complex $\Ohol^\bu(V)\ul$ for some open set $V$, are decorations of the standard $n$-simplex, i.e., they assign to each $\ell$-simplex, polynomials $a\in \Ohol^\bu(V)\ul$ of total degree $-\ell$, 
\begin{equation}\label{EQU:a-for-OM(V)}
a=\begin{cases} 
\sum\limits_{j=0}^{\infty} a^{2j} \cdot u^{\frac \ell 2 +j},  &\text{when $\ell$ is even} \\
\sum\limits_{j=0}^{\infty} a^{2j+1} \cdot u^{\frac{\ell+1} 2+j},  &\text{when $\ell$ is odd}
\end{cases}
\end{equation}
where $a^{p} \in \Ohol^{p}(V)$. The condition  \eqref{EQU:DK-condition} imposed for these decorations is that the alternating sum of the faces of a $\ell$-simplex agree with applying the chain complex's differential to the data of the $\ell$-simplex:
\[0=d_{C}(a)= \sum_{j=0}^\ell (-1)^j d_j(a),\]
where $C$ is the complex $C=\Ohol^\bu(V)\ul$ with zero differential $d_C=0$, (see definition \ref{DEF:OM}).

Now, from appendix \ref{SUBSEC:totalization} and \ref{REM:Tot(K) simplices}, $0$-simplicies of the totalization, $\Tot( \OM (   \CN \mc U_{\bu}) )_0$, consist of coherent decorations of the standard $n$-simplex by data coming from  $\OM (   \CN \mc U_{n})$:
\begin{itemize}
\item on each $U_i$, a $0$-simplex in $\DKSet(\Ohol^\bu(U_i)\ul)$ , i.e., a polynomial $a_i$ as in \eqref{EQU:a-for-OM(V)} with $\ell=0$: $a_i =\sum\limits_{j=0}^\infty a^{2j}_i\cdot u^j$,
\item on each $U_{i_0,i_1}$, a $1$-simplex in $\DKSet(\Ohol^\bu(U_{i_0,i_1})\ul)$ , i.e., a polynomial $a_{i_0,i_1}$ as in \eqref{EQU:a-for-OM(V)} with $\ell=1$: $a_{i_0,i_1} =\sum\limits_{j=0}^\infty a^{2j+1}_{i_0,i_1}\cdot u^{j+1}$,
\item on each $U_{i_0,\dots, i_\ell}$, an $\ell$-simplex in $\DKSet(\Ohol^\bu(U_{i_0,\dots, i_\ell})\ul)$ , i.e., a polynomial $a_{i_0,\dots, i_\ell}$ as in \eqref{EQU:a-for-OM(V)}.
\end{itemize}
These polynomials satisfying the conditions:
\[ 0 = \sum\limits_{j=0}^\ell (-1)^j d_j \left(a_{i_0 \cdots, i_\ell}  \right) = \sum\limits_{j=0}^\ell \restr{a_{i_0 \dots, \widehat{i_j}, \dots, i_\ell}}{U_{i_0, \dots, i_\ell}}, \]
where the last equality follows from equation \eqref{EQU:coherence} and example \ref{EXA:Tot-open-cover}.
\end{note}

Recall from \cite{G} that the \emph{Hodge cohomology} $\bigoplus\limits_{p,q} H^p \left( X, \Omega^q\right)$ is given by a sum over the $p$-th sheaf cohomology of the sheaf of holomorphic $q$ forms (see also ``Hodge theory'' or ``Hodge decomposition'' \cite{Fro}). In \cite[section 4]{OTT1} the Chern character is defined as an element in $\bigoplus\limits_{k} H^k \left( X, \Omega^k\right)$. Below we see how our $\CechSh{\OM}$ relates to the Hodge cohomology.

\begin{proposition}\label{PROP: pi_0 OM Hodge}
The set of connected components of $\CechSh{\OM}(X)$ forms a ring which is isomorphic to the even part of the Hodge cohomology ring, 
\[ \pi_0 \left( \CechSh{\OM}(X) \right) \simeq \bigoplus\limits_{\substack{p,q\\ p+q \text{ even}}} H^p \left( X, \Omega^q\right).\]
\begin{proof}
The proof follows first from a direct observation that the vertices of $\Tot \left( \OM \left( \CN {U_{\bu}}\right) \right)$ are precisely (since the differentials are all zero) a direct sum of \v{C}ech $\ell$-cocycles of holomorphic forms (even degree forms for $\ell$ even and odd degree forms for $\ell$ odd), and then from the observation that edges in  $\Tot \left( \OM \left( \CN {U_{\bu}}\right) \right)$ correspond to \v{C}ech coboundaries.
\end{proof}
\end{proposition}
We next illustrate our sheafified Chern map $\CechSh{\CH}$.

\begin{note}\label{NOTE:Tot(Chern)}
Consider a \v{C}ech cover $(U_\bu\to X)$, and a vertex in $Tot \left( \IVB \left( \CechNerve \mathcal{U} \right) \right)_0$ as provided by note \ref{EX: PERF Tot description}, i.e., the data of holomorphic bundles $E_{0;i}$ with differentials $d=g_{\scalebox{0.6}{$\bmat{[c] 0  \\ 0 };i$}}$ from \eqref{ITEM:Es}, connections $\nabla_{0;i}$ from \eqref{ITEM:nablas}, and maps $g_{\scalebox{0.6}{$\bmat{[c|c|c]  0& \dots & 0  \\  \be_0 &  \dots & \be_q };i_0,\dots,i_\ell$}}$ from \eqref{ITEM:gs}. Then our sheafified Chern character map, $\CechSh{\IVB} \xrightarrow{\CechSh{\CH} } \CechSh{\OM}$, simply applies the Chern character, $\CH:\IVB\to \OM$ from Definition \ref{DEF:CH:PERF-to-OM}, locally to the data in our vertex by allowing the indices from that definition to be given by the indices of the open cover. To clarify, the vertex above gets mapped to the following vertex in $Tot\left( \OM \left( \CechNerve \mathcal{U} \right)  \right)_0$:
\begin{itemize}
\item
on each $U_i$, assign the Euler characteristic of $E_{0;i}$, denoted $\chi ( E_{0;i})\cdot u^0\in \Ohol^0(U_i)\ul$. 
\item
on each $U_{i_0, i_1}$, using $g=\{g_{\scalebox{0.6}{$\bmat{[c|c|c]  0& \dots & 0  \\  \be_0 &  \dots & \be_q };i_0,i_1$}}\}_{(\be_0,\dots, \be_q)\in \EZ^1}$, assign the monomial $Tr_g(\nabla (d+g))_{(0,1)}\cdot u\in \Ohol^1(U_{i_0, i_1})\ul$, and restrict the Euler characteristic above on the vertices (cf. definition \ref{DEF:CH:PERF-to-OM}),
\begin{equation*}
\begin{tikzpicture}[scale=0.5]
\node (E0) at (0, 0) {}; \fill (E0) circle (4pt) node[above] {$\restr{\chi(E_{0;i_0})}{U_{i_0, i_1}}$};
\node (E1) at (12, 0) {}; \fill (E1) circle (4pt) node[above] {$\restr{\chi(E_{0;i_1})}{U_{i_0, i_1}}$};
\draw [<-] (E0) -- node[above] {$ Tr_g( \nabla (d+g))_{(0,1)}\cdot u$}(E1);
\end{tikzpicture}
\end{equation*}
\item
For each $U_{i_0,i_1, \dots, i_{\ell}}$, using $g=\{g_{\scalebox{0.6}{$\bmat{[c|c|c]  0& \dots & 0  \\  \be_0 &  \dots & \be_q };i_0,i_1,\dots, i_\ell$}}\}_{(\be_0,\dots, \be_q)\in \EZ^\ell}$, assign the monomial 
\begin{equation}\label{EQ: Tr g nabla g} 
Tr_g((\nabla (d+g))^\ell)_{(0,1,\dots, \ell)}\cdot \frac{u^\ell}{\ell!} \in \Ohol^{\ell}(U_{i_0, i_1,\dots , i_{\ell}})\ul,\end{equation}
to the top cell and each face is assigned appropriate restrictions of the monomials defined for lower intersections.
\end{itemize}
\end{note}

The above formula is now compared to the one provided by O'Brian, Toledo, and Tong for the Chern character map of a coherent sheaf.
\begin{note}\label{NOTE:Chern-from-OTT}
In \cite{OTT1}, the construction of characteristic classes for coherent sheaves follows four steps:
\begin{enumerate}[(i)]
\item Given a coherent sheaf, a twisting cochain $a$ is constructed using \cite[below Lemma 1.6]{OTT1}. This construction is well-defined with respect to equivalences of twisting cochains; cf. \cite[Proposition 1.10]{OTT1}. 
\item Connection data is chosen for $a$ so that we obtain a twisting cochain with holomorphic connection data; cf. ; cf. \cite[above Proposition 4.4]{OTT1}. 
\item The Atiyah class is represented by the class $\nabla a$ in \cite[Proposition 4.4]{OTT1}.
\item The Chern character is defined \cite[above Proposition 4.5]{OTT1} using the trace map $\tau_a$ to be given by
\begin{equation}\label{EQU:Chern-char-a-la-OTT}
\text{ch}:=\sum_{k\geq 0} \text{ch}_k:=\sum_{k\geq 0} \frac 1 {k!} \tau_a((\nabla a)^k)
\end{equation}
Note that the trace map $\tau_a$ from \cite[above Proposition 3.2]{OTT1} is defined in the same way as our trace map $Tr_g$ in \ref{DEF:trace_g}.
\end{enumerate}
\end{note}

Comparing the formulas \eqref{EQ: Tr g nabla g} and \eqref{EQU:Chern-char-a-la-OTT} for the Chern character, these involve the same trace terms, and so we obtain the following theorem.
\begin{theorem}\label{THM:CH(IVB)=CH(TwCoch)}
For a given coherent sheaf, the formula for the Chern character \eqref{EQU:Chern-char-a-la-OTT} from \cite{OTT1} is given by the terms in the formula \eqref{EQ: Tr g nabla g} of the Chern character map 
\begin{equation}\label{EQ: pi_0 Chern}
\left\{\substack{ \text{Iso. classes of}\\  \text{coherent sheaves}}\right\}   \simeq \pi_0\left( \CechSh{\CohSh}\right) \xrightarrow{\pi_0\left(\CechSh{\CH}\right)}  \pi_0\left( \CechSh{\OM}\right)\simeq \bigoplus_{\substack{p,q\\p+q \text{ even}}} H^p\left( \Omega^q\right)
\end{equation}
 applied to the corresponding twisting cochain interpreted (by theorem \ref{THM: OTT are vertices}) as a $0$-simplex in $\CechSh{\CohSh}$.
\begin{proof}
A twisting cochain $a$ defines the Maurer Cartan element via \eqref{EQU:twist->IVB_0}. With this, the terms in the traces in \eqref{EQ: Tr g nabla g} and \eqref{EQU:Chern-char-a-la-OTT} coincide. (We note that the additional factor $u^\ell$ in equation \eqref{EQU:Chern-char-a-la-OTT} does not add any extra information, as the power $\ell$ is precisely the ``\v{C}ech degree'' given by the number of intersections in $U_{i_0,\dots, i_\ell}$.). Finally, the left and right isomorphisms in \eqref{EQ: pi_0 Chern} are given by propositions \ref{PROP:CohSh justify} and \ref{PROP: pi_0 OM Hodge}, respectively.
\end{proof}
\end{theorem}

Note, in particular, that our sheafified $\CechSh{\CH}$ provides not only a Chern character to coherent sheaves but also provides invariants for morphisms and higher homotopies between coherent sheaves.

\section{The Induced Map on Classifying Stacks}\label{SEC: Local Proj}
In this section we show that the previously considered sheafified Chern Map (definition \ref{DEF: sheafified Chern}) is a map of simplicial sheaves when we restrict $\CechSh{\IVB}$ (see definition \ref{DEF: cech sheafify} and example \ref{EX: PERF Tot description}) to the sub-simplicial sheaf $\nCechSh{\IVB}$ which considers complexes of vector bundles of a fixed length, $n$ (see definition \ref{DEF: restricted IVB}). Moreover, each of these simplicial pre-sheaves contains a sub-simplicial presheaf which considers complexes, $\CechSh{\CohSh}$ and $\nCechSh{\CohSh}$ respectively (see definition \ref{DEF: CohSh}), whose homology is concentrated in degree zero, yielding the  commutative diagram, 

\begin{equation*}
\adjustbox{scale=0.8}{\begin{tikzcd}[row sep = small, column sep=small]
\CohSh_{\le n} \arrow[hook]{rrr} \arrow[hook]{ddd} \arrow{rd} & & &  \CohSh \arrow[hook]{ddd} \arrow{dl} \\ 
& \nCechSh{\CohSh} \arrow[hook]{r} \arrow[hook]{d}  & \CechSh{\CohSh} \arrow[hook]{d} &  \\
& \nCechSh{\IVB} \arrow[hook]{r}& \CechSh{\IVB} &\\
\IVB_{\le n} \arrow{ur} \arrow[hook]{rrr}  & & & \IVB \arrow{ul}& 
\end{tikzcd} }
\end{equation*}
(see proposition \ref{PROP:CohSh justify} for a justification of our notation ``$\CohSh$''). As such we offer in theorem \ref{THM: OTT Chern Sheaves} an upgrade on the statement of theorem \ref{THM:CH(IVB)=CH(TwCoch)} to a statement about sheaves.

\subsection{Sheaves in the local projective model structure}
This section's main goal is to sort out which of the (maps of) presheaves in this paper are in fact (maps of) sheaves.

Given the Verdier site (see \cite[Section 9]{DHI}) of complex manifolds and holomorphic maps, $\CMan$, the category of simplicial presheaves $sPre(\CMan)$ has multiples model structures. One particular choice is the (global) projective model structure whose weak equivalences are object-wise weak equivalences of simplicial sets and whose fibrations are object-wise fibrations of simplicial sets \cite[Theorem 1.5]{Bl}. Further this model structure forms a (proper simplicial cellular) simplicial model category when we use the simplicial mapping space $\underline{sPre}(X,Y)_n:= sPre(X \otimes \Delta^n, Y)$. After localizing this simplicial category over the class of maps induced by hypercovers, we further obtain the local projective (proper simplicial cellular) model structure $sPre(\CMan)_{proj, loc}$ \cite[Theorem 1.6]{Bl}. The relevant criteria in this structure for us is that an object in $sPre(\CMan)_{proj,loc}$ is fibrant if it is fibrant in the projective model structure and satisfies descent with respect to any hypercover \cite{DHI}. Such an object is referred to below as a (hyper-)sheaf. 

In presenting a classifying stack (i.e. classifying hyper-sheaf) for coherent sheaves, one could produce a simplicial presheaf, $\F \in sPre(\CMan)$, and prove (at the very least) that for any manifold $X \in \CMan$, the set of equivalence classes of coherent sheaves coincides with the connected components of the derived mapping space, $\mathbb{R} Hom (X, \F)$. Since we are working with the local projective simpicial model category of simplicial presheaves this mapping space can be computed by cofibrantly approximating $X$ with $\tilde{X} \to X$ (which in this case is the identity since $X$ is representable and thus cofibrant), fibrantly approximating $\F$ by $\F \to \hat{\F}$, and defining the right derived mapping space (i.e. the homotopy function complex from \cite[section 17]{Hirs}) as the simplicial mapping space on the replacements: 
\begin{equation}\label{EQ: RHom} \mathbb{R} Hom (X, \F) := \underline{sPre}(\tilde{X}, \hat{\F}) =  \underline{sPre}(X, \hat{\F}) .\end{equation}
Thus $\hat{\F}$ would provide a more concrete description of this classifying stack and any map of simplicial presheaves $\F \to \OM$ provides cohomological invariants by inducing a map between fibrant replacements $\hat{\F} \to \hat{\OM}$; offering more explicit, cocycle-level cohomological invariants. 

It is not immediate that our \v{C}ech sheafification computes the fibrant replacement. Below we first show that if $\F$ is already a hyper sheaf then $\CechSh{\F}$ is again a hyper sheaf, even though this result is not used in this paper. 
\begin{proposition}
If $\F$ is a hyper sheaf then $\CechSh{\F}$ is a hyperesheaf and the natural map $\F \to \CechSh{\F}$ is an object-wise weak equivalence.
\begin{proof}
By construction, we have already shown in the proof of proposition \ref{Prop: Cech Fibrant} that \v{C}ech sheafification preserves object-wise fibrancy without any assumptions on the homotopy type of $\F$. To see that there is an object-wise weak equivalence, we compute: 
\begin{align*}
 \CechSh{\F}(X ) &= \colim\limits_{(\mathcal{W} \to X) \in S} \underline{sPre} ( \mathcal{W}, \F)\\
\intertext{but since $\F$ already satisfies descent,}
&\xleftarrow{\sim}  \colim\limits_{(\mathcal{W} \to X) \in S} \underline{sPre} ( X, \F) \xleftarrow{\sim} \underline{sPre} ( X, \F) = \F(X).
\end{align*}
Now to show that the \v{C}ech sheafification preserves hyperdescent we choose a hyper cover $\mathcal{U} \to X$ and argue that the nautral map $\underline{sPre}(X, \CechSh{\F}) \to\underline{sPre}(\mathcal{U}, \CechSh{\F})$ is a weak equivalence of simplicial sets. On the one hand we have, 
\[\underline{sPre}\left(X, \CechSh{\F} \right) = \CechSh{\F}(X ) \xleftarrow{\sim} \F(X).\]
while on the other hand we have, 
\begin{align*}
\underline{sPre}\left( \mathcal{U}, \CechSh{\F} \right) &\xrightarrow{\sim}  \underline{sPre}\left( \hocolim\limits_{i \in \Delta}\coprod_{i, \alpha_i} U_{i, \alpha_i}, \CechSh{\F} \right) \\
&=\holim\limits_{i \in \Delta}  \prod_{i, \alpha_i} \underline{sPre}\left(U_{i, \alpha_i}, \CechSh{\F} \right) =\holim\limits_{i \in \Delta}  \prod_{i, \alpha_i}\CechSh{\F} (U_{i, \alpha_i})  \\
& \xleftarrow{\sim} \holim\limits_{i \in \Delta}  \prod_{i, \alpha_i}\F (U_{i, \alpha_i})  = \underline{sPre}\left( \mathcal{U}, \F \right)
\intertext{and since $\F$ already satisfies descent, again we have,}
&\xleftarrow{\sim} \F(X).
\end{align*}
After repeated application of the 2-out-of-3 property for weak equivalences, we see that $\CechSh{\F}$ satisfies descent as well.
\end{proof}
\end{proposition}

Under a modest boundedness condition on a simplicial presheaf $\F$ which takes values in Kan complexes, its \v{C}ech sheafification (Definition \ref{DEF: cech sheafify}) is a sheaf; this result is key to the rest of this paper.

\begin{proposition}\label{Prop: Cech Fibrant}
Let $\F \in sPre(\CMan)$ be a projectively fibrant simplicial presheaf whose homotopy groups are all trivial above level $n$. Then $\CechSh{\F}$ is a fibrant approximation of $\F$ in the local projective model structure of simplicial presheaves on complex manifolds. 
\begin{proof}
Given a projectively fibrant simplicial presheaf $\F \in sPre(\CMan)$ we can consider its fibrant replacement in the local projective model structure $\F \xrightarrow{\sim}  \F' \in sPre(\CMan)_{loc}$. In \cite[Remark 6.2.2.12] {Lurie} we see that in general we can compute this fibrant replacement on a test manifold $X \in \CMan$ with the \emph{hyper sheafification} of $\F$, written $\F^{\dagger}$, by taking a homotopy colimit of the simplicial mapping space $\underline{sPre}(\mathcal{U}, \F)$ over all hypercovers $(\mathcal{U} \to X)$. Below, as it is standard, we identify the manifold $X$ with its representable simplicial presheaf, i.e., with the functor $Y\mapsto \CMan(Y,X)$, post composed by the functor which sends sets to simplicially constant simplicial sets. Thus, if $S$ denotes the category of all hypercovers,
\[ \F^{\dagger}(X) := \hocolim_{(\mathcal{U} \to X) \in {S}} \underline{sPre}\left(\mathcal{U}, \F \right).\]
More formal references for this fact include \cite[Example 3.4.9]{AnSu} and \cite[Proposition 6.6]{Low}. We can now follow a series of steps to rewrite the above sheafification up to weak equivalence: 
\begin{align}
\F^{\dagger}(X) := &\hocolim_{(\mathcal{U} \to X) \in {S}} \underline{sPre}\left(\mathcal{U}, \F \right) = \hocolim_{( \mc U \to X) \in {S}} \underline{sPre}\left(\hocolim_{i \in \Del} {\mc U_{i}}, \F \right) \label{EQ: basal hypercover} \\
 \intertext{pulling the homotopy colimit out as a homotopy limit, and then using the fact that $\F$ is of bounded homotopy type so $ \F \xrightarrow{\sim} {\bcosk{n}}\F$ with both of these projectively fibrant,}
= &\hocolim_{( \mc U \to X) \in {S}} \holim_{i \in \Del} \underline{sPre}\left(\mc U_i, \F \right) \xrightarrow{\sim} \hocolim_{( \mc U \to X) \in {S}}\holim_{i \in \Del} \underline{sPre}\left(\mc U_i, {\bcosk{n} }\F \right) \nonumber\\ 
\intertext{now using the skeleton-coskeleton adjunction and then that we can change the indexing set of hypercovers to also be $n$-skeletal,} 
\xrightarrow{\sim} &  \hocolim_{( \mc U \to X) \in {S}} \holim_{i \in \Del}  \underline{sPre}\left({\bsk{n}}\mc U_i,  \F \right) =  \hocolim_{( \mc U \to X) \in {S_{\le n}}} \holim_{i \in \Del}  \underline{sPre}\left({\bsk{n}}\mc U_i,  \F \right)\nonumber\\\intertext{now since \v{C}ech covers are cofinal in bounded hypercovers on a paracompact manifold  \cite[Proposition 3.6.63]{Sc}, denoting by $\check{S}$ the category of \v{C}ech covers, this is weakly equivalent to} 
\nonumber\xrightarrow{\sim} &  \hocolim_{(\CechNerve U_{\bullet} \to X) \in {\check{S}}} \holim_{i \in \Del} \underline{sPre}\left({\bsk{n}}\CechNerve U_i,  \F \right)  = \hocolim_{(\CechNerve U_{\bullet} \to X) \in {\check{S}}}  \holim_{i \in \Del} \underline{sPre}\left(\CechNerve U_i,  {\bcosk{n}}\F \right) \\
\nonumber\xleftarrow{\sim} &  \hocolim_{(\CechNerve U_{\bullet} \to X) \in {\check{S}}}  \holim_{i \in \Del} \underline{sPre}\left(\CechNerve U_i,  \F \right) \\
\intertext{next we apply a simplicial Yoneda lemma and then use the fact that $\Tot$ computes $\holim$ when the cosimplicial simplicial set is Reedy fibrant \cite[Theorem 18.7.4]{Hirs},} 
\nonumber= & \hocolim_{(\CechNerve U_{\bullet} \to X) \in {\check{S}}}  \holim_i  {\prod\limits_{\alpha_0 \dots \alpha_i}}\F(U_{\alpha_0 \dots \alpha_i}) \xrightarrow{\sim}  \hocolim_{(\CechNerve U_{\bullet} \to X) \in {\check{S}}}   Tot \left( \F(\CechNerve U_\bullet) \right) \\
\intertext{and finally we use the fact that the colimit over \v{C}ech covers is a filtered colimit to compute $\hocolim$ with a $\colim$,}
\nonumber \xrightarrow{\sim} & \colim_{(\CechNerve U_{\bullet} \to X) \in {\check{S}}}  Tot \left( \F(\CechNerve U_\bullet) \right)= \CechSh{\F}(X)  
\end{align} 
By proposition \ref{PROP: Cech Sh F Kan}, $\CechSh{\F}$ is already globally projectively fibrant (i.e. takes values in Kan complexes). Now it remains to show that $\CechSh{\F}$ satisfies hyper-descent. Given a hypercover, $\mathcal{U} \to X$, we use the commutative square, 
\begin{equation} \begin{tikzcd}
\underline{sPre}(X, \F^{\dagger})= \F^{\dagger}(X)\arrow[d] \arrow[r] & \underline{sPre}(\mathcal{U}, \F^{\dagger}) \arrow[d]\\
 \underline{sPre}(X, \CechSh{\F} ) = \CechSh{\F}(X) \arrow[r] & \underline{sPre}(\mathcal{U},\CechSh{\F}) 
\end{tikzcd}\end{equation}
where the equalities are given by Yoneda. Since $\F^{\dagger}$ satisfies descent, the top horizontal map is a weak equivalence by definition of descent. The left vertical map was proven to be an equivalence above. With $\mathcal{U}$ projectively cofibrant it follows that the simplicial mapping spaces preserve the weak equivalence $\F^{\dagger} \xrightarrow{\sim} \CechSh{\F}$ between projectively fibrant objects and so the right vertical map is a weak equivalence. Thus by the two-out-of-three property afforded to our model category we have shown that the bottom horizontal map is a weak equivalence. Since we have shown that $\CechSh{\F}$ is projectively fibrant, satisfies hyper descent, and that $\F \xrightarrow{\sim} \CechSh{\F}$, then $\CechSh{\F}$ is a fibrant replacement of $\F$ in the local projective model structure.  
\end{proof}
\end{proposition}

\begin{lemma}\label{LEM: Restricted height sheaves of complexes are 1-types}
Let $Ch( \mathcal{A})$ be the $dg$-category of non-positively graded chain complexes over some additive category, $\mathcal{A}$, where the hom-complex $Ch^{\bu}\left( E, E'\right))$ consists of chain maps and (higher) chain homotopies from $E$ to $E'$, and let $\mathcal{Q} \hookrightarrow Ch^{\le 0}$ be a full subcategory which only considers complexes of height at most $m$, for some fixed $m \in \mathbb{N}$. Then the simplicial set $\dgN (Q) \simeq \bcosk{m+1} \dgN (Q)$ is $(m+1)$-coskeletal. 
\begin{proof}
For any two objects in $Q$, and for an integer $k> m+1$, we have $\mathcal{Q}^{k}\left( E, E'\right)) = 0$ due to the restricted height of all complexes in our dg-category. Thus the only way to decorate $k > m+1$ simplex is to have the boundary data all satisfy the condition $\hat{\delta} g + g \cdot g = 0$ and then uniquely assign a $0$-homotopy to the $(m+1)$-simplex. But recall that whenever each decorated boundary simplex has a unique filler, this means the simplicial set is isomorphic to it's coskeleton, so in our case we have $\dgN (Q) \simeq \bcosk{m+1} \dgN (Q)$ as required.
\end{proof}
\end{lemma}
 \begin{definition}\label{DEF: restricted IVB}
Define $\perf_{\le n}:\CMan^{op}\to \dgCat$ by setting $\perf_{\le n}(U)$ to be the dg-category of finite chain complexes of holomorphic vector bundles just as in definition \ref{DEF:Perf}, but with the difference that we require the complexes are trivial above level $n$ . Analogously to $\IVB$ from definition \ref{DEF:PERF}, we then define  $\IVB_{\le n}:\CMan^{op}\to \sSet$ by setting $\IVB_{\le n}(U)_n:=\sSet(\EZ^n, \dgN(\perf_{\le n}(U))^\circ)$. 
 \end{definition}
\begin{corollary}\label{COR: Sheaf restricted IVB}
The fibrant replacement of $\IVB_{\le n}$ in the local projective model structure can be computed by its \v{C}ech sheafification, $\IVB_{\le n} \xrightarrow{\sim} \CechSh{\IVB_{\le n}}$. 
\begin{proof}
By construction, $\IVB_{\le n}$ is still (globally) projectively fibrant, while combining lemma \ref{LEM: Restricted height sheaves of complexes are 1-types} and proposition \ref{PROP:X=whX} gives us that $\IVB_{\le n}$ is (globally) a homotopy-$(n+1)$ type. 
\end{proof}
\end{corollary}

\begin{lemma}\label{LEM: Resolution Sheaves are 1-types}
Let $Ch( \mathcal{A})$ be the $dg$-category of non-positively graded chain complexes over some additive category, $\mathcal{A}$, where the hom-complex $Ch^{\bu}\left( E, E'\right))$ consists of chain maps and (higher) chain homotopies from $E$ to $E'$, and let $\mathcal{Q} \hookrightarrow Ch^{\le 0}$ be a full subcategory which only considers complexes with homology concentrated in degree zero. Then the (Kan replacement of the) simplicial set $\dgN (Q)$ is a $1$-type. 
\begin{proof}
If necessary, first replace $\dgN (Q)$ with its maximal Kan subcomplex which only uses quasi-isomoprhisms on edges. We will prove that $\pi_n \left( \dgN (Q) \right)$ is trivial for $n \ge 2$. A class in $\pi_n$ consists of an $n$-simplex in $\dgN (Q)$ whose entire boundary is in the image of a single vertex. Thus the vertices are given by the same chain complex, $E_0 = E, \ldots, E_n = E$, the quasi-isomorphisms on the edges are the identity maps, and any homotopy decorating a $k<n$ face is the zero homotopy. By the definition of $\dgN (Q)$ this data satisfies the condition $\hat{\delta}(g) + Dg +  g \cdot g=0$ using the notation of definition \ref{DEF: MC}. Since in this case $\hat{\delta}(g) + g \cdot g$ is an alternating sum of compositions of $0$-homotopies and/or identity maps, one can show that the above condition reduces to $Dg = 0$. However, since $E$ is a complex whose homology is concentrated in degree zero and $g \in Q^{1-n}(E,E)$ with $n \ge 2$, then $g$ is exact. From here we can fill this $n$-sphere with a higher homotopy and kill the class representing $g$ in $\pi_n$.\end{proof}
\end{lemma}
By a similar argument for corollary \ref{COR: Sheaf restricted IVB} we can use the above lemma to see that $\CohSh$ is a $1$-type and thus $\CechSh{\CohSh}$ is a sheaf, but without needing to further restrict the height of any chain complexes.
\begin{corollary}\label{COR: CohSh 1 type}
The simplicial presheaf $\CohSh$ is a $1$-type and its fibrant replacement in the local projective model structure can be computed by its \v{C}ech sheafification, $\CohSh \xrightarrow{\sim} \CechSh{\CohSh}$. 
\end{corollary}

\begin{remark}
Now that under the right circumstances the \v{C}ech sheafification can act as a fibrant replacement functor, we can briefly present a different argument for lemma \ref{LEM:IVB=IVBtilde} which makes use of equivalences being preserved under the various constructions we use to pass from the $\dgCat$-valued presheaf, $\Perf$, to the simplicial presheaf $\CechSh{\IVB}$. The main idea used in the proof for lemma \ref{LEM:IVB=IVBtilde} is that for a complex manifold $X$, and a point $x \in X$, there exists an (Stein) open subset $x \in U \subset X$ on which we have an equivalence of dg-categories, $\Perf(U) \xrightarrow{\sim} \wt{\Perf}(U)$, where the tilde again means we forget connection data. Since the dg-nerve construction preserves (weak) equivalences, we then obtain an equivalence of simplicial sets, $\IVB(U) \xrightarrow{\sim} \wt{\IVB}(U)$. We claim this then says that we have a weak equivalence for each stalk $\IVB_x \xrightarrow{\sim} \wt{\IVB}_x$ and thus a local weak equivalence of simplicial presheaves a la Jardine, $\IVB \xrightarrow{\sim} \wt{\IVB}$. The local weak equivalences for the local projective model structure happen to coincide with those of Jardine and thus we obtain a weak equivalence in the local projective model structure which is necessarily preserved under our (\v{C}ech) fibrant replacement functor if we restrict appropriately:  $\nCechSh{\IVB} \xrightarrow{\sim} \nCechSh{\wt{\IVB}}$.
\end{remark}
\begin{remark}
At this point, we'd like to take stock and summarize the relationships amongst some of different constructions involving $\IVB$. By the functoriality of our constructions, we obtain two commutative cubes of simplicial presheaves which actually fit together to form a commutative hypercube via the inclusion $\CohSh \hookrightarrow \IVB$: 

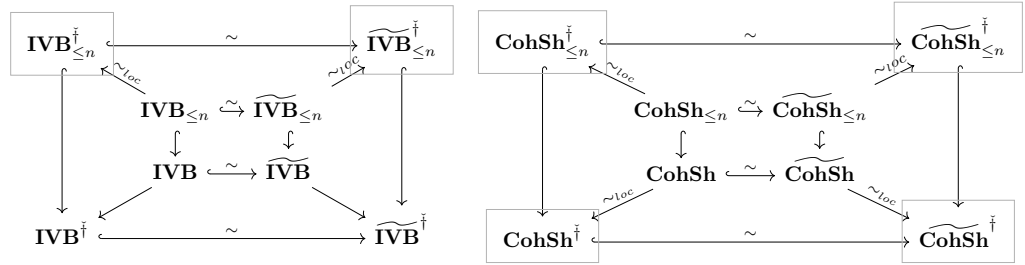
\begin{equation*}
\adjustbox{scale=0.8}{\begin{tikzcd}[row sep = small, column sep=small, /tikz/execute at end picture={
    \node (large) [rectangle, draw, gray!60, fit=(A)] {};
     \node (large) [rectangle, draw, gray!60, fit=(B)] {};
  }]
|[alias=A]|\nCechSh{\IVB} \arrow[hook]{rrr}{\sim} \arrow[hook]{ddd} & & &  |[alias=B]|\nCechSh{\wt{\IVB}} \arrow[hook]{ddd}\\ 
& \IVB_{\le n} \arrow[hook]{r}{\sim} \arrow[hook]{d} \arrow[sloped]{lu}{\sim_{loc}} & \wt{\IVB}_{\le n} \arrow[hook]{d} \arrow[sloped]{ru}{\sim_loc}&  \\
& \IVB \arrow[hook]{r}{\sim} \arrow{ld}& \wt{\IVB} \arrow{rd}&\\
\CechSh{\IVB} \arrow[hook]{rrr}{\sim}  & & & \CechSh{\wt{\IVB}}& 
\end{tikzcd} }\adjustbox{scale=0.8}{\begin{tikzcd}[row sep=small, column sep=small, /tikz/execute at end picture={
    \node (large) [rectangle, draw, gray!60, fit=(A)] {};
     \node (large) [rectangle, draw, gray!60, fit=(B)] {};
    \node (large) [rectangle, draw, gray!60, fit=(C)] {};
  \node (large) [rectangle, draw, gray!60, fit=(D)] {};
  }]
|[alias=A]|\nCechSh{\CohSh} \arrow[hook]{rrr}{\sim} \arrow[hook]{ddd} & & &  |[alias=B]|\nCechSh{\wt{\CohSh}} \arrow[hook]{ddd}\\ 
& \CohSh_{\le n} \arrow[hook]{r}{\sim} \arrow[hook]{d} \arrow[sloped]{lu}{\sim_{loc}} & \wt{\CohSh}_{\le n} \arrow[hook]{d} \arrow[sloped]{ru}{\sim_loc}&  \\
& \CohSh \arrow[hook]{r}{\sim} \arrow[sloped]{ld}{\sim_{loc}}& \wt{\CohSh} \arrow[sloped]{rd}{\sim_{loc}}&\\
|[alias=C]|\CechSh{\CohSh} \arrow[hook]{rrr}{\sim}  & & & |[alias=D]|\CechSh{\wt{\CohSh}}& 
\end{tikzcd} }
\end{equation*}
where the hypersheaves are highlighted with boxes, we used ``$\sim$'' to denote a global projective (i.e. object-wise) weak equivalence, and ``$\sim_{loc}$'' to denote a local projective weak equivalence. Recall that the global weak equivalences are preserved in the local model structure. 
\end{remark}

Recall that in proposition \ref{PROP:CohSh justify} we showed that $\CechSh{\CohSh}$ stands a chance of classifying coherent sheaves since the correspondence is bijective on connected components. We know, however that $\mathcal{N}(\ShgO{X})$ is a 1-type and so if we knew that $\CechSh{\CohSh}$ was also a $1$-type then it would only remain to prove the correspondence on $\pi_1$. 
\begin{lemma}\label{LEM: sheafify n type}
Given $\F \in sPre (\CMan)$ which is object-wise an $n$-type (i.e. $\F \xrightarrow{\sim} \bcosk{n}F$ for some $n$), $\CechSh{\F}$ is again an $n$-type. 
\begin{proof}
We begin by noting that if $\F \xrightarrow{\sim} \bcosk{n} \F$, then we have, 
\begin{align*}
\CechSh{\F}(X) &=  \colim_{( {U}_{\bu} \to X) \in \check{S}} \underline{sPre}\left(\CechNerve {U}_{\bu}  , \F\right) \xrightarrow{\sim} \colim_{( {U}_{\bu} \to X) \in \check{S}} \underline{sPre}\left(\CechNerve {U}_{\bu}  , \bcosk{n} \F\right)  \\
&\xrightarrow{\sim}  \colim_{( {U}_{\bu} \to X) \in \check{S}}\Tot \left( \bcosk{n} \F \left( \CechNerve {U}_{\bu}  \right) \right) \xrightarrow{\sim}  \colim_{( {U}_{\bu} \to X) \in \check{S}}  \bcosk{n} \Tot \left( \F \left( \CechNerve {U}_{\bu}  \right) \right) \\
\intertext{where we used that $\Tot$ computes the homotopy limit in this case and then we commuted the right adjoint $\bcosk{n}$ across this concrete limit,}
&\xleftarrow{\sim}   \colim_{( {U}_{\bu} \to X) \in \check{S}} \bcosk{n}\underline{sPre}\left(\CechNerve {U}_{\bu}  , \F\right) .
\end{align*}
 While we would love to commute this coskeleton across the colimit, we must proceed differently. Recall that filtered colimits commute with finite limits, and since each homotopy group can be written as a finite limit, then we have for $m >n$, 
\begin{align*} \pi_m \left( \CechSh{\F}(X)  \right) &\simeq \pi_m \left(  \colim_{( {U}_{\bu} \to X) \in \check{S}} \bcosk{n}\underline{sPre}\left(\CechNerve {U}_{\bu}  , \F\right) \right) \\
&\simeq  \colim_{( {U}_{\bu} \to X) \in \check{S}} \pi_m \left(  \bcosk{n}\underline{sPre}\left(\CechNerve {U}_{\bu}  , \F\right) \right) =  \colim_{( {U}_{\bu} \to X) \in \check{S}}0 = 0.
\end{align*}
\end{proof}
\end{lemma}

\begin{theorem}\label{THM:IVBStack}
The simplicial presheaf $\CohSh$ is a classifying pre-stack for coherent sheaves.
\begin{proof}
Recall from \cite[Section 17]{Hirs} that the derived mapping space $\mathbb{R} Hom (A, B)$ in a simplicial model category $\mc C$ can be computed by considering the simplicial mapping space $\underline{\mc C} ( \tilde{A}, B' )$ where we use the cofibrant replacement $\tilde{A} \xrightarrow{\sim}$ of $A$ and the fibrant replacement $B \xrightarrow{\sim} B'$ of $B$. Then since corollary \ref{COR: CohSh 1 type} tells us that $\CohSh$ is a $1$-type whose (local projective) fibrant replacement is given by its \v{C}ech sheafification, we can compute the (local projective) derived mapping space from a manifold $X\in  \CMan$ (via its cofibrant representable presheaf) into $\CohSh$ as: 
\[ \mathbb{R} Hom (X, \CohSh) := \underline{sPre}(\tilde{X}, \CohSh') \simeq \underline{sPre}(X, \CechSh{\CohSh}) =\CechSh{\CohSh}(X).\]
After combining proposition \ref{PROP:CohSh justify} and the above lemma, it remains to be shown that the map $\mc H : \CechSh{\CohSh} (X) \to \mathcal{N}(\ShO{X})$ is an isomorphism of fundamental groups. The ideas used to prove this fact are analogous to those of proposition \ref{PROP:CohSh justify} but we will summarize them here for ease of reading. Given a vertex $\mathcal{E} = \left( U_{\bu}, E_{\bu}, g_{\bu}\right) \in \CechSh{\CohSh}(X)_0$ and the coherent sheaf $\mathcal{F} := \mathcal{H}\left( \mathcal{E} \right) \in  \mathcal{N} (\CohShO{X})_0$ we want to prove that there is an isomorphims of based homotopy groups, $\pi_1 \left(  \CechSh{\CohSh}(X), \mc E \right) \xrightarrow{\pi_1 \left( \mathcal{H} \right)} \pi_1 \left(  \mathcal{N} (\CohShO{X}), \mc F \right)$. To prove injectivity, if two loops in $\CechSh{\CohSh}(X)_1$ , $a_{\bu}, b_{\bu}:  \mc E \to \mc E$ have connected images in  $\mathcal{N} (\CohShO{X})$, then by definition of the nerve of a groupoid, we have a commutative square of isomorphisms in $\CohShO{X}$ where all four corners are the coherent sheaf $\mc F$. Lifting this commutative square to a homotopy in $\CechSh{\CohSh} (X)_1$ once again uses the fact that chain maps which induce the same map on homology are homotopic  \cite[Theorem 4.1]{HS} (and then the discussion near \cite[Lemma 1.6]{OTT1}). To prove surjectivity, a loop  $f: \mc F = \mc H ( \mc E) \to \mc F  =\mc H ( \mc E)$ in $ \mathcal{N} (\CohShO{X})_1$ is lifted to a loop in $\CechSh{\CohSh}(X)$ on $\mc E$ by using the fact that an isomorphism on homology lifts to a quasi-isomorphism of chain complexes \cite[Theorem 4.1]{HS} (and then, again, the discussion near \cite[Lemma 1.6]{OTT1}).
\end{proof}
\end{theorem}

If we knew that $\OM$ somehow used complexes of bounded height then our \v{C}ech sheafified Chern map from definition \ref{DEF: sheafified Chern} could be seen to restrict to a map of sheaves $\CechSh{\CH}:\nCechSh{\IVB} \to \CechSh{\OM}$ out of infinity vector bundles of bounded complex height. One way to resolve this is by restricting our site as recorded below 
\begin{proposition}\label{Prop: restricted Cech Ch sheaf map}
On the site $\CMan_{\le n}$ of complex manifolds of dimension at most $n$, the \v{C}ech sheafification of the restricted Chern map, 
\[ \CechSh{\Ch}: \nCechSh{\IVB} \to \CechSh{\OM} \]
is a map of hyper sheaves.
\begin{proof}
By corollary \ref{COR: Sheaf restricted IVB} $\nCechSh{\IVB}$ is already a sheaf. Now that we have restricted the site to $\CMan_{\le n}$, then $\OM$ only makes use of chain complexes of length at most $n$ and so it is coskeletal and by proposition \ref{Prop: Cech Fibrant} its sheafification is a hyper sheaf.
\end{proof}
\end{proposition}

By different application of the same ideas above, we end with an upgrade on theorem \ref{THM:CH(IVB)=CH(TwCoch)}:
\begin{theorem}\label{THM: OTT Chern Sheaves}
On the site $\CMan_{\le n}$ of complex manifolds of dimension at most $n$, the \v{C}ech sheafification of the Chern map restricted to coherent sheaves, 
\[ \CechSh{\Ch}: \CechSh{\CohSh} \to \CechSh{\OM} \]
is a map of hyper sheaves which restricts on $\pi_0$ to the Chern character \eqref{EQU:Chern-char-a-la-OTT} from \cite{OTT1}.
\begin{proof}
By theorem \ref{THM:IVBStack} $\CechSh{\CohSh}$ is already a sheaf. Now that we have restricted the site to $\CMan_{\le n}$, then $\OM$ only makes use of chain complexes of length at most $n$ and so it is coskeletal and by proposition \ref{Prop: Cech Fibrant} its sheafification is a hyper sheaf. The fact that on $\pi_0$ it recovers the Chern map from \cite{OTT1} was already recorded in theorem \ref{THM:CH(IVB)=CH(TwCoch)}.
\end{proof}
\end{theorem}

\begin{remark}\label{REM: Generalized Chern Ch}
For an arbitrary stack (i.e. hyper sheaf), $\F$, recall as in \eqref{EQ: RHom} that the right derived mapping space 
\[  \mathbb{R} Hom (\F, \G) := \underline{sPre}(\tilde{\F}, \hat{\G}) \]
for a simplicial model category can be computed by taking the simplicial mapping space between a cofibrant replacement of $\F$ and a fibrant replacement of $\G$. Letting ${\bf G_1} = \IVB$ and ${\bf G_2}= \OM$, proposition \ref{Prop: restricted Cech Ch sheaf map} says that our pre-sheafified Chern map $\CH:\IVB\to \OM$ from definition \ref{DEF:CH:PERF-to-OM} induces a map of fibrant (ignoring the restrictions of sites and homotopy types for the moment) replacements $\CechSh{\Ch}: \CechSh{\IVB} \to \CechSh{\OM}$, and thus a map of right derived mapping spaces: 
\begin{equation}\label{EQ: generalized Chern-ch} \mathbb{R} Hom (\F, \IVB) = \underline{sPre}(\tilde{\F},  \CechSh{\IVB}) \xrightarrow{\CechSh{\Ch}}  \underline{sPre}(\tilde{\F},  \CechSh{\OM})=: \mathbb{R} Hom (\F, \OM) .\end{equation}
When $\F = X$ is the representable simplicial presheaf for a complex manifold, then the above is explicitly calculated using note \ref{NOTE:Tot(Chern)}. However, \eqref{EQ: generalized Chern-ch} suggests a reasonable definition for a \emph{generalized Chern character map}. In a sequel to this paper, we will study this map for the case when a Lie group $G$ acts on the complex manifold $X$ and $\F_n = X \times G ^{\times n}$ (see \cite[Definition 5.1]{GMTZ}), extending this paper to the equivariant setting.
\end{remark}

\appendix

\section{A Weak Equivalence $\sSet(\EZ^\bu,K)\to\sSet(\Delta^\bu,K)$}

In this appendix, we prove proposition \ref{PROP:X=whX}.

\begin{proposition}\label{PROP:X=whX}
If $K$ is a Kan complex, then there exists a weak equivalence $F^\sharp:\sSet(\EZ^\bu,K)\to \sSet(\Delta^\bu,K)$.
\end{proposition}

In order to define $F^\sharp$, we first establish some notation. Recall from example \ref{EXA:Delta} that $\Delta^n$ is the simplicial set whose $k$-simplicies are non-decreasing sequences $(i_0\leq\dots\leq i_k)$ with $i_0,\dots, i_k\in \{0,\dots,n\}$, and recall from example \ref{EXA:EZ} that $\EZ^n$ is the simplicial set whose $k$-simplicies are any sequences $(i_0,\dots, i_k)$ with $i_0,\dots, i_k\in \{0,\dots,n\}$. Both $\Delta^n$ and $\EZ^n$ have face maps $d_j$ given by removing the $j$th index $i_j$, and degeneracy maps $s_j$ given by repeating the $j$th index $i_j$. Furthermore both $\Delta^\bu$ and $\EZ^\bu$ are cosimplicial simplicial sets so that for $\phi:[n]\to [m]$ in $\Del$ we get an induced map of $\phi_\bu:\both^n_\bu\to\both^m_\bu$ via $\phi_k:\both^n_k\to\both^m_k$, $\phi_k(i_0,\dots,i_k)=(\phi(i_0),\dots, \phi(i_k))$, where $\both^\bu$ is either $\Delta^\bu$ or $\EZ^\bu$. Thus, there is an induced map of cosimplicial simpicial sets $F^\bu:\Delta^\bu\to \EZ^\bu$, $(i_0\leq\dots\leq i_k)\mapsto (i_0,\dots, i_k)$. For any simplicial set $X$, both $X=\sSet(\Delta^\bu,X)$ and $\ov X:=\sSet(\EZ^\bu,X)$ are simplicial sets, and there is an induced map $F^\sharp:\ov X\to X$ by pre-composition with $F$. 

Our first step towards proving proposition \ref{PROP:X=whX} is to show that $\ov K$ is also a Kan complex. 
\begin{proposition}\label{PROP:whK-is-Kan}
If $K$ is a Kan complex, then $\ov K$ is a Kan complex.
\end{proposition}

To begin with, here is a useful lemma.
\begin{lemma}\label{LEM:whK-boundary-compat}
A map $c:\Delta^n\to\ov K=\sSet(\EZ^\bu,K)$ is determined by the element $\overline c=c(0 \leq\dots\leq n):\EZ^n\to K$. Then $\delta_i(c)=c\circ \delta_i:\Delta^{n-1}\cong \delta_i(\Delta^{n-1})\subset \Delta^n \stackrel c\to \ov K$ is determined by $\delta_i(\overline c)=c\circ\delta_i:\EZ^{n-1}\cong \delta_i(\EZ^{n-1})\subset \EZ^n \stackrel c\to K$.
\end{lemma}
\begin{proof}
Note that $\delta_i (\Delta^{n-1})\subset \Delta^n$ are sequences that do not include $i$, which are generated by the $n-1$ simplex $(0\leq\dots \leq i-1\leq i+1\leq \dots n)=d_i(0\leq \dots\leq n)\in \Delta^n_{n-1}$. Thus $\delta_i(c)$ is determined by the image of the simplex $d_i(0\leq \dots\leq n)$. Now $c(d_i(0\leq \dots\leq n))=d_i(c(0\leq \dots\leq n))=d_i(\overline c)=c\circ\delta_i$.
\end{proof}
\begin{proof}[Proof of proposition \ref{PROP:whK-is-Kan}]
Denote by $\Lambda^n_i:=\bigcup_{j\neq i} \delta_j \Delta^{n-1}$ the $i$th horn of $\Delta^n$, which is a sub-simplicial set of $\Delta^n$. Similarly, denote by $\wh \Lambda^n_i :=\bigcup_{j\neq i} \delta_j \wh\Delta^{n-1}$ the $i$th horn of $\EZ^n$, which is a sub-simplicial set of $\EZ^n$. As noted before, a simplicial set map $\Delta^n\to \ov K$ is the same as an element $\ov K_n$, i.e., a simplicial set map $\EZ^n\to K$. Similarly, a simplicial set map $\Lambda^n_i\to \ov K$ is given by $n$ maps $\delta_j \Delta^{n-1}\to \ov K$, i.e., $n$ maps $\EZ^{n-1}\to K$ (cf. lemma \ref{LEM:whK-boundary-compat}), which are compatible at their common boundary, i.e., whose induced common boundary maps $\EZ^{n-2}\to K$ coincide, and thus this is the same as a simplicial set map $\wh\Lambda^n_i\to K$. Thus, the Kan condition for $\ov K$ (left side of \eqref{EQU:Kan-for-whK}) becomes equivalent to lifting a horn $\wh\Lambda^n_i\to X$ to a map $\wh\Delta^n\to X$ (right side of \eqref{EQU:Kan-for-whK}):
\begin{equation}\label{EQU:Kan-for-whK}
\xymatrix{\Lambda^n_i \ar[r] \ar[d] & \ov K \ar[d]\\
\Delta^n \ar[r] \ar@{-->}[ru] & {*}}
\hspace{.6cm}
\begin{matrix} \\ \\ \iff \end{matrix}
\hspace{.6cm}
\xymatrix{\wh\Lambda^n_i \ar[r] \ar[d] &  K \ar[d]\\
\wh\Delta^n \ar[r]  \ar@{-->}[ru] & {*}}
\end{equation}
Since $K$ is a Kan complex, we have such a lift if $\wh\Lambda^n_i\to\wh\Delta^n$ is an trivial cofibration, i.e., if this map is injective and a weak equivalence. Clearly, $\wh\Lambda^n_i\to\EZ^n$ is injective, and the weak equivalence follows since both $\wh\Lambda^n_i$ and $\wh\Delta^n$ are contractible, i.e., they have zero homotopy groups. First, it is well-known that $EG$ for any group $G$ is contractible, since it has an extra degeneracy $s_{-1}(g_0,\dots, g_k)=(e,g_0,\dots, g_k)$, see, for example, \cite[lemma III.5.1 and example III.5.2]{GJ}. Thus, $\EZ^n=E\Z_{n+1}$ is contractible, and, from the explicit extra degeneracy, we can see that it preserves $\wh \Lambda^n_0$. Thus, $\wh \Lambda^n_0$ is contractible as well. Now, there is a $\Z_{n+1}$-action on $E\Z_{n+1}$, which, in particular, can be used to map $\wh \Lambda^n_0$ isomorphically to any other $\wh \Lambda^n_i$, showing that indeed all $\wh \Lambda^n_i$ are contractible. (Or, alternatively, one obtains that the extra degeneracy $s_{-1}(i_0,\dots, i_k)=(i,i_0,\dots, i_k)$ of $\EZ^n$ preserves $\wh\Lambda^n_i$.)
\end{proof}

In order to prove proposition \ref{PROP:X=whX}, we need one more ingredient. Denote by $\wh \Theta^n:=(\bigcup_{\text{all }j}\delta_j\wh\Delta^{n-1})\cup \Delta^n$ the subsimplicial set of $\EZ^n$ generated by \emph{all} $\wh\Delta^{n-1}$ boundary components, together with $\Delta^n\cong F^n(\Delta^n)\subset \EZ^n$.
\begin{lemma}
$\wh \Theta^n$ is contractible.
\end{lemma}
\begin{proof}
For a subset $A\subset \{0,\dots, n\}$, denote by $\Ups^n_A:=(\bigcup_{j\in A}\delta_j\wh\Delta^{n-1})\cup \Delta^n$ the subsimplicial set $\Ups^n_A\subset\EZ^n$, given by $\Delta^n$ with ``thickened'' boundary components determined by $A$. In particular, $\Ups^n_{\{\}}=\Delta^n$ and $\Ups^n_{\{0,\dots, n\}}=\wh\Theta^n$. (Note, that $\Ups^n_A$ may be explicitly described to have $p$-simplicies given by sequences $(i_0,\dots, i_p)\in \{0,\dots,n\}^p$ such that either $i_0\leq\dots \leq i_p$, or there exists an element $i\in A$ so that $i_0\neq i, \dots, i_p\neq i$, or both.) We show that the $|\Ups^n_A|$ are contractible for all $n$ and $A$. Since all $|\Ups^n_A|$ are CW-complexes, this is equivalent to showing that the $|\Ups^n_A|$ are connected and have zero homotopy groups.

We will repeatedly use the fact, if $X$, $Y$, $X\cap Y$ and $X\cup Y$ are CW-complexes, and $X$, $Y$, $X\cap Y$ are contractible, then $X\cup Y$ is also contractible (which follows, since $X\cup Y$ is certainly connected, has vanishing $\pi_1$ due to van Kampen, vanishing homology groups due to Mayer-Vietoris, and thus vanishing homotopy groups due to Hurewicz).

When $n=1$, using that $\EZ^0=\Delta^0$, we have for any $A\subset \{0,1\}$ that $\Ups^1_{A}=\Delta^1$, and $|\Delta^1|$ is contractible.

Now, for $n>1$, assume by induction, that the $|\Ups^k_B|$ are contractible for all $k<n$ and all $B\sub\{0,\dots,k\}$. We perform a second induction on the number of elements of $A\sub \{0,\dots, n\}$. First, note that $\Ups^n_{\{\}}=\Delta^n$, and $|\Delta^n|$ is contractible. Thus, assume by induction that all $|\Ups^n_{A}|$ with $|A|<\ell$ are contractible. Now, let $A=\{i_1,\dots, i_\ell\}\sub \{0,\dots, n\}$ be an $\ell$-element set with, say, $i_1<\dots<i_\ell$. Writing $\Ups^n_{\{i_1,\dots, i_\ell\}}=\Ups^n_{\{i_1,\dots, i_{\ell-1}\}}\cup \delta_{i_\ell}\EZ^{n-1}$, we know by induction that $|\Ups^n_{\{i_1,\dots, i_{\ell-1}\}}|$ is contractible, and also that $|\delta_{i_\ell}\EZ^{n-1}|\approx |\EZ^{n-1}|$ is contractible (which was reviewed in the proof of proposition \ref{PROP:whK-is-Kan}). Furthermore, $\Ups^n_{\{i_1,\dots, i_{\ell-1}\}}\cap \delta_{i_\ell}\EZ^{n-1}=\delta_{i_\ell}\Ups^{n-1}_{\{i_1,\dots, i_{\ell-1}\}}\cong\Ups^{n-1}_{\{i_1,\dots, i_{\ell-1}\}}$, and we know by the first induction that $|\Ups^n_{\{i_1,\dots, i_{\ell-1}\}}|\cap |\delta_{i_\ell}\EZ^{n-1}|=|\Ups^{n-1}_{\{i_1,\dots, i_{\ell-1}\}}|$ is contractible as well. Thus, by the above fact, we see that $|\Ups^n_{\{i_1,\dots, i_\ell\}}|=|\Ups^n_{\{i_1,\dots, i_{\ell-1}\}}|\cup |\delta_{i_\ell}\EZ^{n-1}|$ is also contractible.
\end{proof}
We are now ready to prove proposition \ref{PROP:X=whX}.
\begin{proof}[Proof of proposition \ref{PROP:X=whX}]
Since both $K$ and $\ov K$ are Kan complexes, it suffices to show that $F^\sharp:\ov K\to K$ induces isomorphisms on all simplicial homotopy groups (since these coincide with the homotopy groups of their geometric realizations; see \cite[theorems 16.1, 16.6]{May}).

First, for $n=0$, $F$ induces a map $ \pi_0(\ov K)\to \pi_0(K)$ which is onto since $\wh \Delta^0=\Delta^0$ and thus $\ov K_0=K_0$. To see that the induced map $ \pi_0(\ov K)\to \pi_0(K)$ is one-to-one, assume $a, b\in K_0$ are equivalent $a\sim b$ in $\pi_0(K)$. Since $K$ is a Kan complex, this means that (instead of a sequence of $1$-simplicies) there exists a single $c\in K_1$ so that $d_0(c)=a$ and $d_1(c)=b$. We need to check that $a\sim b$ in $\pi_0(\ov K)$, i.e., there exists a $\ov c\in\ov K_1$ with $d_0(\ov c)=a$ and $d_1(\ov c)=b$. Thus we need a simplicial set map $\Delta^1\to \ov K$, i.e., a map $\EZ^1\to K$ making the following diagram commute
\[
\xymatrix{\wh\Theta^1=\Delta^1\cup \delta_0\wh \Delta^0\cup \delta_1\wh\Delta^0 \ar[rr]^{\hspace{1.5cm} c\cup a\cup b} \ar[d] &&  K \ar[d]\\
\wh\Delta^1 \ar[rr]  \ar@{-->}[rru] && {*}}
\]
Note that the top arrow is well defined, and, since the left map is a trivial cofibration (i.e.  injective and a weak equivalence) and $K$ is a Kan complex, it follows that it lifts to a map $\wh\Delta^1\to K$, as needed.

Now, for $n\geq 1$, $F$ induces a map $\pi_n(\ov K,*)\to  \pi_n(K,*)$ which is onto: if $c\in K_n$ with $d_i(c)=*$ for all $i$, represents an element of $\pi_n(K,*)$, then we want to produce a $\ov c\in \ov K_n$, i.e., $\ov c:\EZ^n\to K$, with $d_i(\ov c)=*$ for all $i$ and which restricts to $c$ under $F$. Thus, we need to find a lift making the following diagram commute
\[
\xymatrix{\wh\Theta^n=\Delta^n\cup \delta_0\wh \Delta^{n-1}\cup\dots\cup \delta_n\wh\Delta^{n-1} \ar[rrr]^{\hspace{2.5cm} c\cup *\cup \dots\cup *} \ar[d] &&&  K \ar[d]\\
\EZ^n \ar[rrr]  \ar@{-->}[rrru] &&& {*}}
\]
Again, the top arrow is well defined, since $c$ restricts trivially to its boundaries. Just as before, we can find a lift, because $\wh\Theta^n\to \EZ^n$ is a trivial cofibration and $K$ is a Kan complex. Finally, we need to check that $F$ induces a map $\pi_n(\ov K,*)\to  \pi_n(K,*)$, which is one-to-one. Since this map is a map of groups, it suffices to check that the kernel is trivial. More explicitly, we need to show that if $\ov c\in K_n$ with $d_i(\ov c)=*$ for all $i$ represents a class of $\pi_n(\ov K,*)$, which maps to $c=\ov c\circ F^n:\Delta^n\stackrel{F^n}\to \EZ^n\stackrel{\ov c}\to K$ which is trivial in $\pi_n(K,*)$, then $\ov c$ is trivial in $\pi_n(\ov K,*)$. For $c$ to be trivial in $\pi_n(K,*)$ means that there is an $n+1$ simplex $q\in K_{n+1}$ so that $d_0(q)=c$ and $d_i(q)=*$ for all $i\geq 1$. We thus have the setup for the following diagram:
\[
\xymatrix{\wh\Theta^{n+1}=\Delta^{n+1}\cup \delta_0\wh \Delta^{n}\cup\delta_1\wh \Delta^{n}\cup\dots\cup \delta_n\wh\Delta^{n} \ar[rrr]^{\hspace{3cm} q\cup \ov c\cup *\cup \dots\cup *} \ar[d] &&&  K \ar[d]\\
\EZ^{n+1} \ar[rrr]  \ar@{-->}[rrru] &&& {*}}
\]
Since $\wh\Theta^{n+1}\to \EZ^{n+1}$ is a trivial cofibration and $K$ is a Kan complex, there exists a lift $\ov q\in \ov K_{n+1}$ with $d_0(\ov q)=\ov c$ and $d_i(\ov q)=*$ for all $i\geq 1$. This shows that $\ov c$ does indeed represent the trivial class in $\pi_n(\ov K,*)$.
\end{proof}

\section{Explicit Description of Totalization}\label{SEC: explicit tot}

We now review the notion of totalization of a cosimplicial simplicial set.
\subsection{\bf Totalization}\label{SUBSEC:totalization}

We recall from \cite[Definition D.1]{GMTZ} and \cite[Definition 18.6.3]{Hirs} the definition of totalization. Let $K^\bu:\Del\to \sSet$ be a cosimplicial simplicial set, i.e., $K^\ell:=K([\ell])$ is a simplicial set $K^\ell=K^\ell_\bu$. Then, the totalization $\Tot (K^\bu_\bu)$ of $K$ is defined as the simplicial set, which is the equalizer of the maps
\begin{equation}\label{DEF:Tot}
\Tot (K^\bu_\bu)\to \prod_{[\ell]\in {\bf Obj}(\Del)} (K^\ell)^{\Delta^{\ell}} \stackrel[\psi]{\phi}{\rightrightarrows} \prod_{\rho:[n]\to [m]} (K^m)^{\Delta^{n}}
\end{equation}
Here, by definition, $(K^p)^{\Delta^q}$ is the simplicial set whose $n$-simplicies are simplical set maps $((K^p)^{\Delta^q})_n=\sSet((\Delta^n\times \Delta^q)_\bu, K^p_\bu)$. Then a $k$-simplex in the totalization is given by some collection 
\begin{equation}\label{EQU:k-simplex-in-Tot}
\{x^{(k,\ell)}\}_{\ell\geq 0}, \text{ where } x^{(k,\ell)} \in \sSet(\Delta^k \times \Delta^{\ell} , K^{\ell}),
\end{equation}
satisfying the coherence condition that they are in the above equalizer. Explicitly, for a fixed $j=0,\dots, \ell+1$ the map $\delta_j:[\ell]\to [\ell+1]$ which skips $j$, induces the maps
\begin{align}\label{EQU:d_j}
x^{(k,\ell+1)}\in && \sSet(\Delta^k\times \Delta^{\ell+1}, K^{\ell+1})&\stackrel {d_j}\to \sSet(\Delta^k\times \Delta^{\ell}, K^{\ell+1})\\ 
\label{EQU:d^j}
x^{(k,\ell)}\in &&\sSet(\Delta^k\times \Delta^{\ell}, K^{\ell})&\stackrel {d^j}\to \sSet(\Delta^k\times \Delta^{\ell}, K^{\ell+1})
\end{align}
Then, for $x^{(k,\ell+1)}$ and $x^{(k,\ell)}$ as above,
\begin{equation}\label{EQU:coherence}
d_j(x^{(k,\ell+1)})=d^j(x^{(k,\ell)}).
\end{equation}

Thus, a $k$-simplex, $\{x^{(k, \ell)}\}_{\ell=0,1,\dots}$, in the totalization of a cosimplicial simplicial set, $Tot\left( K^{\bullet}_{\bullet} \right)$ is given by maps $x^{(k,\ell)} \in \sSet(( \Delta^k \times \Delta^{\ell})_\bu , K^{\ell}_{\bullet})$ for each $\ell=0,1,\dots$, which can be thought of as a coherent ``decoration'' of the simplicial sets $\Delta^k \times \Delta^{\ell}$, for $\ell=0,1,\dots$, by simplices in $K^{\ell}_{\bullet}$.

\subsection{\bf Simplicies of $\Delta^k\times \Delta^\ell$}\label{REM:Tot(K) simplices}
We now recall that there is a nice book-keeping device for the simplicies of $\Delta^k \times \Delta^\ell$.  In fact, the $p$-simplicies of $\Delta^k \times \Delta^\ell$ can be described by non-decreasing paths with $p+1$ vertices in a $(k+1)\times (\ell+1)$ grid; we also call this a $p$-path. For example, the maximally non-degenerate $(4+7)$-simplicies of $\Delta^4 \times \Delta^7$ can be labeled by paths\footnote{Informally, this path might be referred to as a ``taxi-cab''path as it only moves in a rectangular fashion.} through a $(4+1)\times (7+1)$ grid, necessarily starting from  \scalebox{0.8}{$\bmat{0 \\ 0}$} and ending at \scalebox{0.8}{$\bmat{4\\ 7}$}. For example, the following path of labels, which we denote by \scalebox{0.8}{$\bmat{[c|c|c|c|c|c|c|c|c|c|c|c]
   0 & 1 & 1 & 1 & 2 & 3 & 4 & 4 & 4 & 4 & 4 & 4\\ 0 & 0 & 1 & 2 & 2 & 2 & 2 & 3 & 4 & 5 & 6 & 7}$} labels an element of $(\Delta^4\times \Delta^7)_{11}$:
\begin{equation}\label{EQ: Delta Grid} \begin{tikzpicture}
\foreach \x in {0,1,...,7}
\foreach \y in {0, 1, ..., 4}
{
{ 
\node[draw, rectangle, fill=white!30] at ($({\x},{4-\y})$) {\scalebox{.6}{$\bmat{{\y}\\ {\x}}$}};}
  \coordinate (A) at (0, 4);
  \coordinate (B) at (2, 3);
  \coordinate (C) at (7,0);
   \begin{scope}[on background layer]
        \draw (A) |- (B) |-(C);
    \end{scope}
}
\end{tikzpicture} \end{equation}
We can apply $x^{(4,7)} \in \sSet( \Delta^4 \times \Delta^7, K^7)$ to this path, which will give an element $x^{(4,7)}_{\scalebox{0.6}{$\bmat{[c|c|c|c|c|c|c|c|c|c|c|c]   0 & 1 & 1 & 1 & 2 & 3 & 4 & 4 & 4 & 4 & 4 & 4\\ 0 & 0 & 1 & 2 & 2 & 2 & 2 & 3 & 4 & 5 & 6 & 7}$}}\in K^7_{11}$ (note the simplicial degree $11$ comes from the $11$-path with $12$ vertices). Note that, just as the simplices of the standard $n$-simplex have direction, these paths must be non-decreasing in both directions. Additionally, the faces of a $p$-simplex of $\Delta^k\times \Delta^\ell$ given by a path would consist of sub-sequences of that path, e.g., \scalebox{0.8}{$\bmat{[c|c|c|c|c|c]  1 & 1 & 2  & 4 & 4 & 4\\ 0 &  2 & 2 & 3 & 4 & 6}$} describes a $5$-simplex in $(\Delta^4\times\Delta^7)_5$ which is a lower face of  the above $11$-simplex. Degenerate simplicies are described by paths where at least one of the indices is repeated, e.g., \scalebox{0.8}{$\bmat{[c|c|c|c|c]  1 & 2  & 4 & 4 & 4 \\ 2 & 2 & 3 & 3 & 6 }$}.

Using this notation, the coherence condition \eqref{EQU:coherence} can be stated more precisely as follows. Let $K^{\bullet}_{\bullet}$ be a cosimplicial simplicial set and let $\delta_j:[\ell]\to [\ell+1]$ be the map that skips $j$. We have the coface maps, $d^j: K^{\ell}_{\bullet} \to K^{\ell+1}_{\bullet}$, as well as the maps $d_j$ in \eqref{EQU:d_j} given by precomposition with $\Delta^{\ell}_{\bu} \to \Delta^{\ell+1}_{\bu}$. Then we can explicitly describe the $k$-simplices of the {totalization}, $Tot\left( K^{\bullet}_{\bullet} \right)_k$, as collections $\{x^{(k, \ell)} \in \sSet(\Delta^{k} \times \Delta^{\ell}, K^{\ell}_{\bullet})\}_{\ell=0,1,\dots}$, which, applied to $p$-simplices of $\Delta^{k} \times \Delta^{\ell}$ labeled by the paths \scalebox{0.8}{$\bmat{[c|c|c]  \al_0& \dots & \al_p\\ \be_0 &  \dots & \be_p }$} with $0\leq \al_0\leq \dots\leq \al_p\leq k$ and $0\leq \be_0\leq\dots \leq \be_p\leq \ell$ as described above, assign elements $x_{\scalebox{0.6}{$\bmat{[c|c|c]    \al_0& \dots & \al_p\\ \be_0 &  \dots & \be_p }$}}^{(k, \ell)} \in  K^{\ell}_{p} $, satisfying 
\begin{equation}\label{EQ:Tot explicit coherence}  x_{\scalebox{0.6}{$\bmat{[c|c|c] \al_0 &  \dots & \al_p  \\  \delta_j(\be_0)& \dots & \delta_j(\be_p) }$}}^{(k, \ell+1)}  =d^j\left( x_{\scalebox{0.6}{$\bmat{[c|c|c]\al_0& \dots & \al_p \\ \be_0 &  \dots & \be_p  }$}}^{(k, \ell)}\right) \quad\in K^{\ell+1}_p
\end{equation}
For example, for $k=2$, we have the following assignments for $\ell = 0,1$

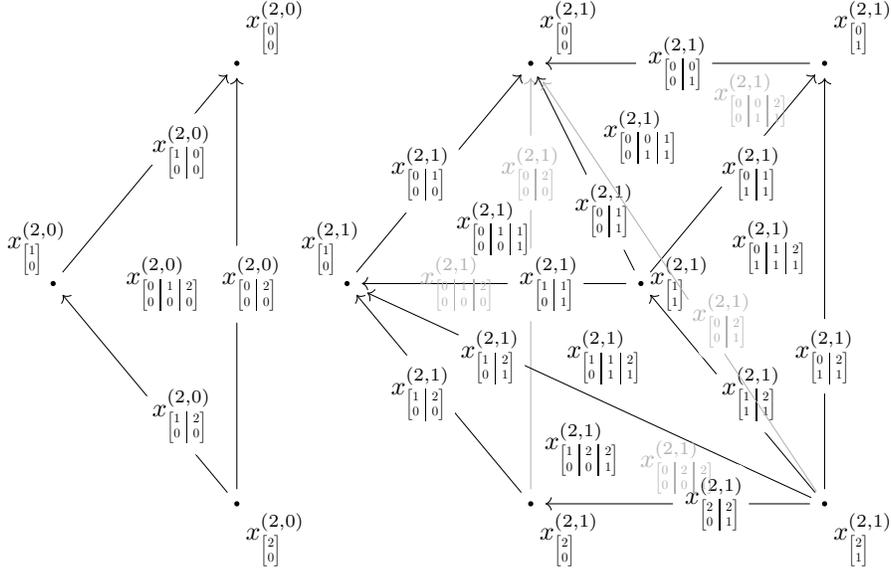
\begin{equation*}
 \resizebox{12cm}{!}{\begin{tikzpicture}[baseline={([yshift=-.5ex]current bounding box.center)},vertex/.style={anchor=base,
     circle,fill=black!25,minimum size=18pt,inner sep=2pt}]   
 \def\top{6};
 \def\bot{0};
 \def\mid{3};
 \def\dep{2.5};
   \foreach \t in {2,10}
    {
        \path[shorten >=0.2cm,shorten <=0.2cm,<-]  (\t,\top)       edge (\t,\bot);   
        \fill (\t, \bot)circle (1pt);
        \fill (\t, \top)circle (1pt);
        \path[shorten >=0.2cm,shorten <=0.2cm,<-]  (\t,\top)       edge (\t- \dep,\mid);   
        \path[shorten >=0.2cm,shorten <=0.2cm,<-]  (\t- \dep,\mid)       edge (\t,\bot); 
        \fill (\t- \dep,\mid)circle (1pt);
        }
   \foreach \t in {6}
    {
                \path[gray!60, shorten >=0.2cm,shorten <=0.2cm,<-]  (\t,\top)       edge (\t,\bot);   
        \fill (\t, \bot)circle (1pt);
        \fill (\t, \top)circle (1pt);
        \path[shorten >=0.2cm,shorten <=0.2cm,<-]  (\t,\top)       edge (\t- \dep,\mid);   
        \path[shorten >=0.2cm,shorten <=0.2cm,<-]  (\t- \dep,\mid)       edge (\t,\bot); 
        \fill (\t- \dep,\mid)circle (1pt);
        }
        
                \fill[white] (6, \mid)circle (3 pt);
     \path[shorten >=0.2cm,shorten <=0.2cm,<-]  (6,\top)       edge (10,\top);   
     \path[shorten >=0.2cm,shorten <=0.2cm,<-]  (6-\dep,\mid)       edge (10-\dep,\mid);   
     \path[shorten >=0.2cm,shorten <=0.2cm,<-]  (6,\bot)       edge (10,\bot);       
     
\node[above right] at (2,\top) {$x_{\scalebox{0.5}{$\bmat{[c] 0\\ 0   }$}}^{(2, 0)}$};
\node[above left, xshift=0.3cm] at (2- \dep,\mid) {$x_{\scalebox{0.5}{$\bmat{[c] 1\\ 0   }$}}^{(2, 0)}$};
\node[below right] at (2,\bot) {$x_{\scalebox{0.5}{$\bmat{[c] 2\\ 0   }$}}^{(2, 0)}$};

\node at (1.25,4.8)  {\contour{white}{$x_{\scalebox{0.5}{$\bmat{[c|c] 1&0\\ 0&0   }$}}^{(2, 0)}$}};
\node at (1.25,1.2)  {\contour{white}{$x_{\scalebox{0.5}{$\bmat{[c|c] 1&2\\ 0&0   }$}}^{(2, 0)}$}};
\node at (2.2,3)  {\contour{white}{$x_{\scalebox{0.5}{$\bmat{[c|c] 0&2\\ 0&0   }$}}^{(2, 0)}$}};

\node at (1,3)  {\contour{white}{$x_{\scalebox{0.5}{$\bmat{[c|c|c] 0& 1&2\\ 0 & 0&0   }$}}^{(2, 0)}$}};

     \path[shorten >=0.2cm,shorten <=0.2cm,<-]  (6,\top)       edge (10 - \dep,\mid);   
     \path[shorten >=0.2cm,shorten <=0.3cm,<-]  (6 - \dep,\mid)       edge (10,\bot);

\node[above right] at (6,\top) {$x_{\scalebox{0.5}{$\bmat{[c] 0\\ 0   }$}}^{(2, 1)}$};
\node[above left, xshift=0.3cm] at (6- \dep,\mid) {$x_{\scalebox{0.5}{$\bmat{[c] 1\\ 0   }$}}^{(2, 1)}$};
\node[below right] at (6,\bot) {$x_{\scalebox{0.5}{$\bmat{[c] 2\\ 0   }$}}^{(2, 1)}$};

\node[above right] at (10,\top) {$x_{\scalebox{0.5}{$\bmat{[c] 0\\ 1   }$}}^{(2, 1)}$};
\node[right] at (10- \dep,\mid) {$x_{\scalebox{0.5}{$\bmat{[c] 1\\ 1   }$}}^{(2, 1)}$};
\node[below right] at (10,\bot) {$x_{\scalebox{0.5}{$\bmat{[c] 2\\ 1   }$}}^{(2, 1)}$};

\node[gray!60] at (9,\top -0.5) {$x_{\scalebox{0.5}{$\bmat{[c|c|c] 0 & 0 & 2\\0 & 1&  1   }$}}^{(2, 1)}$};
\node[gray!60] at (5,\mid) {$x_{\scalebox{0.5}{$\bmat{[c|c|c] 0 & 1 & 2\\0 & 0&  0   }$}}^{(2, 1)}$};
\node[gray!60] at (6,\top-1.5)  {\contour{white}{$x_{\scalebox{0.5}{$\bmat{[c|c] 0&2\\  0&0   }$}}^{(2, 1)}$}};

\node at (6.7,\mid-2.25)  {$x_{\scalebox{0.5}{$\bmat{[c|c|c] 1&2&2\\  0&0&1   }$}}^{(2, 1)}$};
\node at (7,\mid-1)  {$x_{\scalebox{0.5}{$\bmat{[c|c|c] 1&1&2\\  0&1&1   }$}}^{(2, 1)}$};
\node at (7.5,\top-1)  {$x_{\scalebox{0.5}{$\bmat{[c|c|c] 0&0&1\\  0&1&1   }$}}^{(2, 1)}$};
\node at (5.5,\mid+0.75)  {$x_{\scalebox{0.5}{$\bmat{[c|c|c] 0&1&1\\  0&0&1   }$}}^{(2, 1)}$};
\node at (9.25,\mid+0.5)  {$x_{\scalebox{0.5}{$\bmat{[c|c|c] 0&1&2\\  1&1&1   }$}}^{(2, 1)}$};

\node at (6- \dep + 1,\top-1.5)  {\contour{white}{$x_{\scalebox{0.5}{$\bmat{[c|c] 0&1\\  0&0   }$}}^{(2, 1)}$}};
\node at (6- \dep + 1,\bot+1.5)  {\contour{white}{$x_{\scalebox{0.5}{$\bmat{[c|c] 1&2\\  0&0   }$}}^{(2, 1)}$}};
\node at (8,\top)  {\contour{white}{$x_{\scalebox{0.5}{$\bmat{[c|c] 0&0\\  0&1   }$}}^{(2, 1)}$}};
\node at (8.5,\bot)  {\contour{white}{$x_{\scalebox{0.5}{$\bmat{[c|c] 2&2\\  0&1   }$}}^{(2, 1)}$}};
\node[gray!60] at (8,\bot +0.5) {$x_{\scalebox{0.5}{$\bmat{[c|c|c] 0& 2 & 2\\ 0 & 0 & 1   }$}}^{(2, 1)}$};

\node at (10,\mid-1)  {\contour{white}{$x_{\scalebox{0.5}{$\bmat{[c|c] 0&2\\  1&1   }$}}^{(2, 1)}$}};
\node at (5.45,\mid-1)  {\contour{white}{$x_{\scalebox{0.5}{$\bmat{[c|c] 1&2\\  0&1   }$}}^{(2, 1)}$}};
\node at (6.25,\mid)  {\contour{white}{$x_{\scalebox{0.5}{$\bmat{[c|c] 1&1\\  0&1   }$}}^{(2, 1)}$}};
\node at (7,\mid+1)  {\contour{white}{$x_{\scalebox{0.5}{$\bmat{[c|c] 0&1\\  0&1   }$}}^{(2, 1)}$}};
\node at (9,\mid+1.5)  {\contour{white}{$x_{\scalebox{0.5}{$\bmat{[c|c] 0&1\\  1&1   }$}}^{(2, 1)}$}};
\node at (9,\mid-1.5)  {\contour{white}{$x_{\scalebox{0.5}{$\bmat{[c|c] 1&2\\  1&1   }$}}^{(2, 1)}$}};
 \path[gray!60, shorten >=0.2cm,shorten <=0.3cm,<-]  (6, \top)       edge (10,\bot);   
\node[gray!60, xshift = 0.1cm] at (8.5,\mid-0.5)  {\contour{white}{$x_{\scalebox{0.5}{$\bmat{[c|c] 0&2\\  0&1   }$}}^{(2, 1)}$}};
\end{tikzpicture}}
\end{equation*}   
 
As an example, for $\delta_0:[0]\to [1]$, equation \eqref{EQ:Tot explicit coherence} yields $x_{\scalebox{0.6}{$\bmat{[c|c|c] 0& 1 & 2\\ 1 &  1 & 1   }$}}^{(2, 1)}  =d^0\left( x_{\scalebox{0.6}{$\bmat{[c|c|c]  0& 1 & 2\\ 0 &  0 & 0   }$}}^{(2, 0)}\right)$, which relates the cells for different $\ell$'s.

Note, that for a fixed $k$ and $\ell$, the $x_{\scalebox{0.6}{$\bmat{[c|c|c]  \al_0& \dots & \al_p\\ \be_0 &  \dots & \be_p   }$}}^{(k, \ell)} \in  K^{\ell}_{p}$ are in fact determined by the maximal paths $x_{\scalebox{0.6}{$\bmat{[c|c|c]  \al_0& \dots & \al_{k+\ell}\\ \be_0 &  \dots & \be_{k+\ell}  }$}}^{(k, \ell)} \in  K^{\ell}_{k+\ell}$, since each $p$-path is a subpath of a maximal path and so the $p$-cell is in the image of some face map $K^\ell_{k+\ell}\to K^\ell_p$ for some map $[p]\to [k+\ell]$.

\begin{example}\label{EXA:Tot-open-cover}
For example, for a simplicial presheaf $\F:\CMan^{op}\to\sSet$, and an open cover $\mc U=\{U_i\}_{i\in \mc I}$ of $X\in \CMan$, and we take
$$K^\ell_p=\F_p(\CN \mc U_\ell)=\prod_{i_0,\dots,i_\ell \in \mc I}\F_p(U_{i_0,\dots, i_\ell}).$$
In this case a $p$-cell in $x\in K^\ell_p$ is given by $x=\{x_{i_0,\dots, i_\ell}\}$, where for each $(\ell+1)$-fold intersection $U_{i_0,\dots, i_\ell}$, $x_{i_0,\dots, i_\ell} \in \F_p(U_{i_0,\dots, i_\ell})$ is a $p$-cell.  Note, that the map $d^j:K^\ell \to K^{\ell+1}$ in \eqref{EQU:d^j} and \eqref{EQ:Tot explicit coherence} is induced by the inclusions $incl: U_{i_0,\dots,i_{\ell+1}}\hookrightarrow U_{i_0,\dots,\widehat{i_j},\dots,i_{\ell+1}}$ as $\F_p(incl):\F_p(U_{i_0,\dots,\widehat{i_j},\dots,i_{\ell+1}})\to \F_p(U_{i_0,\dots,i_{\ell+1}})$. In particular, continuing the example from the figure above, $x_{\scalebox{0.6}{$\bmat{[c|c|c]  0& 1 & 2\\ 0 &  0 & 0   }$}}^{(2, 0)}$ and $x_{\scalebox{0.6}{$\bmat{[c|c|c] 0& 1 & 2\\ 1 &  1 & 1   }$}}^{(2, 1)}$ have components, 
\[ x_{\scalebox{0.6}{$\bmat{[c|c|c]  0& 1 & 2\\ 0 &  0 & 0   }$};i}^{(2, 0)} \in  \F_2(U_{i}) \quad  x_{\scalebox{0.6}{$\bmat{[c|c|c] 0& 1 & 2\\ 1 &  1 & 1   }$};i_0i_1}^{(2, 1)} \in  \F_2(U_{i_0i_1}) 
\] 
respectively and the compatibility of equation \eqref{EQ:Tot explicit coherence} now yields, 
\[x_{\scalebox{0.6}{$\bmat{[c|c|c] 0& 1 & 2\\ 1 &  1 & 1   }$};i_0i_1}^{(2, 1)}  =d^0\left( x_{\scalebox{0.6}{$\bmat{[c|c|c]  0& 1 & 2\\ 0 &  0 & 0   }$}}^{(2, 0)}\right)= \restr{x_{\scalebox{0.6}{$\bmat{[c|c|c]  0& 1 & 2\\ 0 &  0 & 0   }$}; i_1}^{(2, 0) }}{U_{i_0i_1}}
\]

\end{example}

\subsection{\bf Totalization for the case $K=\sSet(\EZ,\tilde K)$}\label{SUBSEC:Tot(sSet(EZ,Ktilde)}
We are interested in the totalization of $K_{\bullet}^{\bullet}=\PERF^\EZ(\CN U)=\sSet(\EZ,\PERF(\CN U))$. Thus, assume now, that we have a cosimplicial simplicial set $K_{\bullet}^{\bullet}$, which is of the form $K^\ell_p:=\sSet ( \EZ^p, \tilde{K}^{\ell} )$ for some other cosimplicial simplicial set $\tilde{K}^\bu_\bu$.  By rewriting simplicial sets as colimits of their simplices, and using continuity of the hom-functor in the category $\sSet$, we see that, 
\begin{align}\label{EQU:Tot-for-Deltahat}
\sSet \left( \Delta^k \times \Delta^{\ell}, K^{\ell} \right) &= \sSet \left( \colim\limits_{\Delta^p \to  \Delta^k \times \Delta^{\ell}} \Delta^p, K^{\ell}\right)=  \lim\limits_{\Delta^p \to  \Delta^k \times \Delta^{\ell}} \sSet \left( \Delta^p, K^{\ell} \right)\\
\nonumber &= \lim\limits_{\Delta^p \to  \Delta^k \times \Delta^{\ell}} K^{\ell}_{p} = \lim\limits_{\Delta^p \to  \Delta^k \times \Delta^{\ell}} \sSet \left( \EZ^{p} , \tilde{K}^{\ell} \right)\\
\nonumber
&=\sSet \left(   \colim\limits_{\Delta^p \to  \Delta^k \times \Delta^{\ell}}  \EZ^{p} , \tilde{K}^{\ell} \right)
\end{align} 
We see from the above identification that decorations of simplicial sets $ \Delta^k \times \Delta^{\ell}$ by simplices in  $K^{\ell}_{\bullet}$ is equivalent to first glueing the simplicial sets $\EZ^{n}$ along the corresponding $\Delta^n$ sitting inside $ \Delta^k \times \Delta^{\ell}$, and then decorating this colimit made of various $\EZ^{n}$ by simplices in $\tilde{K}^{\ell}$. Using the description of $\EZ$ from example \ref{EXA:EZ}, it now follows that the $k$-simplices of $\Tot \left( K_{\bullet}^{\bullet} \right)$ are in fact given by\label{PAGE:Tot-for-Deltahat-map} $x_{\scalebox{0.6}{$\bmat{[c|c|c] \al_0& \dots & \al_p\\ \be_0 &  \dots & \be_p    }$}}^{(k, \ell)} \in  \tilde{K}^{\ell}_{p}$, where this time the path described by \scalebox{0.8}{$\bmat{[c|c|c]   \al_0& \dots & \al_p \\ \be_0 &  \dots & \be_p }$}  is now permitted to move horizontally and vertically in each direction in the grid, i.e., possibly decreasing, but within the indices of a non-decreasing path. For example, in the $(2+1)\times(3+1)$ grid of vertices, take the $5$-cell given the map $\Delta^5 \hookrightarrow  \Delta^2 \times \Delta^{3}$ whose non-decreasing path is \scalebox{0.8}{$\bmat{[c|c|c|c|c|c]  0 & 0 & 0  & 1 & 1 & 2\\  0 &  1 & 2 & 2 & 3 & 3 }$}. Then, for the corresponding $\EZ^{5}$, there is a non-degenerate $9$-simplex
\begin{equation}\label{EQU:x-23-in-ktilde}
x^{(2,3)}_{\scalebox{0.6}{$\bmat{[c|c|c|c|c|c|c|c|c|c]     0 & 0 & 1 & 0 & 0 & 1 & 0 & 2 & 1 & 2  \\ 0 & 1 & 3 & 1 & 2 & 2 & 1 & 3 & 3 & 3}$}} \in \tilde K^3_9,
\end{equation}
which is both increasing and decreasing using the indices of the $5$-path \scalebox{0.8}{$\bmat{[c|c|c|c|c|c]  0 & 0 & 0  & 1 & 1 & 2\\  0 &  1 & 2 & 2 & 3 & 3 }$} in $\Delta^2\times\Delta^3$. Thus, in the totalization $Tot(K)$, a $2$-simplex $x=\{x^{(2,\ell)}\}$ needs to assign such an element in $\tilde K^3_9$ to the $9$-path from  \eqref{EQU:x-23-in-ktilde}. However, note that there is no assignment to the path \scalebox{0.8}{$\bmat{[c|c|c|c]   0 & 1 & 0  & 1\\  0 &  0 & 1 & 1 }$}, because every map $\Delta^n\to \Delta^k\times \Delta^\ell$ is necessarily non-decreasing in both components and so one can never obtain both \scalebox{0.8}{$\bmat{0\\ 1}$} and \scalebox{0.8}{$\bmat{1\\0}$} in the same path. To summarize, a cell in $Tot(K)$ has to assign elements in $\tilde K$ exactly to any path which uses the indices of a non-decreasing path.

Finally, note that the coherence condition on these simplices of the totalization is the same as expressed in \eqref{EQ:Tot explicit coherence}.

\section{Totalization and Fibrant Objects}
The purpose of this appendix is to prove proposition \ref{PROP: tot cech kan}. 
\begin{proposition}\label{PROP: tot cech kan}
If $\F$ is a projectively fibrant simplicial presheaf (such as, e.g., $\F=\IVB$) then $Tot(\F(\CN U_{\bu}))$ is a Kan complex.
\end{proposition}

We start with the following lemma.
\begin{lemma}
The totalization functor (see appendix \ref{SEC: explicit tot})   $\Tot: (Set^{\Delta^{op}})^\Delta \rightarrow \Set^{\Delta^{op}}$ is a right adjoint. 
\end{lemma}

\begin{proof}
We prove this directly by defining the left adjoint $L$.  For any simplicial set $X^\bullet$, let $L(X^\bullet)$ be the cosimplicial simplicial set $n \mapsto X^\bullet \times \Delta^n$, where $\Delta^n$ is the standard $n$-simplex.  

To show that these functors form an adjoint pair, let $X^\bullet$ be a simplicial set and $Y^\bullet_\bullet$ be a cosimplicial simplicial set.  Since $Set^{\Delta^{op}}$ is a simplicial model category (under the usual Quillen structure), $Set^{\Delta^{op}}(X \times \Delta^n, Y^n_\bullet)$ is in bijection with $Set^{\Delta^{op}}(X, (Y^n_\bullet)^{\Delta^n})$.  Since $\Tot (Y^\bullet_\bullet) = (Y^\bullet_\bullet)^\Delta$, we have our bijection.    
\end{proof}

\begin{lemma}
The functors $(L, \Tot)$ form a Quillen adjunction between the Reedy model structure \cite[section 15]{Hirs} of cosimplicial simplicial sets and the usual Quillen model structure on simplicial sets.
\end{lemma}

\begin{proof}
It is enough to show that $L$ preserves cofibrations and trivial cofibrations.  Suppose $f: X^\bullet \rightarrow Y^\bullet$ is a cofibration of simplicial sets, i.e., a levelwise monomorphism.  By \cite[Theorem 15.9.9]{Hirs}, to show that $L(f)$ is a Reedy cofibration, it is enough to show that $L(f)$  is a monomorphism that takes the maximal augmentation of $L(X^\bullet)$ isomorphically onto the maximal augmentation of $L(Y^\bullet)$.  Since $L(f) = f \times Id$ and $f$ is a levelwise monomorphism, $L(f)$ is a monomorphism.  The maximal augmentation of $L(X^\bullet)$ and $L(Y^\bullet)$ are empty.  So $L$ preserves cofibrations.

Suppose $f: X^\bullet \rightarrow Y^\bullet$ is a trivial cofibration.  We need to show that $L(f): L(X^\bullet) \rightarrow L(Y^\bullet)$ is a Reedy weak equivalence. Since $L(f) = f \times Id$,  then $L_nF: X^\bullet \times \Delta^n  \rightarrow Y^\bullet \times \Delta^n$ is a weak equivalence.   
\end{proof}

\begin{lemma}\label{LEM: Tot Kan}
Let $X$ be a Reedy fibrant cosimplicial simplicial set.  Then $\Tot(X)$ is a Kan complex.
\end{lemma}

\begin{proof}
Since $\Tot$ is a right adjoint, it preserves fibrations and terminal objects.  So $\Tot$ preserves fibrant objects.
\end{proof}

\begin{lemma}\label{LEM: F CN Reedy}
Let $V$ be a manifold and $U_{\bu}$ be an open cover of $V$.  Let $\F$ be a simplicial presheaf that takes values in Kan complexes.  Then $\F(\CechNerve U_{\bu}): \Delta \rightarrow Set^{\Delta^{op}}$ (see \eqref{EQ: F of CN}) is a Reedy fibrant cosimplicial simplicial set.
\end{lemma}

\begin{proof}
This proof uses some conventions from \cite[section 15]{Hirs} for the Reedy model structure and is analogous to \cite[proposition 4.3]{BHW}. We need to show that the matching map $\F(\CechNerve U_n) \rightarrow M_n( \F(\CechNerve U_{\bu}))$ is a fibration for each $n$, where 
\[ \F(\CechNerve U_n):= \underline{sPre}\left( {\coprod_{i_0 \cdots i_n}}y U_{i_0 \cdots i_n}, \F\right) = {\prod_{i_0 \cdots i_n}}\F\left( U_{i_0 \cdots i_n}\right).\]

Write $\CechNerve U_n$ as the coproduct
\begin{eqnarray*}
 \CechNerve U_n &=& \underset{i_j \neq i_{j+1}}  {\coprod_{i_0 \cdots i_n}} yU_{i_0 \cdots i_n} \coprod \left(  \coprod_{k= 1}^n  \underset {i_{j_1} = i_{j_1+1}, \cdots, i_{j_k} = i_{j_k+1}} { \coprod_{i_0 \cdots i_n}} yU_{i_0 \cdots i_n} \right) 
 \end{eqnarray*}
and apply $\F$ to get 
 \begin{eqnarray*}
\prod_{\substack{i_0 \cdots i_n \\ i_j \neq i_{j+1} }}   \F  \left(  U_{i_0 \cdots i_n}  \right) \times  \prod_{k= 1}^n  \left(  \prod_{\substack{i_0 \cdots i_n \\ i_{j_1} = i_{j_1+1}, \cdots, i_{j_k} = i_{j_k+1} }}   \F \left(U_{i_0 \cdots i_n} \right)\right)
\end{eqnarray*}

First, note that the right side of this cartesian product is the matching object at $n$, $M_n \F ( \CechNerve U)$.  This is seen directly by showing that this product is the terminal object in the category of cones under $\F( \CechNerve U)$ restricted to the matching category $\partial ([n] \downarrow \overset{\leftarrow}{\Delta} )$ (see \cite[definition 15.2.3.2]{Hirs}).  The product 
\begin{eqnarray*}
\xymatrix{
  & &\overset{n}{ \underset{k=1} {\prod }}  \left( \prod\limits_{ \substack{i_0 \cdots i_n \\ i_{j_1} = i_{j_1+1}, \cdots, i_{j_k} = i_{j_k+1} }} \F \left( U_{i_0 \cdots i_n} \right)\right) \ar[rrd] \ar[rd] \ar[ld] \ar[lld]  & &  \\
\F (\CechNerve U_{n-1})  \ar[r] &  \F (\CechNerve U_{n-2}) \ar[r]& \cdots \ar[r] & \F (\CechNerve U_1) \ar[r]&  \F(\CechNerve U_0)& 
}
\end{eqnarray*}
is a cone under under $\F(\CechNerve U)$, where $\F (\CechNerve U_{n-j}) =  \prod\limits_{i_0 \cdots i_{n-j}} \F \left(U_{i_0 \cdots i_{n-j}}\right) $ and the vertical maps are projections.

Now, suppose we have a cone under $\F (\CechNerve U)$.  
\begin{eqnarray}
\xymatrix{
  & & Y \ar^{f_n}[rrd] \ar^{f_{n-1}}[rd] \ar_{f_{2}}[ld]  \ar_{f_{1}} [lld]  & &  \\
\F (\CechNerve U_{n-1})  \ar[r] &  \F (\CechNerve_{n-2}) \ar[r]& \cdots \ar[r] & \F (\CechNerve U_1) \ar[r]&  \F(\CechNerve U_0)& 
}
\end{eqnarray}
Then to define the map $Y$ into the product, send $y$ to $(f_1 (y), f_2(y), \cdots f_n(y))$.  

Finally, we see that the matching map $\F(\CechNerve U_n) \times M_n(\CechNerve U) \rightarrow M_n(\CechNerve U)$ is the projection onto the second factor.  Since $\F(\CechNerve U_n)$ is a Kan complex, the projection is a fibration. 
\end{proof}
Applying lemma \ref{LEM: Tot Kan} to lemma \ref{LEM: F CN Reedy} proves proposition \ref{PROP: tot cech kan}.

\end{document}